\documentclass[11pt]{amsart}

\usepackage{amssymb}
\usepackage{overpic}
\usepackage{enumitem}   
\usepackage{graphicx} 
\usepackage{multicol}
\usepackage{xcolor}
\usepackage{amsrefs}
\usepackage[a4paper,margin=2cm]{geometry}
\usepackage[hidelinks]{hyperref}

\graphicspath{{Figures/}}

\newtheorem{theorem}{Theorem}
\newtheorem{question}{Question}

\newtheorem{proposition}{Proposition} 
 
\newtheorem{remark}{Remark} 
\newtheorem{lemma}{Lemma} 
\newtheorem{definition}{Definition} 

\begin{document}
	
\title[On structural stability of Evolutionary Stables Strategies]{On structural stability of Evolutionary Stables Strategies}

\author[Jefferson Bastos, Claudio Buzzi and Paulo Santana]
{Jefferson Bastos$^1$, Claudio Buzzi$^1$ and Paulo Santana$^1$}

\address{$^1$ IBILCE--UNESP, CEP 15054--000, S. J. Rio Preto, S\~ao Paulo, Brazil}
\email{jefferson.bastos@unesp.br; claudio.buzzi@unesp.br; paulo.santana@unesp.br}



\begin{abstract}
	The introduction of concepts of Game Theory and Ordinary Differential Equations into Biology gave birth to the field of Evolutionary Stable Strategies, with applications in Biology, Genetics, Politics, Economics and others. In special, the model composed by two players having two pure strategies each results in a planar polynomial vector field with an invariant octothorpe. In this paper we provide sufficient and necessary conditions for structural stability in this model.
\end{abstract}

\maketitle

\section{Introduction}\label{Sec1}

In 1937 Andronov and Pontrjagin \cite{AndPon1937} coined the notion of a \emph{robust} vector field. Roughly speaking, a vector field $X$ is robust if small perturbations does not change the topological character of its orbits. Their work dealt with planar analytical vector fields defined on the disk $B^2=\{x\in\mathbb{R}^2\colon ||x||\leqslant1\}$ and transversal to its boundary $\partial B^2$. Since Andronov and Pontrjagin omitted some of their proofs, it was up to DeBaggis \cite{Bag1952}, 1952, to provide the omitted proofs under the less restrictive hypothesis of $X$ being of class $C^1$. On January 1959 M. M. Peixoto \cite{Pei1959} provided an equivalent definition of robustness, calling it \emph{structural stability}, and then extend some of the previous results to $B^n=\{x\in\mathbb{R}^n\colon ||x||\leqslant1\}$. On June of 1959 M. M. Peixoto and his wife, M. C. Peixoto \cite{PeiPei1959}, extended the notion of structural stability to $C^1$-vector fields $X$ defined on any two-dimensional manifold $M\subset\mathbb{R}^2$, with boundary and corners, and allowing contact between $X$ and $\partial M$. In 1962 M. M. Peixoto \cite{Pei1962} (from now on referred only as Peixoto) provided a complete characterization of the structurally stable vector fields $X$ of class $C^1$ defined on any two-dimensional closed manifold (i.e. compact and without boundary). Such characterization is known as \emph{Peixoto's Theorem}. In 1973 Peixoto \cite{Pei1973} also derived a relation between structural stability and Graph Theory. Given a closed two-dimensional manifold $M$, let $\mathfrak{X}$ be the family of all $C^1$-vector fields defined over $M$. Let also $\Sigma\subset\mathfrak{X}$ be the set of all the structurally stable vector fields over $M$ and denote $\mathfrak{X}_1=\mathfrak{X}\backslash\Sigma$. Under the supervision of Peixoto, Sotomayor \cite{Soto1974}, 1974, provided in his thesis the complete characterization of all the structurally stable vector fields of $\mathfrak{X}_1$, giving rise to the notion of \emph{structural stability of first order} (or vector fields of \emph{codimension one}). In 1977 Teixeira \cite{Tei1977} extended the work of Sotomayor, allowing $M$ to be a manifold with boundary. From there on, the notion of structural stability flourished and many papers were published, characterizing the structural stability of $C^1$-vector fields on open surfaces \cite{Kot1982}, polynomial vector fields on compact \cites{Soto1985,Cai1979,San1977,Vel} and non-compact \cites{Sha1987,DumSha1990} two dimensional manifolds, $C^1$-vector fields on compact $n$-dimensional manifolds \cite{Mar1961}, gradient flows \cite{Sha1990} and polynomial foliations \cite{JarLliSha2005}. For a historical paper about Peixoto, see \cite{SotGarMel2020}. 

Given $d\geqslant 1$, in this paper we characterize the structurally stable planar polynomial vector fields $X=(P,Q)$ given by,
\begin{equation}\label{29}
	P(x,y)=x(x-1)f(x,y), \quad Q(x,y)=y(y-1)g(x,y),
\end{equation}
with $f$ and $g$ polynomials of degree at most $d$. The motivation for such classification follows from Game Theory. More precisely, in 1973 Smith and Price \cite{SmiPri1973} introduced concepts of Game Theory in to the Biology, coining the notion of \emph{evolutionary stable strategies} (ESS). Roughly speaking, given a game between two or more players (modeling for example a conflict between different species of animals), an ESS is a strategy such that if most of the players follows this strategy, then there is no ``disruptive'' strategy that would give higher advantages for the other players. In other words, if most of the players follows this strategy, then the best strategy for the other players is to also follow it. In 1978 Taylor and Jonker \cite{TayJon1978} introduced concepts of Ordinary Differential Equations in to the field of ESS in such a way that the game is now molded by a system of Ordinary Differential Equations. In particular, if the game has two players, then the models results in the planar polynomial vector field given by \eqref{29}. We now present briefly how the model is constructed. Let $\Gamma_1$ and $\Gamma_2$ be two players and $\{X_1,X_2\}$, $\{Y_1,Y_2\}$ be their \emph{pure strategies}. We denote by $a_{ij}^*\in\mathbb{R}$ the \emph{payoff} of the pure strategy $X_i$ against the pure strategy $Y_j$, and by $b_{ij}^*\in\mathbb{R}$ the payoff of $Y_i$ against $X_j$. Given a probabilistic vector of dimension two
	\[x=(x_1,x_2)\in S^2:=\{(x_1,x_2)\in\mathbb{R}^2\colon x_1\geqslant0,\; x_2\geqslant0,\; x_1+x_2=1\},\]
we associate to it the \emph{mix strategy} $x_1X_1+x_2X_2$. Similarly, given $y\in S^2$, we associate the mix strategy $y_1Y_1+y_2Y_2$. Let
	\[A^*=\left(\begin{array}{cc} a_{11}^* & a_{12}^* \\ a_{21}^* & a_{22}^* \end{array}\right), \quad B^*=\left(\begin{array}{cc} b_{11}^* & b_{12}^* \\ b_{21}^* & b_{22}^* \end{array}\right),\]
be the \emph{payoff matrices}. Given $y\in S^2$ we define
\begin{equation}\label{30}
	\left<e_i,A^*y\right>=a_{i1}^*y_1+a_{i2}^*y_2,
\end{equation}
as the payoff of $X_i$ against $y_1Y_1+ y_2Y_2$ (here, $\left<\cdot,\cdot\right>$ denotes the standard inner product of $\mathbb{R}^2$). Similarly, given $x\in S^2$, the payoff of the pure strategy $Y_i$ against the mix strategy $x_1X_1+x_2X_2$ is given by,
\[\left<e_i,B^*x\right>=b_{i1}^*x_1+b_{i2}^*x_2.\]
Furthermore, the \emph{average payoff} of $x_1X_1+x_2X_2$ against $y_1Y_1+y_2Y_2$ is given by,
\begin{equation}\label{31}
	\left<x,A^*y\right>=a_{11}^*x_1y_1+a_{12}^*x_1y_2+a_{21}^*x_2y_1+a_{22}^*x_2y_2,
\end{equation}
while the average payoff of $y_1Y_1+y_2Y_2$ against $x_1X_1+x_2X_2$ is given by,
\[\left<y,B^*x\right>=b_{11}^*x_1y_1+b_{12}^*x_1y_2+b_{21}^*x_2y_1+b_{22}^*x_2y_2.\]
The dynamic between players $\Gamma_1$ and $\Gamma_2$ is defined by the ordinary system of differential equations,
\begin{equation}\label{32}
	\begin{array}{ll }
		\dot x_1=x_1\bigl(\left<e_1,A^*y\right>-\left<x,A^*y\right>\bigr), & \quad\dot y_1=y_1\bigl(\left<e_1,B^*x\right>-\left<y,B^*x\right>\bigr), \vspace{0.2cm} \\
		\dot x_2=x_2\bigl(\left<e_2,A^*y\right>-\left<x,A^*y\right>\bigr), & \quad\dot y_2=y_2\bigl(\left<e_2,B^*x\right>-\left<y,B^*x\right>\bigr). 
	\end{array}
\end{equation} 
In other words, the rate of change of the population $x_i$ depends on the difference between the payoffs \eqref{30} and \eqref{31} (the bigger the difference, the bigger the superiority of strategy $X_i$), and on the size of the population itself, as a percentage of the total population $x_1+x_2$. Since $x_1+x_2=y_1+y_2=1$, it follows that to describe the dynamic between players $\Gamma_1$ and $\Gamma_2$, it is necessary only two variables, namely $x=x_1$ and $y=y_1$. Therefore, if we replace $x_2=1-x$ and $y_2=1-y$, then we can simplify system \eqref{32} and obtain the planar ordinary system of differential equations,
\begin{equation}\label{33}
	\begin{array}{l}
		\dot x=x(x-1)\bigl(a_{22}^*-a_{12}^*+(a_{12}^*+a_{21}^*-a_{11}^*-a_{22}^*)y\bigr), \vspace{0.2cm} \\
		\dot y=y(y-1)\bigl(b_{22}^*-b_{12}^*+(b_{12}^*+b_{21}^*-b_{11}^*-b_{22}^*)x\bigr). 
	\end{array}
\end{equation}
Applications of such model can be found in a myriad of research areas, such as Parental Investing, Biology, Genetics, Evolution, Ecology, Politics and Economics. See \cites{Hines,SchSig,Bomze,AccMarOvi,CanBer,XiaoYu} and the references therein. The chasing for more realistic models demand the payoff coefficients to depend on the weight given to strategies $X_i$ and $Y_j$, rather than being constant, i.e. $a_{ij}^*=a_{ij}^*(x,y)$ and $b_{ij}^*=b_{ij}^*(x,y)$. Hence, if we assume that each payoff $a_{ij}^*$, $b_{ij}^*$ is a polynomial of degree at most $n$, then system \eqref{33} becomes
	\[\dot x=x(x-1)f(x,y), \quad \dot y=y(y-1)g(x,y),\]
with $f$ and $g$ polynomials of degree at most $n+1$. 
 
\section{Statement and discussion of the main results}\label{Sec2}

As stated in Section~\ref{Sec1}, the structural stability of the $C^1$-vector fields over a closed two-dimensional manifold $M$ was completely described by Peixoto and it is known as \emph{Peixoto's Theorem} \cite{Pei1962}. Moreover, if $M$ can be embedded in $\mathbb{R}^2$, then $M$ is allowed to have borders and corners \cite{PeiPei1959}. In both cases, it is also proved that the family of the structurally stable vector fields is dense in the set of all the $C^1$-vector fields. Similar results also hold for the case in which $M\subset\mathbb{R}^2$ is an open surface (in particular, for the case $M=\mathbb{R}^2$). See \cite{Kot1982}. However, the characterization of structural stability for \emph{polynomial} vector fields is harder to tackle. This is due because to prove that a given object \emph{is not} structurally stable (e.g. non-hyperbolic limit cycle) one need to \emph{construct} some kind of perturbation that affect such object, changing the topological character of the orbits around it (e.g. to split the non-hyperbolic limit cycle in two or more limit cycles). Hence, the more restrict the space is, the more difficult it can get to obtain suitable perturbations. For example, when dealing with $C^1$-vector fields one can use \emph{bump-functions} to construct suitable perturbations that ``break'' a given object. However, when dealing with polynomial vector fields, this tool cannot be used. Therefore, when dealing with a more restrict space of vector fields, one may obtain a more broader class of structurally stable objects. Concerns about this go back to Andronov et al \cite[$\mathsection6.3$]{And1971}. Still, there are approaches to the structural stability of polynomial vector fields. For this, we refer to the works of Sotomayor \cite{Soto1985} and Shafer \cite{Sha1987}. In the former, the author defines the structural stability of a planar polynomial vector field $X$ as the structural stability of its Poincar\'e compactification $p(X)$ (see Section~\ref{sub1}) and endows the space of vector fields with the coefficients topology. In the later, the author approaches $X$ as a vector field defined on the open surface $M=\mathbb{R}^2$, where the space of vector fields can be endowed either with the \emph{Whitney's Topology} or with the coefficients topology. These approaches results in different sets of necessary and sufficient conditions for structural stability. For a more deeper discussion about these differences, we refer to \cite{DumSha1990}. In this paper we use an approach that resembles the approach of Sotomayor. That is, we define the structural stability of $X$ as the structural stability of its Poincar\'e compactification $p(X)$ and we endow the space of vector fields with the coefficients topology. For this paper, we think that this approach is better because the Poincar\'e compactification is a powerful tool to understand the dynamics of $X$ near the infinity and because the coefficients topology is a more practical topology to deal with concrete perturbations in the model. We now state our results concretely. Given $d\in\mathbb{N}$, let $\mathfrak{X}_d$ be the set of planar polynomial vector fields $X=(P,Q)$ given by
\begin{equation}\label{0}
	P(x,y)=x(x-1)f(x,y), \quad Q(x,y)=y(y-1)g(x,y),
\end{equation}
where $f$, $g\colon\mathbb{R}^2\to\mathbb{R}$ are real polynomials of degree at most $d$ given by,
	\[f(x,y)=\sum_{i+j=0}^{d}a_{ij}x^iy^j, \quad g(x,y)=\sum_{i+j=0}^{d}b_{ij}x^i y^j.\]
Given $X\in\mathfrak{X}_d$, let $\varphi\colon\mathfrak{X}_d\to\mathbb{R}^{(d+1)(d+2)}$ be the linear isomorphism that identify $X$ with the coefficients of $f$ and $g$. The \emph{coefficients topology} of $\mathfrak{X}_d$ is the topology given by the metric
\begin{equation}\label{36}
	\rho(X,Y)=||\varphi(X)-\varphi(Y)||,
\end{equation}
where $||\cdot||$ denotes the standard norm of $\mathbb{R}^{(d+1)(d+2)}$. Given $X\in\mathfrak{X}_d$, let $p(X)$ be its compacti\-fication in the Poincar\'e Sphere (see Section~\ref{sub1}). Let also $\mathbb{S}^1$ be the equator of $\mathbb{S}^2$. Given $X$, $Y\in\mathfrak{X}_d$, we say that $X$ and $Y$ are \emph{topologically equivalent} if there exists an homeomorphism $h\colon\mathbb{S}^2\to\mathbb{S}^2$, preserving $\mathbb{S}^1$, which sends orbits of $p(X)$ to orbits of $p(Y)$, either preserving or reversing the direction of all orbits. 

\begin{definition}\label{Def0}
	Given $d\geqslant1$ and $X\in\mathfrak{X}_d$, we say that $X$ is \emph{structurally stable} (or \emph{structurally stable in the sense of Peixoto}) if there exists an neighborhood $N\subset\mathfrak{X}_d$ of $X$ such that $X$ is topologically equivalent to any $Y\in N$.
\end{definition}

Let $\Sigma_d\subset\mathfrak{X}_d$ be the family of the structurally stable vector fields of $\mathfrak{X}_d$. Given $X\in\mathfrak{X}_d$, let $\gamma_1,\gamma_2,\gamma_3,\gamma_4\subset\mathbb{S}^2$ be the respective invariant curves of $p(X)$, given by the compactification of the four invariant straight lines $x=0$, $x=1$, $y=0$, $y=1$. Let also
\begin{equation}\label{37}
	\Lambda=\gamma_1\cup\gamma_2\cup\gamma_3\cup\gamma_4\cup\mathbb{S}^1, 
\end{equation}
where $\mathbb{S}^1$ denotes the infinity of the Poincar\'e sphere. See Figure~\ref{Fig15}. 
\begin{figure}[ht]
	\begin{center}
		\begin{overpic}[height=4cm]{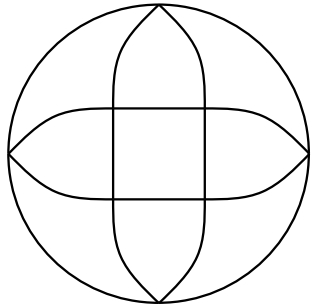} 
		\end{overpic}
	\end{center}
	\caption{An illustration of the set $\Lambda$.}\label{Fig15}
\end{figure}
In our first main result, we provide necessary conditions for structural stability in $\mathfrak{X}_d$.

\begin{theorem}\label{Main1}
	If $X\in\Sigma_d$, then the following statements hold. 
	\begin{enumerate}[label=(\alph*)]
		\item $p(X)$ have a finite number of singularities, each of them being topologically equivalent to a node or a saddle;
		\item $p(X)$ have a finite number of closed orbits, none o them being semi-stable;
		\item If $\gamma$ is a connection between saddles, then $\gamma\subset\Lambda$;
		\item The $\alpha$ and $\omega$-limits of every orbit of $p(X)$ is either a singularity, a closed orbit or a polycycle contained in $\Lambda$.
	\end{enumerate}
\end{theorem}

Let $\Omega\subset\mathbb{R}^2$ be a polycycle (also known as a graph) composed by $n$ hyperbolic saddles $p_1,\dots p_n$, whose eigenvalues are given by $\mu_i<0<\nu_i$, $i\in\{1,\dots,n\}$. The \emph{hyperbolicity ratio} of $p_i$ is given by $r_i=\frac{|\mu_i|}{\nu_i}$. Cherkas \cite{Cherkas} proved that if $r(\Omega):=r_1\dots r_n>1$, (resp. $r(\Omega)<1$), then $\Omega$ is stable (resp. unstable). Therefore, we say that $\Omega$ is \emph{generic} if $r(\Omega)\neq1$. In special, observe that a generic polycycle cannot be accumulated by periodic orbits. 

\begin{definition}\label{Def00}
Given $d\geqslant1$ and $X\in\mathfrak{X}_d$, we say that $X$ is \emph{strong structurally stable} (or \emph{structurally stable in the sense of Andronov-Pontrjagin}) if for every $\varepsilon>0$ there is a neighborhood $N\subset\mathfrak{X}_d$ of $X$ and a map $h\colon N\to\text{Hom}(\mathbb{S}^2,\mathbb{S}^2)$ such that $h_Y\colon\mathbb{S}^2\to\mathbb{S}^2$ is an homeomorphism, preserving $\mathbb{S}^1$, which sends orbits of $p(X)$ to orbits of $p(Y)$, either preserving or reversing the direction of all orbits, and such that the following statements holds.
\begin{enumerate}[label=(\roman*)]
	\item $h_X=Id_{\mathbb{S}^2}$;
	\item For every $Y\in N$ and $x\in\mathbb{S}^2$, $||h_Y(x)-x||<\varepsilon$.
\end{enumerate}
\end{definition}

Let $\Sigma_d^s\subset\mathfrak{X}_d$ be the family of the strong structurally stable vector fields of $\mathfrak{X}_d$ and observe that $\Sigma_d^s\subset\Sigma_d$. In our second main result, we provide necessary conditions for strong structural stability.

\begin{theorem}\label{Main2}
	If $X\in\Sigma_d^s$, then $X$ satisfies statements $(b)$ and $(c)$ of Theorem~\ref{Main1} and the following restricted versions of statements $(a)$ and $(d)$. 
	\begin{enumerate}
		\item[$(a')$] $p(X)$ have a finite number of singularities, all hyperbolic;
		\item[$(d')$] The $\alpha$ and $\omega$-limits of every orbit of $p(X)$ is either a singularity, a closed orbit or a generic polycycle contained in $\Lambda$.
	\end{enumerate}
\end{theorem}

\begin{definition}
Let $\mathcal{P}_d\subset\mathfrak{X}_d$ be the family of vector fields satisfying statement $(c)$ of Theorem~\ref{Main1}, statements $(a')$ and $(d')$ of Theorem~\ref{Main2} and the following restricted version of statement $(b)$.
\begin{enumerate}
	\item[$(b')$] $p(X)$ have a finite number of closed orbits, all hyperbolic;
\end{enumerate}
\end{definition}

In our third main result, we provide sufficient conditions for strong and non-strong structural stability and we prove that such families are dense in $\mathfrak{X}_d$.

\begin{theorem}\label{Main3}
	The following statements hold.
	\begin{enumerate}[label=(\alph*)]
		\item $\mathcal{P}_d$ is open and dense in $\mathfrak{X}_d$;
		\item $\mathcal{P}_d\subset\Sigma_d^s$. 
	\end{enumerate}
	In special, $\Sigma_d^s$ and $\Sigma_d$ are open and dense in $\mathfrak{X}_d$.
\end{theorem}

In our fourth main result, we prove that almost every $X\in\mathfrak{X}_d$ (in the topological sense) does not have a global first integral in $\mathbb{R}^2$ (in special, $p(X)$ does not have a global first integral in $\mathbb{S}^2$).

\begin{theorem}\label{Main4}
	Let $\Upsilon_d\subset\mathfrak{X}_d$ be the set of vector fields which do not have a global first integral in $\mathbb{R}^2$. Then $\mathcal{P}_d\subset\Upsilon_d$. In special, $\Upsilon_d$ is dense in $\mathfrak{X}_d$. 
\end{theorem}

When working with the coefficients topology, there is an open question that keeps everyone from obtaining necessary \emph{and} sufficient conditions for structural stability in planar polynomial vector fields.

\begin{question}\label{Q1}
	Let $\mathcal{X}_n$ be the set of the planar polynomial vector fields of degree at most $n$, endowed with the coefficients topology. If $X\in\mathcal{X}_n$ has a non-hyperbolic limit cycle of odd degree, then is $X$ structurally unstable in $\mathcal{X}_n$?
\end{question}

In simple words, \emph{are non-hyperbolic limit cycles of odd degree structurally unstable in the set planar polynomial vector fields?} This question was explicitly raised both by Sotomayor \cite[Problem $1.1$]{Soto1985} and Shafer \cite[Question $3.4$]{Sha1987} and kept both of them from obtaining a complete characterization of structural stability. The structural instability of non-hyperbolic limit cycles of even degree follows by an application of the theory of rotated vector fields. More precisely, if $X=(P,Q)\in\mathcal{X}_n$ has an non-hyperbolic limit cycle of even degree, then let $Y=(R,S)\in\mathcal{X}_n$ be given by
	\[R(x,y)=P(x,y)-\lambda Q(x,y), \quad S(x,y)=Q(x,y)+\lambda P(x,y),\]
with $|\lambda|>0$ small enough. It follows from the theory of rotated vector fields (see \cite[Theorem $2$, p. $387$]{Perko2001} that for $|\lambda|>0$ small enough, $\gamma$ has either vanished or had split in at least two limit cycles. For more details, see \cite[Section $4.6$]{Perko2001}. In the case of smooth or analytical vector fields it is well known that non-hyperbolic limit cycles (in particular, those with odd degree) are structurally unstable. Roughly speaking, the proof work as follows. Let $X=(P,Q)$ be a planar vector field of class $C^r$, $r\geqslant 1$, and let $\gamma(t)$ be a limit cycle of $X$, with period $T>0$. Andronov et al \cite[p. $124$]{And1971}, proved that there are a neighborhood $G\subset\mathbb{R}^2$ of $\gamma$ and a function $F\colon G\to\mathbb{R}$ of class $C^{r+1}$, such that 
	\[F(\gamma(t))=0, \quad \frac{\partial F}{\partial x}(\gamma(t))^2+\frac{\partial F}{\partial y}(\gamma(t))^2>0,\]
for all $t\in[0,T]$. Peixoto \cite[Lemma~$6$]{Pei1959}, proved that if $X$ is analytical, then $F$ is also analytical. With such function, the canonical way of proving that a non-hyperbolic limit cycle is not structurally stable is by considering the perturbed vector field $Y=(R,S)$ given by,
	\[R(x,y)=P(x,y)+\lambda F(x,y)\frac{\partial F}{\partial x}(x,y), \quad S(x,y)=Q(x,y)+\lambda F(x,y)\frac{\partial F}{\partial y}(x,y),\]
with $|\lambda|>0$ small enough. By using the Poincar\'e-Bendixson Theory, one can prove that for $|\lambda|>0$ small enough, $\gamma$ had split in at least two limit cycles. For more details, see \cite[$\mathsection15.2$]{And1971} and \cite[Lemma~$6$]{Pei1959}. Observe that if $F$ is polynomial, then $\gamma$ is an algebraic limit cycle (see Section~\ref{sub6}). However, since not every limit cycle of a polynomial vector field is algebraic (see \cite{GasGiaTor2007}), it follows that even if $X$ is polynomial, $F$ may not be. Hence, the usual tools seems to not work in this case.

In this paper we were not able to solve Question~\ref{Q1}. As a result, it follows from Theorems~\ref{Main2} and \ref{Main3} that conditions $(a')$, $(b)$, $(c)$ and $(d')$ are necessary for strong structural stability and if we replace condition $(b)$ by $(b')$, then we have a set of sufficient conditions for strong structurally stability. Moreover, observe that the difference between conditions $(b)$ and $(b')$ is precisely given by the non-hyperbolic limit cycles of odd degree. More precisely, it follows from Theorems~\ref{Main2} and \ref{Main3} that $\Sigma_d^s\backslash\mathcal{P}_d$ is given precisely by the strong structurally stable vector fields with non-hyperbolic limit cycles of odd degree, if any. For the particular case $d=1$, it follows from a result of Kooij \cite{Kooij} that we can ensure $\Sigma_1^s=\mathcal{P}_1$ and thus we have a \emph{complete characterization} of the structurally stable vector fields for this case, i.e. we have a practical set of necessary \emph{and} sufficient conditions for structural stability.

\begin{remark}
	It follows from Theorem $2.5.4$ of \cite{Kooij} that if $X\in\mathfrak{X}_1$ has a limit cycle, then it is unique and hyperbolic. In particular, it follows that if $X\in\Sigma_1^s$ has a limit cycle, then it is hyperbolic. Thus, it follows from Theorem~\ref{Main2} that if $X\in\Sigma_1^s$, then $X\in\mathcal{P}_1$. Hence, for $d=1$ we have $\Sigma_1^s=\mathcal{P}_1$, i.e. we have a complete characterization of the strong structurally stable vector fields. For the phase portraits of $\mathcal{P}_1$, see \cite{BasBuzSan2022}.
\end{remark}

In our final main result, we prove that non-hyperbolic \emph{algebraic} limit cycles (see Section~\ref{sub6}) are not structurally stable.

\begin{theorem}\label{Main5}
	Non-hyperbolic algebraic limit cycles are not structurally stable.
\end{theorem}

We now highlight another main difference when studying the structural stability of polynomial vector fields, rather than its smooth or analytical counterpart. Our definition of \emph{strong} structural stability (recall Definition~\ref{Def00}) resembles the original definition of \emph{robust} vector fields, given by Andronov and Pontrjagin \cite{AndPon1937}. Their definition asks the equivalence homeomorphisms (i.e. the homeomorphism that sends orbits of one vector field to the orbits of the other one) to be in a $\varepsilon$-neighborhood of the identity map, in relation to the $C^0$-topology. One of the first main contributions of Peixoto \cite{Pei1959} was to prove that this condition, for smooth and analytical vector fields, is redundant. That is, it is enough to asks just for an homeomorphisms that sends orbits to orbits, without caring if it is near the identity map or not. However, to prove this redundancy, Peixoto relies heavily on bump-functions and on suitable analytical functions that are zero wherever he wants to. For example, if one have a finite set of limit cycles $\gamma_1,\dots,\gamma_k$, then one can build a suitable analytical map $\varphi$ such that $\varphi$ has suitable properties near $\gamma_1$ and at the same time satisfies $\varphi(\gamma_i)=0$ for $i\geqslant 2$. This allowed Peixoto to work on a specific limit cycle without making disruptive changes on the other ones. However, when dealing with polynomial vector fields of some maximum degree $n$ it is impossible to construct a polynomial perturbation satisfying so many conditions. Therefore, in the polynomial case it is common to work with a definition that resembles the definition of Andronov and Pontrjagin, rather than the definition of Peixoto. For a deeper discussion on this topic, we refer to \cite{DumSha1990}. The extra conditions $(i)$ and $(ii)$ on Definition~\ref{Def00} are fundamental when looking for \emph{necessary} conditions for structurally stability. For example, if we assume that $X\in\mathfrak{X}_d$ is \emph{strong} structurally stable, then we can work on a specific limit cycle $\gamma$ without worrying with disruptive perturbations on the other limit cycles. As a result, when working with Definition~\ref{Def00} rather than Definition~\ref{Def0}, we are able to prove that all singularities must be hyperbolic and that all polycycles must be generic. These are precisely the difference between Theorems~\ref{Main1} and \ref{Main2}, i.e. between $\Sigma_d$ and $\Sigma_d^s$.

The paper is organized as follows. In Section~\ref{Sec3} we state some preliminary results, which will be used through the paper. In Section~\ref{Sec4} and \ref{Sec5} we study the necessary and sufficient conditions for strong structural stability. The main Theorems are proved in Section~\ref{Sec6}.

\section{Preliminary results}\label{Sec3}

\subsection{The Poincar\'e Compactification}\label{sub1}

Let $X=(P,Q)$ be a planar polynomial vector field of degree $n\in\mathbb{N}$. The \emph{Poincar\'e compactified vector field} $p(X)$ is an analytic vector field on $\mathbb{S}^2$ constructed as follows (for more details see \cite{Vel} or Chapter $5$ of \cite{DumLliArt2006}). First we identify $\mathbb{R}^2$ with the plane $(x_1,x_2,1)$ in $\mathbb{R}^3$ and define the \emph{Poincar\'e sphere} as $\mathbb{S}^2=\{y=(y_1,y_2,y_3)\in\mathbb{R}^3:y_1^2+y_2^2+y_3^2=1\}$. We define the \emph{northern hemisphere}, the \emph{southern hemisphere} and the \emph{equator} respectively by $H_+=\{y\in\mathbb{S}^2:y_3>0\}$, $H_-=\{y\in\mathbb{S}^2:y_3<0\}$ and $\mathbb{S}^1=\{y\in\mathbb{S}^2:y_3=0\}$. Consider now the projections $f_\pm:\mathbb{R}^2\rightarrow H_\pm$, given by $f_\pm(x_1,x_2)=\pm \Delta(x_1,x_2)(x_1,x_2,1)$, where $\Delta(x_1,x_2)=(x_1^2+x_2^2+1)^{-\frac{1}{2}}$. These two maps define two copies of $X$, one copy $X^+$ in $H_+$ and one copy $X^-$ in $H_-$. Consider the vector field $X'=X^+\cup X^-$ defined in $\mathbb{S}^2\backslash\mathbb{S}^1$. Note that the \emph{infinity} of $\mathbb{R}^2$ is identified with the equator $\mathbb{S}^1$. The Poincar\'e compactified vector field $p(X)$ is the analytic extension of $X'$ from $\mathbb{S}^2\backslash\mathbb{S}^1$ to $\mathbb{S}^2$ given by $y_3^{n-1}X'$. The \emph{Poincar\'e disk} $\mathbb{D}$ is the projection of the closed northern hemisphere to $y_3=0$ under $(y_1,y_2,y_3)\mapsto(y_1,y_2)$ (the vector field given by this projection will also be denoted by $p(X)$). Note that to know the behavior $p(X)$ near $\mathbb{S}^1$ is the same than to know the behavior of $X$ near the infinity. We define the local charts of $\mathbb{S}^2$ by $U_i=\{y\in\mathbb{S}^2:y_i>0\}$ and $V_i=\{y\in\mathbb{S}^2:y_i<0\}$ for $i\in\{1,2,3\}$. In these charts we define $\phi_i:U_i\rightarrow\mathbb{R}^2$ and $\psi_i:V_i\rightarrow\mathbb{R}^2$ by $\phi_i(y_1,y_2,y_3)=-\psi_i(y_1,y_2,y_3)=(y_m/y_i,y_n/y_i)$, where $m\neq i$, $n\neq i$ and $m<n$. Denoting by $(u,v)$ the image of $\phi_i$ and $\psi_i$ in every chart (therefore $(u,v)$ will play different roles in each chart), one can see the following expression for $p(X)$ at chart $U_1$ is given by,
	\[\dot u = v^n \left[Q\left(\frac{1}{v},\frac{u}{v}\right)-uP\left(\frac{1}{v},\frac{u}{v}\right)\right], \quad \dot v = -v^{n+1}P\left(\frac{1}{v},\frac{u}{v}\right),\]
and at chart $U_2$ it is given by,
	\[\dot u = v^n \left[P\left(\frac{u}{v},\frac{1}{v}\right)-uQ\left(\frac{u}{v},\frac{1}{v}\right)\right], \quad \dot v = -v^{n+1}Q\left(\frac{u}{v},\frac{1}{v}\right).\]
The expressions of $p(X)$ in $V_1$ and $V_2$ is the same as that for $U_1$ and $U_2$, except by a multiplicative factor of $(-1)^{n-1}$. In these coordinates for $i\in\{1,2\}$, $v=0$ represents the points of $\mathbb{S}^1$ and thus the infinity of $\mathbb{R}^2$. Note that $\mathbb{S}^1$ is invariant under the flow of $p(X)$. 

\subsection{Singularities}\label{sub2}

Let $X=(P,Q)$ be a planar polynomial vector field. We say $p\in\mathbb{R}^{2}$ is a \emph{singularity} of $X$ if $P(p)=Q(p)=0$. The Jacobian matrix of $X$ at $p$ is given by,
\begin{equation}\label{4}
	DX(p)=\left(\begin{array}{cc}
		\displaystyle \frac{\partial P}{\partial x}(p) & \displaystyle \frac{\partial P}{\partial y}(p) \vspace{0.2cm} \\ 
		\displaystyle \frac{\partial Q}{\partial x}(p) & \displaystyle \frac{\partial Q}{\partial y}(p)
	\end{array}\right).
\end{equation}
Let $\det DX(p)$ and $T(DX(p))$ be the \emph{determinant} and the \emph{trace} of $DX(p)$. Let $\lambda_1$, $\lambda_2\in\mathbb{C}$ be the eigenvalues of $DX(p)$.
\begin{enumerate}[label=(\alph*)]
	\item If $\det DX(p)\neq0$, we say that $p$ is \emph{simple};
	\item If $\lambda_1$ and $\lambda_2$ have real part different from zero, we say that $p$ is \emph{hyperbolic}. Observe that if $p$ is hyperbolic, then $p$ is simple. We distinguish the following cases.
	\begin{enumerate}[label=(\roman*)]
		\item If $\det DX(p)<0$, then $p$ is saddle.
		\item If $\det DX(p)>0$ and $T(DX(p))>0$, then $p$ is an unstable focus/node.
		\item If $\det DX(p)>0$ and $T(DX(p))<0$, then $p$ is a stable focus/node.
	\end{enumerate}
	\item If $\det DX(p)>0$ and $T(DX(p))=0$, we say that $p$ is \emph{degenerated and monodromic}. In special, $p$ is topologically equivalent to a focus or to a center.
\end{enumerate}
For more details, see Chapter $2$ of \cite{DumLliArt2006} or \cite{Perko2001}.

\subsection{Separatrices, sources and sinks}\label{sub3}

Let $X$ be a planar polynomial vector field and let $p(X)$ be its Poincar\'e compactification. An orbit $\gamma\subset\mathbb{S}^2$ is a sepatrix of $p(X)$ if one of the following statements hold.
\begin{enumerate}[label=(\alph*)]
	\item $\gamma$ is a singularity of $p(X)$;
	\item $\gamma$ is a limit cycle of $p(X)$;
	\item $\gamma\subset\mathbb{S}^1$, where $\mathbb{S}^1$ denotes the equator of $\mathbb{S}^2$;
	\item $\gamma$ is in the boundary of a hyperbolic sector of $p(X)$.
\end{enumerate}
For a detailed definition of \emph{hyperbolic sector} we refer to \cite[p. $18$]{DumLliArt2006}. The set $S\subset\mathbb{S}^2$ of all the separatrices of $p(X)$ is closed and the connected components of $\mathbb{S}^2\backslash S$ are the \emph{canonical regions}. When restricted to a canonical region, the flow of $p(X)$ is topologically equivalent to one of the following systems.
\begin{enumerate}[label=(\roman*)]
	\item $\mathbb{R}^2$ with $\dot y=0$;
	\item $\mathbb{R}^2\backslash\{0\}$ with (in the polar coordinates) $\dot r = 0$ and $\dot\theta=1$;
	\item $\mathbb{R}^2\backslash\{0\}$ with $\dot r=1$ and $\dot\theta=0$.
\end{enumerate}
Such regions are called \emph{strip}, \emph{annular} and \emph{radial} regions. Given an orbit $\gamma\subset\mathbb{S}^2$, we define the $\alpha$ and $\omega$-limits of $\gamma$ by
	\[\alpha(\gamma)=\{q\in\mathbb{S}^2\colon\exists t_n\to-\infty; \gamma(t_n)\to q\}, \quad \omega(\gamma)=\{q\in\mathbb{S}^2\colon\exists t_n\to+\infty; \gamma(t_n)\to q\}.\]
	
\begin{proposition}[Property~$1$, \cite{PeiPei1959}]\label{Prop0}
	Let $D$ be a canonical region of $p(X)$. If $\gamma$, $\xi\subset D$ are orbits of $p(X)$, then $\alpha(\gamma)=\alpha(\xi)$ and $\omega(\gamma)=\omega(\xi)$. 
\end{proposition}

Given a canonical region $D$ of $p(X)$ and an orbit $\gamma\subset D$, it follows from Proposition~\ref{Prop0} that we can define $\alpha(D)=\alpha(\gamma)$ and $\omega(D)=\omega(\gamma)$. In special, $\alpha(D)$ and $\omega(D)$ are the \emph{source} and \emph{sink} of $D$. For more details about separatrices, see \cites{Mar1954,Neu1975,BueLop2018}. 

\subsection{Index of singularities and closed curves}\label{sub4}

Let $X$ be a planar polynomial vector field and $S\subset\mathbb{R}^2$ a simple closed curve such that $X$ has no singularities in $S$. The \emph{index of $X$ relative to $S$}, denoted by $i(X,S)$, is the number of revolutions, taking in to account their orientation, that $X$ makes when moving along $S$.

\begin{proposition}[Proposition~$6.15$ of \cite{DumLliArt2006}]\label{P1}
	Let $X$ and $Y$ be planar polynomial vector fields. Let $S\subset\mathbb{R}^2$ be a simple closed curve such that $X$ and $Y$ have no singularities on $S$. If $X(s)$ and $Y(s)$ never have opposite direction at any $s\in S$, then $i(X,S)=i(Y,S)$.	
\end{proposition}

\begin{proposition}[Proposition~$6.16$ of \cite{DumLliArt2006}]\label{P2}
	Let $X$ be planar polynomial vector field. Let $S\subset\mathbb{R}^2$ be a simple closed curve such that $X$ has no singularities on $S$. If $i(X,S)\neq0$, then $X$ has at least one singularity in the bounded region limited by $S$.	
\end{proposition}

\begin{remark}\label{Remark0}
	Observe that if $\gamma$ is a periodic orbit of $X$, then $i(X,\gamma)=1$. Hence, it follows from Proposition~\ref{P2} that $X$ has at least one singularity in the bounded region limited by $\gamma$.
\end{remark}

Let $p$ be an isolated singularity of $X$. Let $e$ and $h$ be the number of elliptical and hyperbolic sectors of $p$, respectively. The \emph{Poincar\'e index} of $p$ is given by
	\[i(X,p)=\frac{e-h}{2}+1.\]
It is known that $i(X,p)\in\mathbb{Z}$. See Proposition~$6.32$ of \cite{DumLliArt2006}. For a detailed definition of \emph{elliptical sector}, we refer to \cite[p. $18$]{DumLliArt2006}.

\begin{proposition}[Proposition~$6.26$ of \cite{DumLliArt2006}]\label{P3}
	Let $X$ be a planar polynomial vector field and $S\subset\mathbb{R}^2$ be a simple closed curve such that $X$ has no singularities in $S$. Suppose that $X$ has at most a finite number of singularities $p_1,\dots,p_k$ in the bounded region limited by $S$. Then,
		\[i(X,S)=\sum_{j=1}^{k}i(X,p_j).\]
\end{proposition}

For more details on the Index Theory, see Chapter~$6$ of \cite{DumLliArt2006}. 

\subsection{Whitney's Stratification}\label{sub5}

Let $Z\subset\mathbb{R}^n$ be a closed set. An analytical \emph{Stratification} of $Z$ is a filtration of $Z$ by closed sets
	\[Z=Z_d\supset Z_{d-1}\supset\dots\supset Z_1\supset Z_0,\]
such that $Z_i\backslash Z_{i-1}$ is either empty or an analytical manifold of dimension $i$. Each connected component of $Z_i\backslash Z_{i-1}$ is called a \emph{stratum} of dimension $i$. Thus, $Z$ is the disjoint union of the strata. An analytical \emph{Whitney Stratification} is, among other things, a locally finite analytical stratification. That is, given $p\in Z$ there is a neighborhood $U\subset\mathbb{R}^n$ of $p$ such that at most a finite number of strata intersects $U$. A set $Z\subset\mathbb{R}^n$ is \emph{analytic} if there are a finite number of analytical functions $f_1,\dots,f_k\colon\mathbb{R}^n\to\mathbb{R}$, such that
	\[Z=\{x\in\mathbb{R}^n\colon f_1(x)=\dots=f_k(x)=0\}.\]

\begin{theorem}[Theorem~$1.2.10$ of \cite{Trot}]\label{T5}
	Every analytic subset of $\mathbb{R}^n$ admits an analytical Whitney Stratification.
\end{theorem}

Let $f\colon\mathbb{R}^n\to\mathbb{R}$ be an analytical non-constant function. If $0\in\mathbb{R}$ is a regular value of $f$, then it follows from the Implicit Function Theorem that $f^{-1}(0)\subset\mathbb{R}^n$ is a analytical manifold of codimension $1$. Hence, Theorem~\ref{T5} is stating that if $0$ is not a regular value of $f$, then $f^{-1}(0)$ is yet endowed with some regularity. More precisely, in this case it follows from Theorem~\ref{T5} that
	\[f^{-1}(0)=B_1\cup B_2\cup\dots\cup B_n,\]
where the union is disjoint and $B_i$ is an analytical manifold of codimension $i$. Moreover, if we are interested in a particular point $p\in f^{-1}(0)$, then it follows from the locally finite property that we can restrict the domain of $f$ to a neighborhood of $p$ and thus assume that each $B_i$ has at most a finite number of connected components. In special, we conclude that $f^{-1}(0)$ has zero Lebesgue measure on that neighborhood.

\subsection{Non-hyperbolic limit cycles and algebraic limit cycles}\label{sub6}

Let $X=(P,Q)\in\mathfrak{X}_d$ and let $\gamma$ be a limit cycle of $X$ with period $T>0$ and parametrization $\gamma(t)$. Let also,
	\[r(\gamma)=\int_{0}^{T}\frac{\partial P}{\partial x}(\gamma(t))+\frac{\partial Q}{\partial y}(\gamma(t))\;dt.\]
It is well known (see Theorem~$2$, p. $216$ of \cite{Perko2001}) that $\gamma$ is hyperbolic if, and only if, $r(\gamma)\neq0$. Moreover, if $r(\gamma)<0$ (resp. $r(\gamma)>0$), then $\gamma$ is stable (resp. unstable). Let $X\in\mathfrak{X}_d$ be given by $X=(P,Q)$. Let $F\colon\mathbb{R}^2\to\mathbb{R}$ be a polynomial. We say that $F$ is an \emph{invariant algebraic curve} for $X$ if
\begin{equation}\label{24}
	P(x,y)\frac{\partial F}{\partial x}(x,y)+Q(x,y)\frac{\partial F}{\partial y}(x,y)=K(x,y)F(x,y),
\end{equation}
for some polynomial $K\colon\mathbb{R}^2\to\mathbb{R}$. In special, observe that the set $F^{-1}(0)$ is invariant by the flow of $X$. If $\gamma$ is limit cycle of $X$ such that $\gamma\subset F^{-1}(0)$, then $\gamma$ is an \emph{algebraic limit cycle} of $X$.  

\section{Necessary conditions for strong structural stability}\label{Sec4}

\begin{lemma}\label{Lemma1}
	Let $X\in\Sigma_d^s$. Then $p(X)$ has a finite number of singularities.
\end{lemma}

\begin{proof} Let $X=(P,Q)$. If $P$ and $Q$ have a common factor, then it is clear that we can take an arbitrarily small perturbation $Y=(R,S)$ of $X$ such that $R$ and $S$ has no common factor. Therefore, it follows from Bezout's Theorem (see \cite{Ful}, p. 57) that $Y$ has finitely many singularities. Let
	\[R=R_0+R_1+\dots+R_{d+2}, \quad S=S_0+S_1+\dots+S_{d+2},\]
where $R_i$ and $S_i$, $i\in\{1,\dots,d+2\}$ are homogeneous polynomials of degree $i$, i.e.
	\[R_i(tx,ty)=t^iR_i(x,y), \quad S_i(tx,ty)=t^iS_i(x,y),\]
for all $t\in\mathbb{R}$. Let $U_1$ and $V_1$ denote the first and second chart of the Poincar\'e Compactification (see Section~\ref{sub1}). Observe (see \cite{DumLliArt2006}, p. 154) that a point $(u_0,0)\in\mathbb{S}^1\cap(U_1\cup V_1)$ is a singularity at infinity of $p(Y)$ if, and only if,
\begin{equation}\label{1}
	S_{d+2}(1,u_0)=u_0R_{d+2}(1,u_0).
\end{equation}
Similarly, a point $(u_0,0)\in\mathbb{S}^1\cap(U_2\cup V_2)$ is a singularity at infinity of $p(Y)$ if, and only if,
\begin{equation}\label{2}
	R_{d+2}(u_0,1)=u_0S_{d+2}(u_0,1).
\end{equation}
Therefore, taking $Z$ arbitrarily close to $Y$ if necessary, we can assume that $Y$ is such that \eqref{1} and \eqref{2} have at most a finite number of solutions. Hence, $p(Y)$ has only a finite number of singularities. Since $Y$ is arbitrarily close to $X\in\Sigma_d^s$, it follows that $p(X)$ and $p(Y)$ are topologically equivalent and thus $p(X)$ has at most a finite number of singularities. \end{proof}

\begin{remark}
	In the context of the proof of Lemma~\ref{Lemma1}, we stress that $X=(P,Q)$ is given by
		\[P(x,y)=x(x-1)f(x,y), \quad Q(x,y)=y(y-1)g(x,y),\]
	with $f$ and $g$ polynomials. Hence, when we take arbitrarily small perturbation of $Y=(R,S)$ of $X$ we mean
		\[R(x,y)=x(x-1)r(x,y), \quad S(x,y)=y(y-1)s(x,y),\]
	where $r$ and $s$ are arbitrarily small perturbations of $f$ and $g$, respectively. We stress this because in our context the perturbations taken through this paper must still lie in $\mathfrak{X}_d$, i.e. we must always take care to ensure that $Y\in\mathfrak{X}_d$. As the reader shall in the next lemma and along this paper, this preoccupation will force us to construct suitable perturbations that ``breaks'' different kind of objects, but without breaking the invariant set $\Lambda$ (recall \eqref{37} and Figure~\ref{Fig15}).
\end{remark}

\begin{lemma}\label{Lemma2}
	Let $X\in\Sigma_d^s$. Then every singularity of $p(X)$ is simple.
\end{lemma}

\begin{proof} Suppose by contradiction that $p(X)$ has a non-simple singularity $p$. Suppose first that $p$ is a finite singularity and thus a singularity of $X$. Let $S$ be a small circle centered at $p$ and let $K$ be the bounded component of $\mathbb{R}^2\backslash S$. It follows from Lemma~\ref{Lemma1} that we can choose $S$ small enough such that there are no other singularity of $p(X)$ at $K$. Let $N\subset\mathfrak{X}_d$ be a small neighborhood of $X$. It follows from the \emph{strong} structural stability of $X$ that we can choose $N$ small enough such that every $Y\in N$ has an unique singularity $p_Y$ at $K$. Moreover, taking $N$ sufficiently small, we can also assume that for every $s\in S$, $X(s)$ and $Y(s)$ never have opposite directions. Hence, it follows from Proposition~\ref{P1} that $i(Y,S)=i(X,S)$, for any $Y\in N$. Moreover, it follows from Proposition~\ref{P3} that $i(Y,S)=i(Y,p_Y)$. We claim that there are $Y_1$, $Y_2\in N$ such that $p_{Y_1}$ and $p_{Y_2}$ have Poincar\'e index $+1$ and $-1$, contradicting the fact that $i(Y,S)$ is constant. Indeed, since $X=(P,Q)$ is given by \eqref{0}, it follows that if we translate $p$ to the origin, we obtain a vector field $Z=(R,S)$ given by
	\[R(x,y)=(x+\alpha)(x+\alpha-1)r(x,y), \quad S(x,y)=(y+\beta)(y+\beta-1)s(x,y),\]
for some $\alpha$, $\beta\in\mathbb{R}$, with $r$ and $s$ polynomials. Since the origin is a non-simple singularity of $Z$, it follows that $\det DZ(0,0)=0$. Given $\varepsilon>0$ and $a$, $b\in\mathbb{R}$, consider $\overline{Z}=(\overline{R},\overline{S})$ given by
	\[\overline{R}(x,y)=(x+\alpha)(x+\alpha-1)\bigl(r(x,y)+\varepsilon a x\bigr), \quad \overline{S}(x,y)=(y+\beta)(x+\beta-1)\bigl(s(x,y)+\varepsilon b y\bigr).\]
Observe that the origin is a singularity of $\overline{Z}$ and that,
	\[D\overline{Z}(0,0)=\left(\begin{array}{cc} \displaystyle \frac{\partial R}{\partial x}(0,0)+\alpha(\alpha-1)a\varepsilon & \displaystyle \frac{\partial R}{\partial y}(0,0) \vspace{0.2cm} \\ \displaystyle \frac{\partial S}{\partial x}(0,0) & \displaystyle \frac{\partial S}{\partial y}(0,0)+\beta(\beta-1)b\varepsilon \end{array}\right).\]
Since $\det DZ(0,0)=0$, it follows that
	\[\det D\overline{Z}(0,0)=\varepsilon\left(\alpha(\alpha-1)a\frac{\partial S}{\partial y}(0,0)+\beta(\beta-1)\frac{\partial R}{\partial x}(0,0)+\alpha(\alpha-1)\beta(\beta-1)ab\varepsilon\right).\]
If $\alpha(\alpha-1)\neq0$ or $\beta(\beta-1)\neq0$, then it is clear that we can choose $a$, $b\in\mathbb{R}$ and $\varepsilon>0$ such that $\det D\overline{Z}(0,0)$ is positive or negative. Moreover, taking $\varepsilon>0$ sufficiently small, it also follows that $\overline{Z}$ is arbitrarily close to $Z$. Hence, if we translate the origin back to $p$, we obtain vector fields $Y_1$ and $Y_2$, as desired. If $\alpha(\alpha-1)=0$ and $\beta(\beta-1)=0$, then we define
	\[\overline{R}(x,y)=(x+\alpha)(x+\alpha-1)\bigl(r(x,y)+\varepsilon a \bigr), \quad \overline{S}(x,y)=(y+\beta)(x+\beta-1)\bigl(s(x,y)+\varepsilon b \bigr),\]
and the proof is the same. Suppose now that $p$ is non-finite singularity of $p(X)$. Suppose first that $p$ is a singularity at the chart $U_1$ of $p(X)$. Let
	\[f=f_0+f_1+\dots+f_d, \quad g=g_0+g_1+\dots+g_d,\]
where $f_i$ and $g_i$ are homogeneous polynomials of degree $i$, and observe that the decomposition of $P$ and $Q$ in homogeneous polynomials are given by
	\[P=(-xf_0)+(x^2f_0-xf_1)+\dots+x^2f_d, \quad Q=(-yg_0)+(y^2g_0-yg_1)+\dots+y^2g_d.\]
Calculations (see \cite{DumLliArt2006}, p. 154) show that a point $(u,0)\in \mathbb{S}^1\cap U_1$ is a singularity of $p(X)$ if, and only if, $F(u)=0$, where
\begin{equation}\label{7}
	F(u)=u\bigl(ug_d(1,u)-f_d(1,u)\bigr).
\end{equation}
Moreover, if $p=(u_0,0)$, then the Jacobian matrix of $p(X)$ at $p$ is given by,
	\[J=\left(\begin{array}{cc} F'(u_0) & \star \vspace{0.2cm} \\ 0 & -f_d(1,u_0) \end{array}\right).\]
Observe that,
\begin{equation}\label{8}
	F'(u_0)=u_0g_d(1,u_0)-f_d(1,u_0)+u_0\left(g_d(1,u_0)+u_0\frac{\partial g_d}{\partial y}(1,u_0)-\frac{\partial f_d}{\partial y}(1,u_0)\right).
\end{equation}
Suppose first $u_0=0$. In this case, we have
	\[J=\left(\begin{array}{cc} -f_d(1,0) & \star \vspace{0.2cm} \\ 0 & -f_d(1,0) \end{array}\right).\]
Therefore, if $f_d(x,y)=\sum_{i=0}^{d}a_ix^{d-i}y^{i}$, then $f_d(1,0)=a_0$. Since $p$ is non-simple, it follows that $\det J=0$ and thus $a_0=0$. If we replace $y=0$, then $X=(P,Q)$ is given by,
	\[P(x,0)=x(x-1)f(x,0), \quad Q(x,0)=0.\]
Therefore, if $f(x,y)=\sum_{i,j=0}^{d}a_{ij}x^i y^j$, then 
	\[P(x,0)=x(x-1)(a_{00}+a_{10}x+a_{20}x^2+\dots+a_{d-1,0}x^{d-1}).\]
It follows from Lemma~\ref{Lemma1} that $a_{i,0}\neq0$, for some $i\in\{0,1,\dots,d-1\}$ (otherwise $X$ would have infinitely many singularities). Let $k$ be the greatest index such that $a_{k,0}\neq0$. Without loss of generality, suppose $a_{k,0}>0$. Hence, it follows that there is $x_0>0$ such that $P(x,0)>0$, for all $x>x_0$. Consider now $Y=(R,S)$ given by
	\[R(x,y)=x(x-1)\bigl(f(x,y)+\lambda x^d\bigr), \quad S(x,y)=y(y-1)g(x,y),\]
with $\lambda<0$ such that $|\lambda|$ is small enough. Observe that $S(x,0)=0$. Since $\lambda<0$, it follows that there is $x_1>0$ such that $P(x,0)<0$ for all $x>x_1$. If $|\lambda|>0$ is small enough, then $x_1>x_0$ and therefore there is $x_0<\overline{x}<x_1$ such that $(\overline{x},0)$ is a singularity of $Y$. Such singularity bifurcated from $p$ and thus $Y$ has one singularity more than $X$, contradicting the strong structural stability of $X$. If $u_0\neq0$, then it follows from \eqref{7} and \eqref{8} that
\begin{equation}\label{16}
	F'(u_0)=u_0\left(g_d(1,u_0)+u_0\frac{\partial g_d}{\partial y}(1,u_0)-\frac{\partial f_d}{\partial y}(1,u_0)\right).
\end{equation}
Let,
	\[f_d(x,y)=\sum_{i=0}^{d}a_ix^{d-i}y^{i}, \quad g_d(x,y)=\sum_{i=0}^{d}b_ix^{d-i}y^{i}.\]
If $f_d(1,u_0)=0$, then we take a small perturbation in $a_0$ and thus $f_d(1,u_0)\neq0$. If $F'(u)=0$, then we take a small perturbation in $b_0$. Such perturbation affect $g_d(1,u_0)$ without changing
	\[\frac{\partial g_d}{\partial y}(1,u_0), \quad \frac{\partial f_d}{\partial y}(1,u_0).\]
Thus, it is not difficult to see that we can take an arbitrarily small perturbation on $a_0$ and $b_0$ such that $\det J>0$ and $\det J<0$. Such bifurcations will, similarly to the finite case, lead to a contradiction in the index theory. Finally, if the $p$ is the origin of the chart $U_2$ of $p(X)$, then the proof follows similarly to the case $u_0=0$. \end{proof}

\begin{lemma}\label{Lemma3}
	Let $X\in\Sigma_d^s$. Then no singularity of $p(X)$ is a center.
\end{lemma}

\begin{proof} Let $p_1,\dots,p_k$ be the singularities of $p(X)$ such that the determinant of the Jacobian matrix of $p(X)$ at $p_i$ is positive and its trace is zero. In special, observe that $p_i$ is either a center or a weak focus (see Section~\ref{sub2}). Let $q_1,\dots,q_r$ be the other singularities of $p(X)$. It follows from Lemma~\ref{Lemma2} that $q_j$ is hyperbolic, for all $j$. Therefore, there is a neighborhood $N\subset\mathfrak{X}_d$ of $X$ such that $q_j(Y)$ is hyperbolic, for all $Y\in N$ and $j\in\{1,\dots,r\}$. Moreover, since the infinity of $p(X)$ is invariant, it follows that $p_1,\dots,p_k$ are finite singularities of $p(X)$ and thus are singularities of $X$. Similarly, $p_1,\dots,p_k$ cannot lie in the invariant straight lines $x=0$, $x=1$, $y=0$ and $y=1$. Therefore, if $X=(P,Q)$ is given by \eqref{0}, then $f(p_i)=g(p_i)=0$, $i\in\{1,\dots,k\}$. It follows from the choice of $p_i$ that $\det DX(p_i)>0$ and $\text{T}(DX(p_i))=0$. Given $\varepsilon>0$, let $Y=(R,S)$ be given by
	\[R(x,y)=x(x-1)\bigl(f(x,y)-\varepsilon g(x,y)\bigr), \quad S(x,y)=y(y-1)\bigl(g(x,y)+\varepsilon f(x,y)\bigr).\]
Since $f(p_i)=g(p_i)=0$, it follows that $p_i$ is also a singularity of $Y$, $i\in\{1,\dots,n\}$. Moreover, if $\varepsilon>0$ is small enough, then $Y\in N$ and thus the other singularities of $p(Y)$ are hyperbolic. If $p_i=(a_i,b_i)$, then observe that
\begin{equation}\label{11}
	\begin{array}{l}
		\displaystyle DX(p_i)=\left(\begin{array}{cc} \displaystyle a_i(a_i-1)\frac{\partial f}{\partial x}(p_i) & \displaystyle a_i(a_i-1)\frac{\partial f}{\partial y}(p_i) \vspace{0.2cm} \\ \displaystyle b_i(b_i-1)\frac{\partial g}{\partial x}(p_i) & \displaystyle b_i(b_i-1)\frac{\partial g}{\partial y}(p_i) \end{array}\right), \vspace{0.2cm} \\
		\displaystyle DY(p_i)=\left(\begin{array}{cc} \displaystyle a_i(a_i-1)\left(\frac{\partial f}{\partial x}(p_i) - \varepsilon\frac{\partial g}{\partial x}(p_i)\right) & \star \vspace{0.2cm} \\ \star & \displaystyle b_i(b_i-1)\left(\frac{\partial g}{\partial y}(p_i)+\varepsilon\frac{\partial f}{\partial y}(p_i)\right) \end{array}\right).
	\end{array}
\end{equation}
Therefore,
\begin{equation}\label{9}
	T(DY(p_i))=\varepsilon\left(b_i(b_i-1)\frac{\partial f}{\partial y}(p_i)-a_i(a_i-1)\frac{\partial g}{\partial x}(p_i)\right).
\end{equation}
We claim that \eqref{9} does not vanish. Indeed, if \eqref{9} does vanish, then
\begin{equation}\label{10}
	a_i(a_i-1)\frac{\partial g}{\partial x}(p_i)=b_i(b_i-1)\frac{\partial f}{\partial y}(p_i).
\end{equation}
Moreover, since $T(DX(p_i))=0$, it follows from \eqref{11} that,
\begin{equation}\label{12}
	 a_i(a_i-1)\frac{\partial f}{\partial x}(p_i)=-b_i(b_i-1)\frac{\partial g}{\partial y}(p_i).
\end{equation}	
Replacing \eqref{12} and \eqref{10} at \eqref{11} we have,
\begin{equation}\label{14}
	DX(p_i)=\left(\begin{array}{cc} \displaystyle -b_i(b_i-1)\frac{\partial g}{\partial y}(p_i) & \displaystyle a_i(a_i-1)\frac{\partial f}{\partial y}(p_i) \vspace{0.2cm} \\ \displaystyle b_i(b_i-1)\frac{\partial g}{\partial x}(p_i) & \displaystyle b_i(b_i-1)\frac{\partial g}{\partial y}(p_i) \end{array}\right).
\end{equation}
Observe that $a_i(a_i-1)\neq0$ and $b_i(b_i-1)\neq0$ because $p_i$ does not lie in the invariant lines $x=0$, $x=1$, $y=0$ or $y=1$. Moreover, observe that
	\[\frac{\partial f}{\partial y}(p_i)\neq0, \quad \frac{\partial g}{\partial x}(p_i)\neq0,\]
otherwise, it follows from \eqref{14} that we would have $\det DX(p_i)\leqslant0$. Then, it follows from \eqref{10} that
\begin{equation}\label{13}
	\frac{\partial f}{\partial y}(p_i)=\frac{a_i(a_i-1)}{b_i(b_i-1)}\frac{\partial g}{\partial x}(p_i).
\end{equation}
Replacing \eqref{13} at \eqref{14} we obtain $\det DX(p_i)\leqslant0$, contradicting the choice of $p_i$. Hence, $T(DY(p_i))\neq0$ and thus $p_i$ is a hyperbolic focus of $Y$, for all $i\in\{1,\dots,k\}$. Therefore, $Y$ has no centers. Since $p(X)$ is topologically equivalent to $p(Y)$, it follows that $X$ has no centers. \end{proof}

\begin{remark}\label{Remark1}
	Let $D_1\subset\mathfrak{X}_d$ be the family of vector fields satisfy condition $(a')$ stated in Section~\ref{Sec2}. Then it follows from the proofs of Lemmas~\ref{Lemma1}, \ref{Lemma2} and \ref{Lemma3} that $D_1$ is open and dense in $\mathfrak{X}_d$.
\end{remark}

A common fact on the structural stability of smooth, analytical and polynomial vector fields is that none of them allow finite saddles connection. However, in our context the set $\Lambda$ (recall \eqref{37} Figure~\ref{Fig15}) is a ``fixed structure'', i.e. every $X\in\mathfrak{X}_d$ has $\Lambda$ as an invariant set. Therefore, as we shall in the next lemma and further ahead in this paper, different from the previous works on structural stability, this fixed structure will ``protect'' the saddle connections within and thus allow the existence of certain saddle connections and polycycles as structurally stable objects.

\begin{lemma}\label{Lemma4}
	Let $X\in\Sigma_d^s$. If $p(X)$ has an orbit $\gamma$ going from saddle to saddle, then $\gamma\subset\Lambda$.
\end{lemma}

\begin{proof} Suppose by contradiction that $p(X)$ has an orbit $\gamma\not\subset\Lambda$ from saddle to saddle. Let $X=(P,Q)$ be given by
\begin{equation}\label{3}
	P(x,y)=x(x-1)f(x,y), \quad Q(x,y)=y(y-1)g(x,y).
\end{equation}
Given $\varepsilon>0$, let $Y_\lambda=(P_\lambda,Q_\lambda)$ be the $1$-parameter family given by
\begin{equation}\label{4}
	P_\lambda(x,y)=x(x-1)\bigl(f(x,y)-\lambda\varepsilon g(x,y)\bigr), \quad Q_\lambda(x,y)=y(y-1)\bigl(g(x,y)+\lambda\varepsilon f(x,y)\bigr),
\end{equation}
with $\lambda\in[-1,1]$. Observe that $Y_\lambda$ is arbitrarily close to $X$, provided that $\varepsilon>0$ is small enough. Let $p_1,\dots,p_k$ be the saddles of $p(X)$. It follows from Lemma~\ref{Lemma2} that $p_i$ is hyperbolic, $i\in\{1,\dots,k\}$. Since $p(X)$ is topologically equivalent to $p(Y_\lambda)$, it follows that $p_i(\lambda)$, $i\in\{1,\dots,k\}$, are the only saddles of $p(Y_\lambda)$, $\lambda\in[-1,1]$. Let $p_i=p$ and $p_j=q$ (we are not excluding the case $p=q$) be such that $p(X)$ has an orbit $\gamma$ from $p$ to $q$, with $\gamma\not\subset\Lambda$. Since the infinity of $p(X)$ is invariant, it follows that $\gamma$ is an orbit of $X$ (even if $p$ or $q$ is at infinity). Given $x_0\in\gamma$, let $l_0$ be a transversal section to $\gamma$ at $x_0$, endowed with a coordinate system $n$. Let $\gamma^s_\lambda$ and $\gamma^u_\lambda$ be the perturbations of $\gamma$ such that $\omega(\gamma^s_\lambda)=q(\lambda)$ and $\alpha(\gamma^u_\lambda)=p(\lambda)$. Let $x^s_\lambda$ and $x^u_\lambda$ be the intersections of $\gamma^s_\lambda$ and $\gamma^u_\lambda$ with $l_0$ and $n^s(\lambda)$, $n^u(\lambda)$ be its coordinates along $l_0$. It follows from \cite{Perko1994} that the \emph{displacement function} $D\colon[-1,1]\to\mathbb{R}$ given by
	\[D(\lambda)=n^u(\lambda)-n^s(\lambda),\]
is well defined and of class $C^\infty$.  See Figure~\ref{Fig2}.
\begin{figure}[h]
	\begin{center}
		\begin{minipage}{5.5cm}
			\begin{center}
				\begin{overpic}[width=5cm]{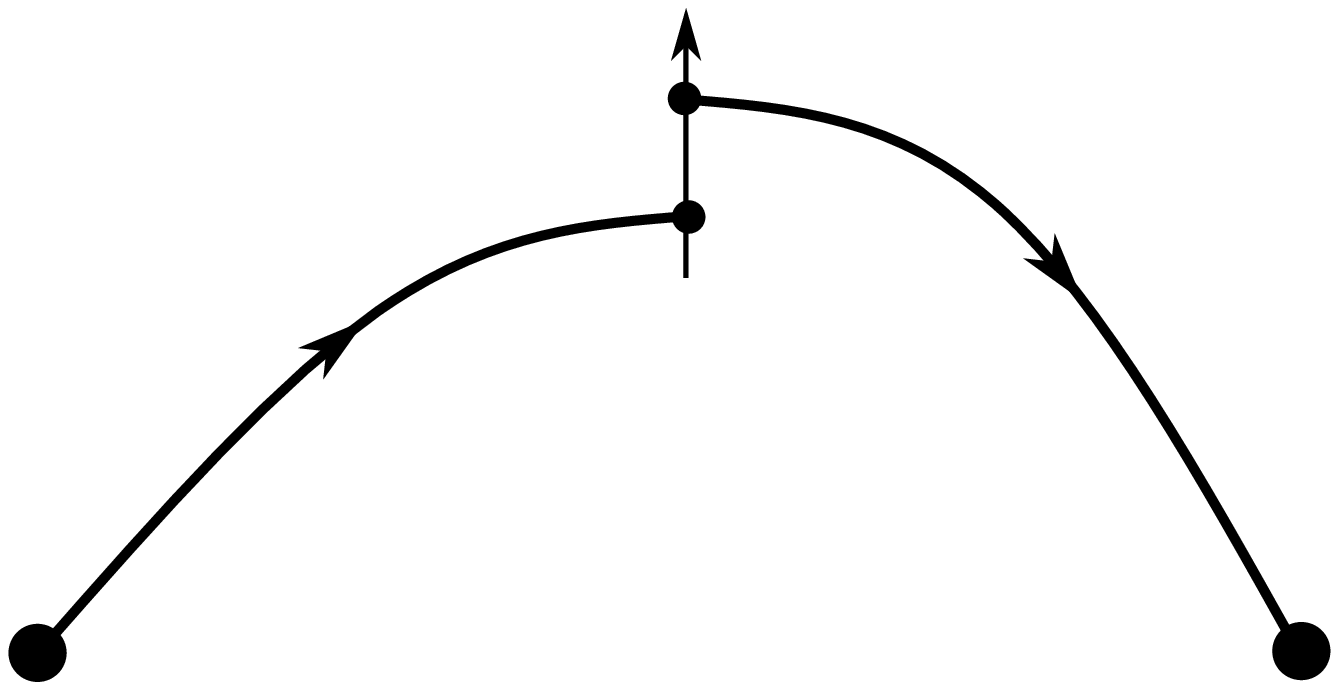} 
					\put(7,0){$p$}
					\put(90,0){$q$}
					\put(54,47){$x_\lambda^s$}
					\put(54,33){$x_\lambda^u$}
					\put(85,25){$\gamma_\lambda^s$}
					\put(5,18){$\gamma_\lambda^u$}
				\end{overpic}
				
				$D(\lambda)<0$.
			\end{center}
		\end{minipage}
		\begin{minipage}{5.5cm}
			\begin{center}
				\begin{overpic}[width=5cm]{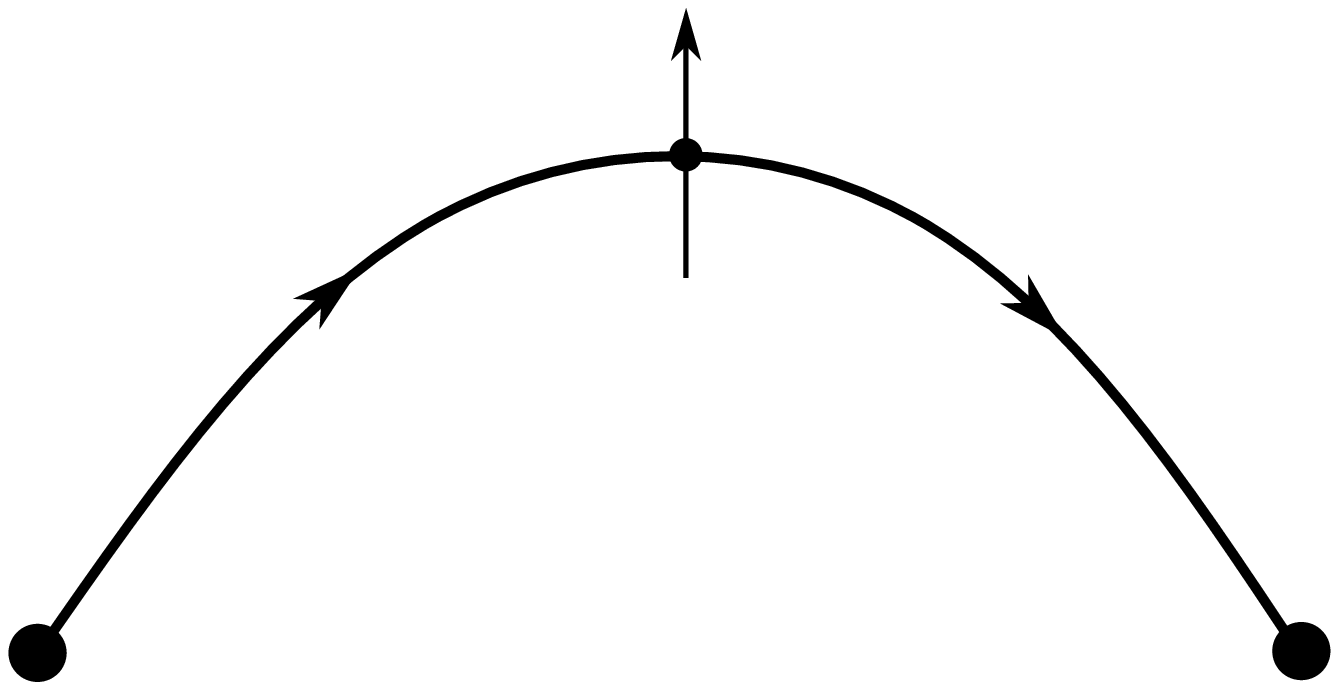} 
					\put(7,0){$p$}
					\put(90,0){$q$}
					\put(53,48){$l_0$}
					\put(52,34){$x_0$}
					\put(85,22){$\gamma$}
				\end{overpic}
				
				$D(\lambda)=0$.
			\end{center}
		\end{minipage}
		\begin{minipage}{5.5cm}
			\begin{center}
				\begin{overpic}[width=5cm]{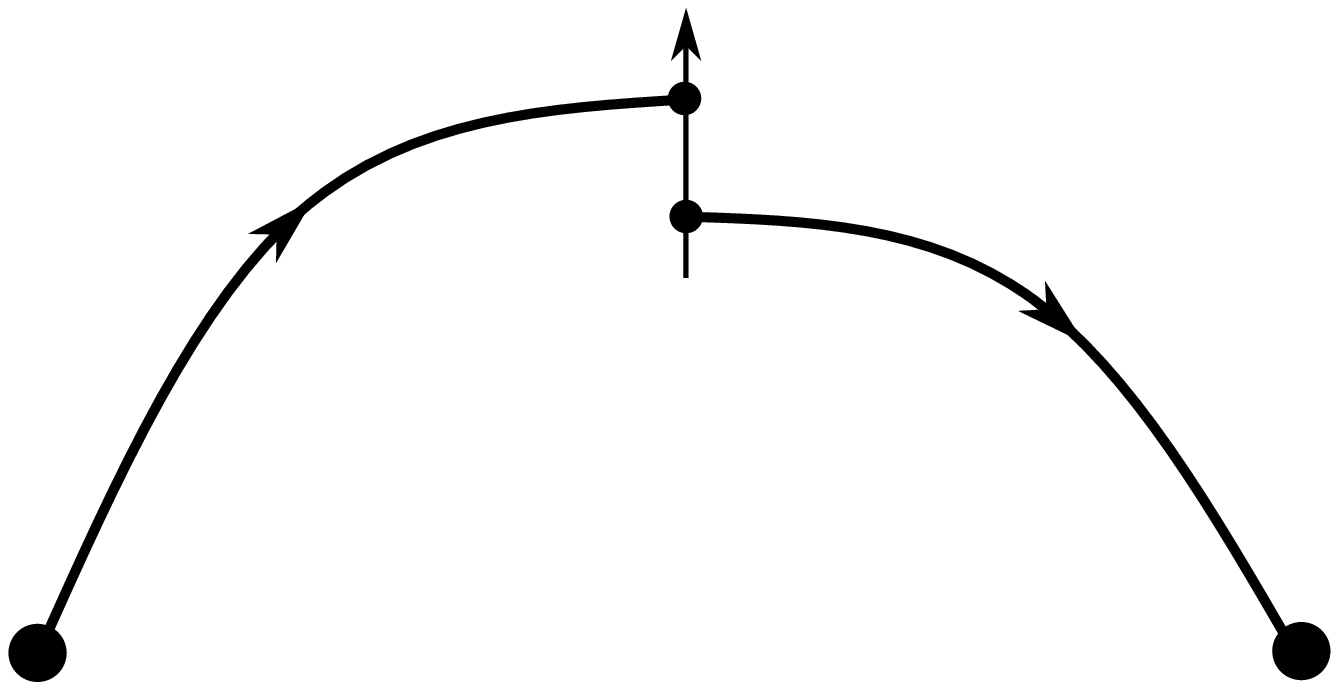} 
					\put(7,0){$p$}
					\put(90,0){$q$}
					\put(54,43){$x_\lambda^u$}
					\put(41,33){$x_\lambda^s$}
					\put(85,25){$\gamma_\lambda^s$}
					\put(13,18){$\gamma_\lambda^u$}
				\end{overpic}
				
				$D(\lambda)>0$.
			\end{center}
		\end{minipage}
	\end{center}
	\caption{The displacement map of a connection between saddles.}\label{Fig2}
\end{figure}

Moreover, it also follows from \cite{Perko1994} that
\begin{equation}\label{5}
	D'(0)=-\frac{1}{||Y_0(x_0)||}\int_{-\infty}^{+\infty}\left(e^{-\int_{0}^{t}Div(\gamma(s))\;ds}\right)Y_0(\gamma(t))\land\dfrac{\partial Y_0}{\partial \lambda}(\gamma(t))\;dt,
\end{equation}
where $\gamma(t)$ is a parametrization of $\gamma$ such that $\gamma(0)=x_0$ and $(x_1,x_2)\land(y_1,y_2)=x_1y_2-x_2y_1$. Let
	\[H(x,y)=Y_0(x,y)\land\dfrac{\partial Y_0}{\partial \lambda}(x,y),\]
and observe that,
	\[H(x,y)=x(x-1)y(y-1)\varepsilon\bigl[f(x,y)^2+g(x,y)^2\bigr].\]
Since $\gamma\not\subset\Lambda$ and $\Lambda$ is invariant, it follows that $\gamma\cap\Lambda=\emptyset$. In particular, $\gamma\subset\mathbb{R}^2$. Therefore, $H$ has constant sign on $\gamma$ and thus it follows from \eqref{5} that $D'(0)\neq0$. Hence, $\lambda=0$ is an isolated zero of $D(\lambda)$. Let $\varepsilon>0$ be small enough such that $\lambda=0$ is the unique zero of $D(\lambda)$. We now proof that if $|\lambda|>0$ is small enough, then $Y_\lambda$ does not have a connection between $p$ and $q$. Suppose by contradiction that $Y_\lambda$ has an orbit $\gamma_\lambda$ from $p$ to $q$. See Figure~\ref{Fig3}.
\begin{figure}[h]
	\begin{center}
		\begin{overpic}[height=4cm]{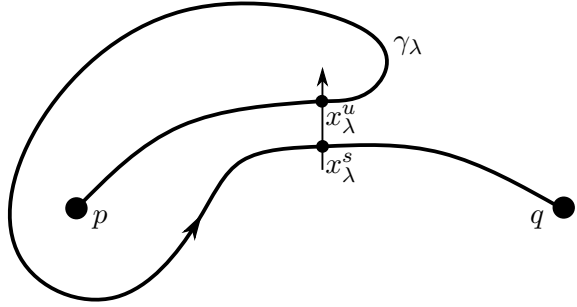} 
			\put(15,14){$p$}
			\put(92,14){$q$}
			\put(68,45){$\gamma_\lambda$}
			\put(56,31.5){$x_\lambda^u$}
			\put(56,23){$x_\lambda^s$}
		\end{overpic}
	\end{center}
	\caption{An illustration of $\gamma_\lambda$.}\label{Fig3}
\end{figure}
Let $\gamma_\lambda(t)$ be a parametrization of $\gamma_\lambda$ such that $\gamma_\lambda(0)=x_\lambda^u$. It follows from the continuity of solutions with respect to initial conditions (see Theorem~$4$, p. 92 of \cite{Perko2001}) that 
\begin{equation}\label{17}
	\lim\limits_{\lambda\to0}\gamma_\lambda(t)=\gamma(t)
\end{equation}
for all $t\in\mathbb{R}$, with the convergence being uniform on any compact interval $t\in[-t_0,t_0]$. Moreover, it follows from the Stable Manifold Theorem (see Section $2.7$ of \cite{Perko2001}) that the limit \eqref{17} is uniform for all $|t|\geqslant t_0$, provided $t_0>0$ is sufficiently large. Hence, if $\ell_0$ denotes the length (in relation to the standard norm of $\mathbb{R}^2$) of $\gamma$ and $\ell(\lambda)$ denotes the length of $\gamma_\lambda$, it follows that
	\[\lim\limits_{\lambda\to0}\ell(\lambda)=\ell_0,\]
contradicting the fact that $D(\lambda)\neq0$. \end{proof}

\begin{remark}\label{Remark2}
	Let $D_2\subset\mathfrak{X}_d$ be the family of vector fields satisfy conditions $(a')$ and $(c)$, stated in Section~\ref{Sec2}. Then it follows from Remark~\ref{Remark1} and from the proof of Lemma~\ref{Lemma4} that $D_2$ is open and dense in $\mathfrak{X}_d$.
\end{remark}

\begin{lemma}\label{Lemma5}
	Let $X\in\Sigma_d^s$. Then every closed orbit of $p(X)$ is a limit cycle. 
\end{lemma}

\begin{proof} Let $\gamma$ be a closed orbit of $p(X)$. Observe that the infinity of $p(X)$ has at least four singularities. Therefore, $\gamma$ is in fact a periodic orbit of $X$. Since $X$ is polynomial, it follows from the \emph{Poincar\'e Theorem} (see \cite{Perko2001}, p. 217) that $\gamma$ is either a limit cycle or it is contained in an open ring $R$ of periodic orbits. Suppose by contradiction that the later holds and let $R$ be the maximal ring of periodic orbits. Let $K$ be the bounded region limited by $\gamma$. It follows from Remark~\ref{Remark0} that there is at least one singularity $p\in K$. It follows from Lemma~\ref{Lemma3} that $p$ is not a center and thus there is a neighborhood $V$ of $p$ such that $V\cap R=\emptyset$. Let $q$ be a point of $\gamma$ such that the segment $\overline{pq}$ is contained in $K$ and is transversal to $\gamma$ (such point exists because $\gamma$ is an analytical curve). The set $\overline{pq}\backslash\{\overline{pq}\cap R\}$ is closed in $\overline{pq}$. Let $r\in\overline{pq}\cap\partial R$ be the closest point to $q$. We claim that we can suppose that $r$ is not a singularity of $X$. Indeed, if $r$ is a singularity, then it follows from Lemma~\ref{Lemma1} that it is isolated. Therefore, it follows from the transversality between $\overline{pq}$ and $\gamma$ that we can take a slight change in $q$ and thus suppose that $r$ is not a singularity.	Observe that the orbit passing through any $s\in \overline{rq}$ is contained in $R$ and thus it is closed. Let $\nu_+$ be the positive semi-orbit passing through $r$. It follows from the Poincar\'e-Bendixson Theorem (see Section~$3.7$ of \cite{Perko2001}) that one of the following statements holds.
\begin{enumerate}[label=(\alph*)]
	\item $\nu_+$ approaches a singularity $\rho$;
	\item $\nu_+$ approaches a limit cycle or a polycycle.
\end{enumerate}
Observe that statement $(b)$ cannot hold, otherwise the orbits near $r$ on $\overline{rq}$ would also approach the same limit cycle or polycycle, contradicting the fact that they are closed. Hence, statement $(a)$ holds and $\rho$ must be a saddle. The same argument shows that the negative semi-orbit $\nu_-$ through $r$ also goes to some saddle. Hence, the orbit $\nu$ though $r$ is a connection between saddles, contradicting Lemma~\ref{Lemma4}. This contradiction finishes this prove.\end{proof}

As stated in Section~\ref{Sec2}, the proof that non-hyperbolic limit cycles of even degree are structurally unstable in the family of planar polynomial vector fields follows from the theory of rotated vector fields. More precisely, if $X=(P,Q)$ is a polynomial vector field, then $Y_\lambda=X+\lambda X^\perp$ is the suitable perturbation that breaks the non-hyperbolic limit cycle, where $X^\perp=(-Q,P)$ and $|\lambda|>0$ is small enough. However, in our context this perturbation is not suitable because it breaks the set $\Lambda$ (recall \eqref{37} and Figure~\ref{Fig15}), i.e. $Y_\lambda\not\in\mathfrak{X}_d$. Hence, as we shall see in the proof of the next lemma, to overcome this situation we will make use of a similar perturbation but with a different proof, that does not rely on rotated vector fields.

\begin{lemma}\label{Lemma6}
	Let $X\in\Sigma_d^s$. If $\gamma$ is a limit cycle of $X$, then it is not semi-stable.
\end{lemma}

\begin{proof} Observe that the infinity of $p(X)$ has at least four singularities. Therefore, any limit cycle of $p(X)$ is a limit cycle of $X$. Suppose by contradiction that $\gamma$ is a semi-stable limit cycle of $X$, i.e. $\gamma$ is a non-hyperbolic limit cycle of $X$, with even multiplicity. Let $N\subset\mathfrak{X}_d$ be a small enough neighborhood of $X$ such that the displacement map $D\colon N\times L\to\mathbb{R}$ is well defined, where $L$ is a transversal section of $\gamma$. Let $n$ be the multiplicity of $\gamma$. That is, $n\geqslant2$ is the first integer such that
	\[\frac{\partial^n D}{\partial x^n}(X,0)\neq0.\]
It follows from Weierstrass Preparation Theorem (see Chapter 4 of \cite{GolGui2012}) that
	\[D(Y,x)=U(Y,x)P(Y,x),\] 
where $U$ is a strictly positive analytical function and
	\[P(Y,x)=x^n+a_{n-1}(Y)x^{n-1}+\dots+a_1(Y)x+a_0(Y),\]
where $a_i$ is analytical and $a_i(X)=0$, $i\in\{1,\dots,n-1\}$. Let $X=(P,Q)$. Given $\varepsilon>0$, let $Y_\lambda$ be the $1$-parameter family given by
\begin{equation}\label{15}
	P_\lambda(x,y)=x(x-1)\bigl(f(x,y)-\lambda\varepsilon g(x,y)\bigr), \quad Q_\lambda(x,y)=y(y-1)\bigl(g(x,y)+\lambda\varepsilon f(x,y)\bigr),
\end{equation}
with $\lambda\in[-1,1]$. Observe that if $\varepsilon>0$ is small enough, then $Y_\lambda\in N$, for all $\lambda\in[-1,1]$. Observe also that $Y_0=X$. Let $\gamma(t)$ be a parametrization of $\gamma$ such that $\gamma(0)=p$, where $\{p\}=L\cap\gamma$. Let $T>0$ be the period of $\gamma$. It follows from Lemma~$2$ of \cite{Perko1992} that,
\begin{equation}\label{18}
	\frac{\partial D}{\partial\lambda}(Y_0,0)=C\int_{0}^{T}\left(e^{-\int_{0}^{t}Div(\gamma(s))\;ds}\right)Y_0(\gamma(t))\land\dfrac{\partial Y_0}{\partial \lambda}(\gamma(t))\;dt,
\end{equation}
where $C\in\mathbb{R}\backslash\{0\}$ and $(x_1,x_2)\land(y_1,y_2)=x_1y_2-x_2y_1$. Observe that, 
	\[a_0(Y)=\frac{D(Y,0)}{U(Y,0)}.\] 
To simplify the notation, let $V(Y,x)=U(Y,x)^{-1}$. Hence,
	\[\frac{\partial a_0}{\partial\lambda}(Y_0)=\frac{\partial V}{\partial\lambda}(Y_0,0)\underbrace{D(Y_0,0)}_{0}+V(Y_0,0)\frac{\partial D}{\partial\lambda}(Y_0,0).\]
Thus, it follows from \eqref{18} that,
\begin{equation}\label{19}
	\frac{\partial a_0}{\partial\lambda}(Y_0)=V(Y_0,0)C\int_{0}^{T}\left(e^{-\int_{0}^{t}Div(\gamma(s))\;ds}\right)Y_0(\gamma(t))\land\dfrac{\partial Y_0}{\partial \lambda}(\gamma(t))\;dt.
\end{equation}
It follows from \eqref{15} that,
\begin{equation}\label{34}
	Y_0(x,y)\land\dfrac{\partial Y_0}{\partial \lambda}(x,y)=x(x-1)y(y-1)\varepsilon\bigl[f(x,y)^2+g(x,y)^2\bigr].
\end{equation}
Since $\gamma$ is a limit cycle, it follows that $\gamma\cap\Lambda=\emptyset$ and thus $x(x-1)y(y-1)$ is non-zero and has constant sign on $\gamma$. Therefore, it follows from \eqref{34} that either
	\[Y_0(\gamma(t))\land\dfrac{\partial Y_0}{\partial \lambda}(\gamma(t))>0, \quad  \text{ or } \quad Y_0(\gamma(t))\land\dfrac{\partial Y_0}{\partial \lambda}(\gamma(t))<0,\]
for all $t\in[0,T]$. Hence, it follows from \eqref{19} that,
	\[\frac{\partial a_0}{\partial\lambda}(X)\neq0.\]
Therefore we conclude that $a_0$ has a non-zero directional derivative at $X$ and thus it is not constant. In special, the polynomial
	\[P(Y,x)=x^n+a_{n-1}(Y)x^{n-1}+\dots+a_1(Y)x+a_0(Y),\]
is not constant in $Y$. From now on, we denote $P(Y)=P(Y,\cdot)$. Given $Y\in N$, it follows from the strong structural stability of $X$ that $Y$ has a limit cycle $\gamma(Y)$ near $\gamma$. In special, if $\gamma(Y)$ is a non-hyperbolic limit cycle of $Y$, then $P(Y)$ has a multiple root. Let $\Delta\colon\mathbb{R}^{n+1}\to\mathbb{R}$ be the discriminant of polynomials of degree $n$ (see Chapter 12 of \cite{GelKapZel1994}). Observe if $\gamma(Y)$ is a non-hyperbolic limit cycle of $Y$, then $\Delta(P(Y))=0$. Since $\Delta\circ P$ is non-constant, it follows from Whitney's Stratification (see Section~\ref{sub5}) that the set $\{Y\in N\colon \Delta(P(Y))=0\}$ is given by the disjoint union of analytical manifolds of codimension at least one, each of them having at most a finite number of connected components. In special, it has zero Lebesgue measure. Therefore, we can take $Y$ arbitrarily close to $X$ such that $\Delta(P(Y))\neq0$. Therefore, since $n$ is even, it follows that $P(Y)$ has either zero or at least two real roots. In the former, $Y$ has no limit cycles near $\gamma$. In the latter, $Y$ has at least two hyperbolic limit cycles near $\gamma$. In either case we have a contradiction with the strong structural stability of $X$. \end{proof}

\begin{remark} Under the context of the proof of Lemma~\ref{Lemma6}, we observe that we had proved that there is $Y\in N$ such that $\Delta(P(Y))\neq0$. However, we were not able to study the sign of $\Delta(P(Y))$. We claim that if one is able to prove that $\Delta(P(Y))<0$, then the non-hyperbolic limit cycle of odd degree is also structurally unstable. Indeed, it follows from \cite{GelKapZel1994}, p. 403, that 
\begin{equation}\label{4}
	\Delta(P(Y))=(-1)^{\frac{n(n-1)}{2}}\prod_{i<j}\bigl(x_i(Y)-x_j(Y)\bigr)^2,
\end{equation}
where $\{x_1(Y),\dots,x_n(Y)\}$ are the roots of $P(Y)$. If $n$ is odd, it follows from Section~$4$ of \cite{NickDye} that the following statements hold.
\begin{enumerate}[label=(\roman*)]
	\item If $\Delta(P(Y))>0$, then the number of real roots of $P$ is congruent to $1$ modulo $4$;
	\item If $\Delta(P(Y))<0$, then the number of real roots of $P$ is congruent to $3$ modulo $4$.
\end{enumerate}
Therefore, if there is $Y\in N$ such that $\Delta(P(Y))<0$, then $Y$ has at least three hyperbolic limit cycles near $\gamma$ and thus one could conclude that non-hyperbolic limit cycles are indeed structurally unstable.
\end{remark}

\begin{remark}\label{Remark3}
	Let $X\in\mathfrak{X}_d$. If $X$ has a non-hyperbolic limit cycle $\gamma$ (whether with odd or even multiplicity), then it follows from the proof of Lemma~\ref{Lemma6} that there is a neighborhood $N\subset\mathfrak{X}_d$ of $X$ and a set $Z\subset N$, with zero Lebesgue measure, such that every $Y\in N\backslash Z$ has either no limit cycle or only hyperbolic limit cycles near $\gamma$.
\end{remark}

\begin{lemma}\label{Lemma7}
	Let $X\in\Sigma_d^s$. Then every singularity of $p(X)$ is hyperbolic.
\end{lemma}

\begin{proof} Suppose by contradiction that $p$ is a non-hyperbolic singularity of $p(X)$. It follows from Lemma~\ref{Lemma2} that $p$ is simple. Therefore, $p$ is either a focus or a center (see Section~\ref{sub2}). Since the infinity of $p(X)$ is invariant, it follows that $p$ is a singularity of $X$. Moreover, it follows from Lemma~\ref{Lemma3} that $p$ is a focus. Without loss of generality, suppose that $p$ is a stable focus. Let $S$ be a small circle centered at $p$ and $K$ be the bounded region delimited by $S$. Take $S$ small enough such that there are no other singularity or limit cycle of $X$ in $K$. Since $p$ is stable, it follows that for every $s\in S$, the vector $X(s)$ points towards $K$. Moreover, we can take a small neighborhood $N\subset\mathfrak{X}_d$ of $X$ such that for every $Y\in N$ and $s\in S$, the vector $Y(s)$ also points towards $K$. It follows from the proof of Lemma~\ref{Lemma3} that we can take $Y\in N$ such that $p$ is a hyperbolic unstable focus of $Y$. Hence, it follows from the Poincar\'e-Bendixson Theorem that $Y$ must have a limit cycle $\gamma\subset\text{Int}(K)$ surrounding $p$ (where $\text{Int}(K)$ denotes the topological interior of $K$). But this imply that $Y$ has one limit cycle more than $X$, contradicting its strong structural stability. \end{proof}

\begin{lemma}\label{Lemma8}
	Let $X\in\Sigma_d^s$. If $p(X)$ has a polycycle $\Omega$, then $\Omega$ is generic and $\Omega\subset\Lambda$.
\end{lemma}

\begin{proof} Let $\Omega$ be a polycycle of $p(X)$. It follows from Lemma~\ref{Lemma4} that $\Omega\subset\Lambda$. Since the north, south, east and west poles of the infinity of $p(X)$ are not hyperbolic saddles, it follows that $\Omega$ is the boundary of the square
	\[\{(x,y)\in\mathbb{R}^2\colon 0\leqslant x\leqslant1,\;0\leqslant y\leqslant 1\}.\]
Without loss of generality, suppose that $\Omega$ has the counterclockwise orientation. Therefore, it follows that,
\begin{equation}\label{21}
	r(\Omega)=\frac{g(0,0)f(1,0)g(1,1)f(0,1)}{f(0,0)g(1,0)f(1,1)g(0,1)}.
\end{equation}
Suppose by contradiction that $r(\Omega)=1$. Suppose first that $\Omega$ is not accumulated by limit cycles. Hence, it is either stable or unstable. Suppose that $\Omega$ is stable. Similarly to the proof of Lemma~\ref{Lemma7}, we claim that we can take $Y$ arbitrarily close to $X$ such that $r(\Omega)<1$, i.e., such that $\Omega$ is unstable. Therefore, similarly to the proof of Lemma~\ref{Lemma7}, it follows that $Y$ has one limit cycle more than $X$, contradicting the strong structural stability of $X$. The desired perturbation is given as follows. Suppose first that $d\geqslant2$ and let
	\[g(x,y)=\sum_{i,j\geqslant0}^{d}b_{ij}x^iy^j.\]
Observe that an arbitrarily small perturbation at the coefficient $b_{11}$ will change $g(1,1)$, without changing $g(0,0)$, $g(1,0)$ or $g(0,1)$. Therefore, if $r(\Omega)=1$, we can use this perturbation to obtain $r(\Omega)>1$ or $r(\Omega)<1$. Suppose now $d=1$. We claim that
\begin{equation}\label{6}
	\frac{g(0,0)g(1,1)}{g(1,0)g(0,1)}>0, \quad \frac{f(1,0)f(0,1)}{f(0,0)f(1,1)}>0.
\end{equation}
Indeed, let $X=(P,Q)$ be given by \eqref{0} and let 
	\[p_1=(0,0), \quad p_2=(1,0), \quad p_3=(1,1), \quad p_4=(0,1),\]
be the hyperbolic saddles of $\Omega$. Since $\Omega$ has the counterclockwise orientation, it follows that 
	\[P(x,0)>0, \quad P(x,1)<0, \quad Q(0,y)<0, \quad Q(1,y)>0,\]
for all $0<x<1$ and $0<y<1$. Hence, we have
\begin{equation}\label{20}
	f(x,0)<0, \quad f(x,1)>0, \quad g(0,y)>0, \quad g(1,y)<0,
\end{equation}
for all $0<x<1$ and $0<y<1$. Since $p_i$ is hyperbolic, it follows that $f(p_i)\neq0$ and $g(p_i)\neq0$, $i\in\{1,2,3,4\}$. Therefore, it follows from \eqref{20} that
	\[f(0,0)<0, \quad f(1,0)<0, \quad f(1,1)>0, \quad f(0,1)>0,\]
	\[g(0,0)>0, \quad g(1,0)<0, \quad g(1,1)<0, \quad g(0,1)>0,\]
and thus we have \eqref{6}. Let
	\[f(x,y)=a_{00}+a_{10}x+a_{01}y, \quad g(x,y)=b_{00}+b_{10}x+b_{01}y.\]
It follows from \eqref{21} that,	
\begin{equation}\label{22}
	r(\Omega)=\frac{g(0,0)g(1,1)}{g(1,0)g(0,1)}\frac{f(1,0)f(0,1)}{f(0,0)f(1,1)} = \frac{b_{00}(b_{00}+b_{10}+b_{01})}{(b_{00}+b_{10})(b_{00}+b_{01})}\frac{(a_{00}+a_{10})(a_{00}+a_{01})}{a_{00}(a_{00}+a_{10}+a_{01})}.
\end{equation}
To simplify the notation, let
	\[a=\frac{(a_{00}+a_{10})(a_{00}+a_{01})}{a_{00}(a_{00}+a_{10}+a_{01})}, \quad b=\frac{(b_{00}+b_{10})(b_{00}+b_{01})}{b_{00}(b_{00}+b_{10}+b_{01})}.\]
It follows from \eqref{6} and \eqref{22} that $a>0$, $b>0$ and $r(\Omega)=a/b$. Therefore, we have that $r(\Omega)\geqslant 1$ if, and only if, $a\geqslant b$, with the equality holding if, and only if, $a=b$. Observe that,
	\[\begin{array}{l}
		\displaystyle \frac{(a_{00}+a_{10})(a_{00}+a_{01})}{a_{00}(a_{00}+a_{10}+a_{01})}=1+\frac{a_{10}a_{01}}{a_{00}(a_{00}+a_{10}+a_{01})}, \vspace{0.2cm} \\ \displaystyle \frac{(b_{00}+b_{10})(b_{00}+b_{01})}{b_{00}(b_{00}+b_{10}+b_{01})}=1+\frac{b_{10}b_{01}}{b_{00}(b_{00}+b_{10}+b_{01})}.
	\end{array}\]
Hence, it follows that $r(\Omega)\geqslant1$ if, and only if,
\begin{equation}\label{23}
	\frac{a_{10}a_{01}}{a_{00}(a_{00}+a_{10}+a_{01})}\geqslant\frac{b_{10}b_{01}}{b_{00}(b_{00}+b_{10}+b_{01})},
\end{equation}
with $r(\Omega)=1$ if, and only if, the equality of \eqref{23} holds. Therefore, it is clear that slight perturbations at $a_{00}$ or $b_{00}$ can make $r(\Omega)>1$ or $r(\Omega)<1$, as desired. Suppose now that $\Omega=\Omega(X)$ is accumulated by limit cycles of $X$. In particular, it follows that $p(X)$ has infinitely many limit cycles. It follows from the previous part of the proof that we can take $Y$ arbitrarily close to $X$ such that $r(\Omega)>1$. In particular, $\Omega=\Omega(Y)$ is stable and thus cannot be accumulated by limit cycles of $Y$. We claim that $p(Y)$ has at most finitely many limit cycles. Indeed, suppose by contradiction that $p(Y)$ has infinitely many limit cycles $\{\gamma_n\}_{n\geqslant 1}$. Since $\mathbb{S}^2$ is compact, it follows that $\{\gamma_n\}$ must accumulate somewhere. More precisely, it follows from \cite{Perko1987} that it must accumulate either on a singularity or in a polycycle. However, it follows from Lemma~\ref{Lemma7} that every singularity of $p(X)$ is hyperbolic. In particular, every singularity of $p(Y)$ is hyperbolic and its unique polycycle is generic. Hence, neither the singularities nor the polycycle can be accumulated by limit cycles and thus $p(Y)$ has finitely many limit cycles, contradicting the strong structural stability of $X$. \end{proof}

\begin{remark}\label{Remark4}
	Let $D_3\subset\mathfrak{X}_d$ be the family of vector fields satisfy conditions $(a')$, $(c)$ and $(d')$ stated in Section~\ref{Sec2}. Then it follows from Remark~\ref{Remark2} and from the proof of Lemma~\ref{Lemma8} that $D_3$ is open and dense in $\mathfrak{X}_d$.
\end{remark}

\begin{lemma}\label{Lemma0}
	Let $X\in\Sigma_d^s$. Then $p(X)$ has a finite number of limit cycles.
\end{lemma}

\begin{proof} Similarly to the last part of the proof of Lemma~\ref{Lemma8}, suppose by contradiction that $p(X)$ has infinitely many limit cycles $\{\gamma_n\}_{n\geqslant 1}$. Since $\mathbb{S}^2$ is compact, it follows that $\{\gamma_n\}$ must accumulate somewhere. More precisely, it follows from \cite{Perko1987} that it must accumulate either on a singularity or in a polycycle. However, it follows from Lemmas~\ref{Lemma7} and \ref{Lemma8} that every singularity of $p(X)$ is hyperbolic and every polycycle, if any, is generic. Hence, it cannot be accumulated by limit cycles. \end{proof}

\begin{remark}\label{Remark5}
	It follows from Remark~\ref{Remark3} that if $X\in\mathfrak{X}_d$ as a non-hyperbolic limit cycle $\gamma$, then there is $Y\in\mathfrak{X}_d$ arbitrarily close to $X$ such that $Y$ has only hyperbolic limit cycles near $\gamma$. Moreover, it follows from Lemma~\ref{Lemma0} that $X$ has at most a finite number of limit cycles. Hence, it follows that there is $Y\in\mathfrak{X}_d$ arbitrarily close to $X$ having only hyperbolic limit cycles. Therefore, if we let $D_4\subset\mathfrak{X}_d$ be the family of vector fields satisfy conditions $(a')$, $(b')$, $(c)$ and $(d')$ stated in Section~\ref{Sec2}, then it follows from Remarks~\ref{Remark3}, \ref{Remark4} and from Lemma~\ref{Lemma0} that $D_4$ is open and dense in $\mathfrak{X}_d$. In particular, $\mathcal{P}_d$ is open and dense in $\mathfrak{X}_d$.
\end{remark}

\section{Sufficient conditions for structural stability}\label{Sec5}

In what follows, we recall that $\mathcal{P}_d\subset\mathfrak{X}_d$ denote the family of vector fields $X$ satisfying the following statements.
\begin{enumerate}
	\item[$(a')$] $p(X)$ have a finite number of singularities, all hyperbolic;
	\item[$(b')$] $p(X)$ have a finite number of closed orbits, all hyperbolic;
	\item[$(c)$] If $\gamma$ is a connection between saddles, then $\gamma\subset\Lambda$;
	\item[$(d')$] The $\alpha$ and $\omega$-limits of every orbit of $p(X)$ is either a singularity, a closed orbit or a generic polycycle contained in $\Lambda$.
\end{enumerate}

\begin{lemma}[Lemma~$1$ of \cite{PeiPei1959}]\label{Lemma9}
	Let $X\in\mathcal{P}_d$ and $p$ be a singularity of $p(X)$. There are neighborhoods $D\subset\mathbb{R}^2$ of $p$ and $N\subset\mathfrak{X}_d$ of $X$ such that for every $Y\in N$, there is only one singularity $p_Y$ of $p(Y)$ inside $D$. Moreover, $p$ and $p_Y$ is of the same type. Furthermore, if $p$ is not a saddle, then we can take $D$ such that its boundary has no contact with $p(X)$.
\end{lemma}

\begin{lemma}[Lemma~$2$ of \cite{PeiPei1959}]\label{Lemma10}
	Let $X\in\mathcal{P}_d$ and $\gamma$ be a limit cycle of $p(X)$. There are an open ring $S\subset\mathbb{R}^2$ bounded by two closed curves $C$ and $C'$, with $\gamma\subset S$, and a neighborhood $N\subset\mathfrak{X}_d$ of $X$ such that for every $Y\in N$, the following statements holds.
	\begin{enumerate}[label=(\alph*)]
		\item The segments of the normals to $\gamma$ contained in $S$ are disjoint and $p(Y)$ has no contact relatively to these segments and to $C$ and $C'$;
		\item $p(Y)$ has exactly one limit cycle $\gamma_Y$ inside $S$. Moreover, $\gamma_Y$ and $\gamma$ have the same stability.
	\end{enumerate}
\end{lemma}

\begin{definition}\label{Def1}
	Let $X\in\mathcal{P}_d$ and $p_j$, $j\in\{1\dots,n\}$, and $\gamma_k$, $k\in\{1\dots,m\}$ be the singularities and limit cycles of $p(X)$. The regions $D_j$ and $S_k$ given by Lemmas~\ref{Lemma9} and \ref{Lemma10} are called the critical regions associated with $p_j$ and $\gamma_k$.
\end{definition}

\begin{lemma}[Lemma~$3$ of \cite{PeiPei1959}]\label{Lemma11}
	Let $X\in\mathcal{P}_d$. There is a neighborhood $N\subset\mathfrak{X}_d$ of $X$ such that if $Y\in N$, then each singularity of $p(Y)$ is interior to some $D_j$, $j\in\{1,\dots,n\}$, and each limit cycle $\gamma$ is interior to some $S_k$, $k\in\{1,\dots,m\}$.
\end{lemma}

\begin{remark}
	We can choose the regions $D_j$ small enough such that if $p_j$ is a saddle, then the only separatrices which intersect $D_j$ are the four separatrices associated to $p_j$, except if $p_j$ lies in some generic polycycle $\Omega\subset\Lambda$.
\end{remark}

As discussed in the previous section, in our context we have to deal with the \emph{persistent} polycycle $\Omega\subset\Lambda$ given by the boundary of the square,
	\[\Delta=\{(x,y)\in\mathbb{R}^2\colon 0\leqslant x\leqslant1,\;0\leqslant y\leqslant 1\}.\]
To deal with this, in the next lemma we develop a tool similarly to Lemmas~\ref{Lemma9} and \ref{Lemma10}.

\begin{lemma}\label{Lemma12}
	Let $X\in\mathcal{P}_d$ be such that the boundary of the square
		\[\Delta=\{(x,y)\in\mathbb{R}^2\colon 0\leqslant x\leqslant1,\;0\leqslant y\leqslant 1\},\]
	is a generic polycycle of $\Omega$ of $X$. There are a neighborhood $N\subset\mathfrak{X}_d$ of $X$ and a closed curve $S\subset\text{Int}(\Delta)$ such that if $K$ is the bounded region limited by $\Omega$ and $S$, then for each $Y\in N$, the following statements holds.
	\begin{enumerate}[label=(\alph*)]
		\item There is no singularity of $Y$ in $K$, except by the hyperbolic saddles of $\Omega$;
		\item There is no limit cycle of $Y$ in $K$;
		\item $Y$ has no contact relatively to $S$;
		\item If $r(\Omega)>1$, then $Y(s)$ points towards $K$, for every $s\in S$;
		\item If $r(\Omega)<1$, then $Y(s)$ points in opposite direction to $K$, for every $s\in S$;
	\end{enumerate}
\end{lemma}

\begin{proof} Since $\Omega$ is generic, it follows that $r(\Omega)\neq1$. Suppose $r(\Omega)>1$. Then, there is an open ring $A$ containing $\Omega$ such that for every $x\in A\cap\text{Int}(\Delta)$, the orbit of $X$ through $x$ goes to $\Omega$. In special, there are no limit cycles or singularities of $X$ in $A\cap\text{Int}(\Delta)$. Therefore, we can take a sufficiently small neighborhood $N\subset\mathfrak{X}_d$ of $X$ such that for every $Y\in N$, there are no limit cycles or singularities of $Y$ in $A\cap\text{Int}(\Delta)$. Moreover, since $r(\Omega)$ depends continuously on the parameters of $X$, it follows that we can take $N$ small enough such that for every $Y\in N$ we also have $r(\Omega_Y)>1$. If $X=(P,Q)$ is given by \eqref{0}, then let $Z=(R,S)$ be given by
	\[R(x,y)=x(x-1)\bigl(f(x,y)-\varepsilon g(x,y)\bigr), \quad S(x,y)=y(y-1)\bigl(g(x,y)+\varepsilon f(x,y)\bigr).\]
If $\varepsilon>0$ is sufficiently small, then $Z\in N$ and thus $r(\Omega_Z)>1$. Observe that $Z$ is a small counterclockwise rotation of $X$. Since $r(\Omega_Z)>1$, it follows that there exists an open ring $A_Z$ containing $\Omega_Z$ such that for every $x\in A_Z\cap\text{Int}(\Delta)$, the orbit of $Z$ through $x$ goes to $\Omega_Z$. Let $\ell$ be a small transversal section to $\Omega_Z$ and let $\Omega_Z\cap\ell=\{p\}$. Let $\xi$ be a coordinate system over $\ell$ such that $\xi(p)=0$ and $\xi(q)>0$ if $q\in\ell\cap\text{Int}(\Delta)$. Let $P\colon\ell\to\ell$ be the return map associated to $\Omega_Z$. Given $q\in\ell\cap\text{Int}(\Delta)$, let $q_1=P(q)$ and observe that $\xi(q_1)<\xi(q)$. Let $\gamma$ be the orbit of $Z$ from $q$ to $q_1$. Let also $J\subset\ell$ be the segment between $q$ and $q_1$. See Figure~\ref{Fig4}.
\begin{figure}[h]
	\begin{center}
		\begin{overpic}[height=5cm]{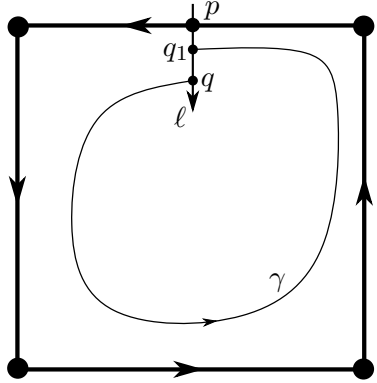} 
			\put(52,97){$p$}
			\put(51,78){$q$}
			\put(41,86){$q_1$}
			\put(44,67){$\ell$}
			\put(69,25){$\gamma$}
		\end{overpic}
	\end{center}
	\caption{An illustration of the closed curve $S$.}\label{Fig4}
\end{figure}
Let $S=\gamma\cup J$ and observe that $S$ is a continuous simple closed curve. Let also $K$ be the bounded region limited by $\Omega$ and $S$. Since $Z$ is a small counterclockwise rotation of $X$, it is clear that given any $s\in\gamma$, the orbit of $X$ through $s$ is transversal to $\gamma$ and points towards $K$. Moreover, given $s\in J$, is it s clear that the orbit of $X$ through $s$ is also transversal to $J$ and points towards $K$. Therefore, since $S$ is compact, it follows that there are a neighborhood $N\subset\mathfrak{X}_d$ of $X$ such that if $Y\in N$, then for any $s\in S$, the orbit of $Y$ through $s$ is transversal to $S$ and points towards $K$. The case $r(\Omega)<1$ follows by reversing the time variable. \end{proof}

\begin{remark}\label{Remark9}
	Similar to Definition~\ref{Def1}, if $\Omega$ is a generic polycycle of $X$, then we say that the region $K$ given by Lemma~\ref{Lemma12} is the critical region associated to $\Omega$.
\end{remark}

In the next Lemma we deal with the $\omega$ and $\alpha$-limits of the separatrices of the hyperbolic saddles of $p(X)$. As stated before, in this paper we also have to deal with some persistent saddle connections and polycycles. In particular, in the second part of the proof of the following lemma we deal with persistent saddles connection allowing hyperbolic saddles as structurally stable $\omega$ and $\alpha$-limits of other hyperbolic saddles. See Figure~\ref{Fig1}.

\begin{lemma}\label{Lemma13}
	Let $X\in\mathcal{P}_d$ and $p$ be a saddle of $p(X)$ to which the separatrices $\gamma_1$, $\gamma_3$ tend when $t\to\infty$ and $\gamma_2$, $\gamma_4$ when $t\to-\infty$. Let $\gamma_1$, $\gamma_3$ come from $\alpha_1$ and $\alpha_3$ and $\gamma_2$, $\gamma_4$ goes to $\omega_2$ and $\omega_4$. There is a neighborhood $N\subset\mathfrak{X}_d$ of $X$ such that if $Y\in N$, then the separatrices $(\gamma_1)_Y$, $(\gamma_3)_Y$ associated to the saddle $p_Y$ come from $(\alpha_1)_Y$ and $(\alpha_3)_Y$, while $(\gamma_2)_Y$, $(\gamma_4)_Y$ goes to $(\omega_2)_Y$ and $(\omega_4)_Y$, respectively.
\end{lemma}

\begin{proof} To simplify the approach, let $\gamma$ be a separatrix of $p$ such that $p$ is the $\alpha$-limit of $\gamma$. Let also $\omega_0$ be the $\omega$-limit of $\gamma$. Suppose first that $\omega_0$ is not a hyperbolic saddle. Hence, $\omega_0$ is a hyperbolic focus/node, a hyperbolic limit cycle or the generic polycycle $\Omega$ given by the boundary of the region,
	\[\Delta=\{(x,y)\in\mathbb{R}^2\colon 0\leqslant x\leqslant1,\;0\leqslant y\leqslant 1\}.\]
In any case, it follows from Lemmas~\ref{Lemma9}, \ref{Lemma10} and \ref{Lemma12} that there are neighborhoods $N\subset\mathfrak{X}_d$ of $X$ and $D\subset\mathbb{R}^2$ of $\omega_0$ such that $D$ is a well defined critical region associated to $\omega_0$ for every $Y\in N$. In particular, $(\omega_0)_Y\in D$ for every $Y\in N$. Restricting $N$ if necessary, it follows from Lemma~$4.3$ of \cite{Soto1974} that there is a neighborhood $W$ of $p$ such that for every $Y\in N$ the following statements holds.
\begin{enumerate}[label=(\alph*)]
	\item $Y$ has an unique singularity $p_Y\in W$, which is a hyperbolic saddle;
	\item The boundary $\partial W$ of $W$ is a differentiable curve;
	\item The stable (resp. unstable) separatrices of $p_Y$ intersect $\partial W$ in two points $s_1(Y)$ and $s_2(Y)$ (resp. $u_1(Y)$, $u_2(Y)$);
	\item The maps $s_i\colon N\to \partial W$ (resp. $u_i\colon N\to \partial W$) are of class $C^\infty$;
	\item There are closed arcs $S_i$ (resp. $U_i$) containing $s_i(Y)$ (resp. $u_i(Y))$ on which $Y$ is transversal to $\partial W$, for every $Y\in N$.
\end{enumerate}
Let $u_1(X)$ be the intersection of $\gamma$ and $W$. Since $\gamma$ goes to $\omega_0$, it follows that there is $q\in\gamma$ such that $q$ lies in the interior of $D$. Given $\varepsilon>0$, let $U_\varepsilon$ and $V_\varepsilon$ be the $\varepsilon$-neighborhoods of $u_1(X)$ and $q$. Let $\varepsilon>0$ be small enough such that $V_\varepsilon\subset D$. See Figure~\ref{Fig5}.
\begin{figure}[h]
	\begin{center}
		\begin{overpic}[height=5cm]{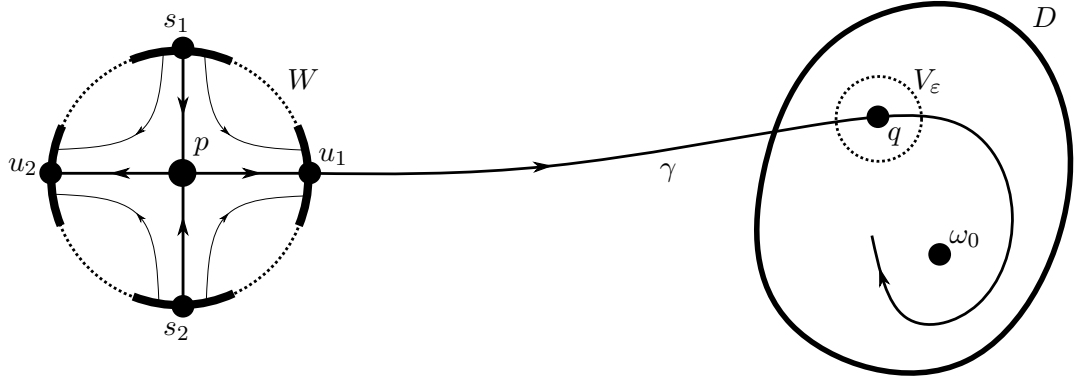} 
			\put(27,21){$u_1$}
			\put(-3,20){$u_2$}
			\put(12,34){$s_1$}
			\put(12,4){$s_2$}
			\put(15,22){$p$}
			\put(24,28){$W$}
			\put(96,34){$D$}
			\put(60,19.5){$\gamma$}
			\put(82,23){$q$}
			\put(84.5,28){$V_\varepsilon$}
			\put(88,13){$\omega_0$}
		\end{overpic}
	\end{center}
	\caption{An illustration of $\gamma$.}\label{Fig5}
\end{figure}
Since $u_1\colon N\to\partial W$ is of class $C^\infty$, it follows that restricting $N$ if necessary, we can suppose $u_1(N)\subset U_\varepsilon$. Let $\gamma(t)$ be the parametrization of $\gamma$ given by the flow of $X$ and such that $\gamma(0)=u_1(X)$. Let $t_0>0$ be such that $\gamma(t_0)=q$. Let $\gamma_Y(t)$ be the perturbation of $\gamma(t)$ associated to $Y$, i.e. $\gamma_Y(t)$ is the orbit of $Y$ passing through $u_1(Y)$ at $t=0$. Since $u_1(Y)\in U_\varepsilon$, it follows from the continuous dependence of initial conditions (see Theorem~$8$, p. 25 of \cite{And1971}) that $\gamma_Y(t_0)\in V_\varepsilon$. Hence, $\gamma_Y$ enters the critical region $D$ and thus goes to $(\omega_0)_Y$. 

Suppose now that $\omega_0=\rho$ is also a hyperbolic saddle. It follows from the definition of $\mathcal{P}_d$ that $\gamma\subset\Lambda$. In special $p$, $\rho\in\Lambda$, $p\neq\rho$ and $\gamma$ is the segment between $p$ and $q$. See Figure~\ref{Fig1}.
\begin{figure}[h]
	\begin{center}
		\begin{minipage}{4cm}
			\begin{center}
				\begin{overpic}[height=3cm]{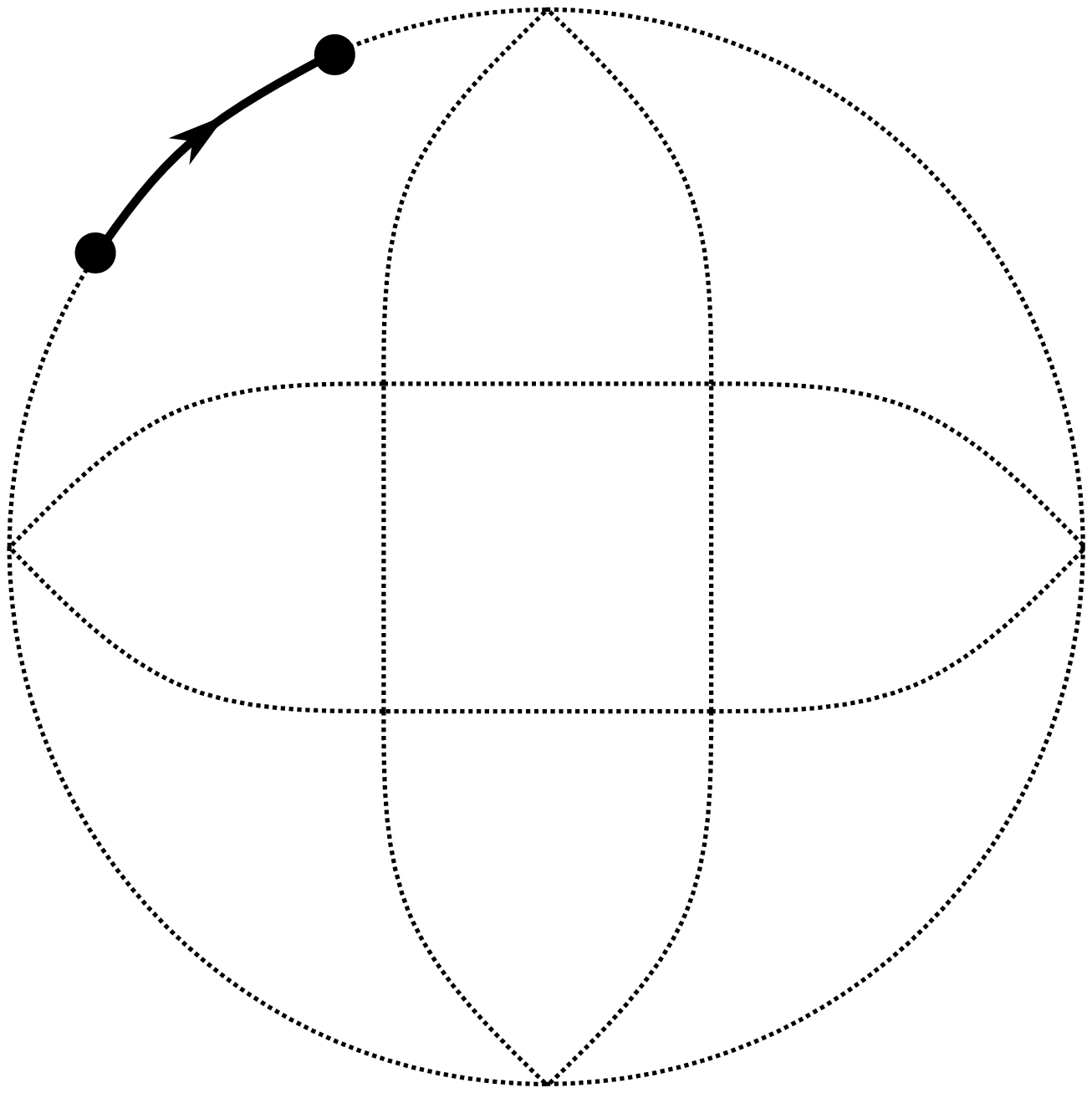} 
					\put(0,80){$p$}
					\put(25,100){$\rho$}
				\end{overpic}
			\end{center}
		\end{minipage}
		\begin{minipage}{4cm}
			\begin{center}
				\begin{overpic}[height=3cm]{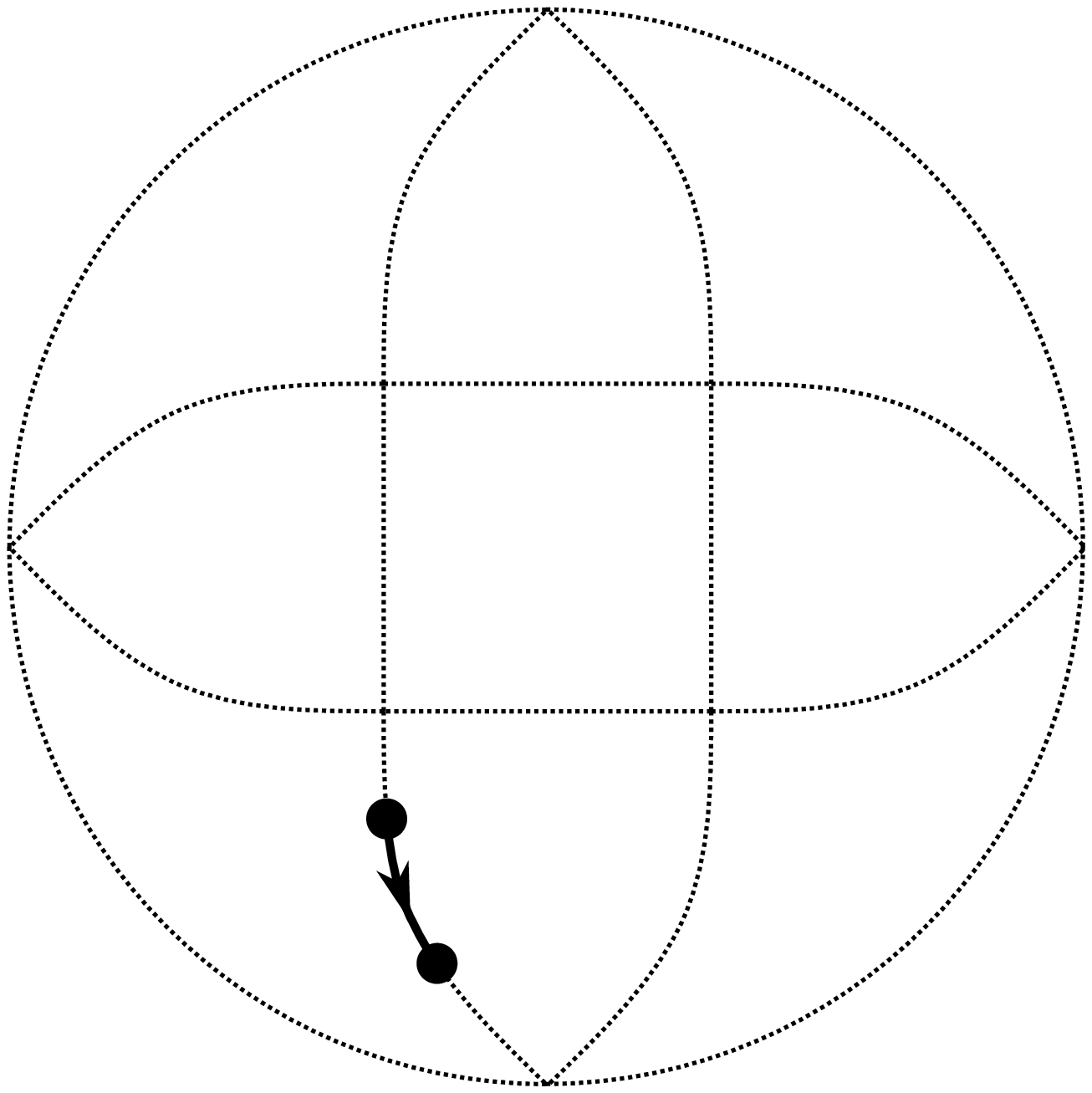} 
					\put(26,25){$p$}
					\put(29,8){$\rho$}
				\end{overpic}
			\end{center}
		\end{minipage}
		\begin{minipage}{4cm}
			\begin{center}
				\begin{overpic}[height=3cm]{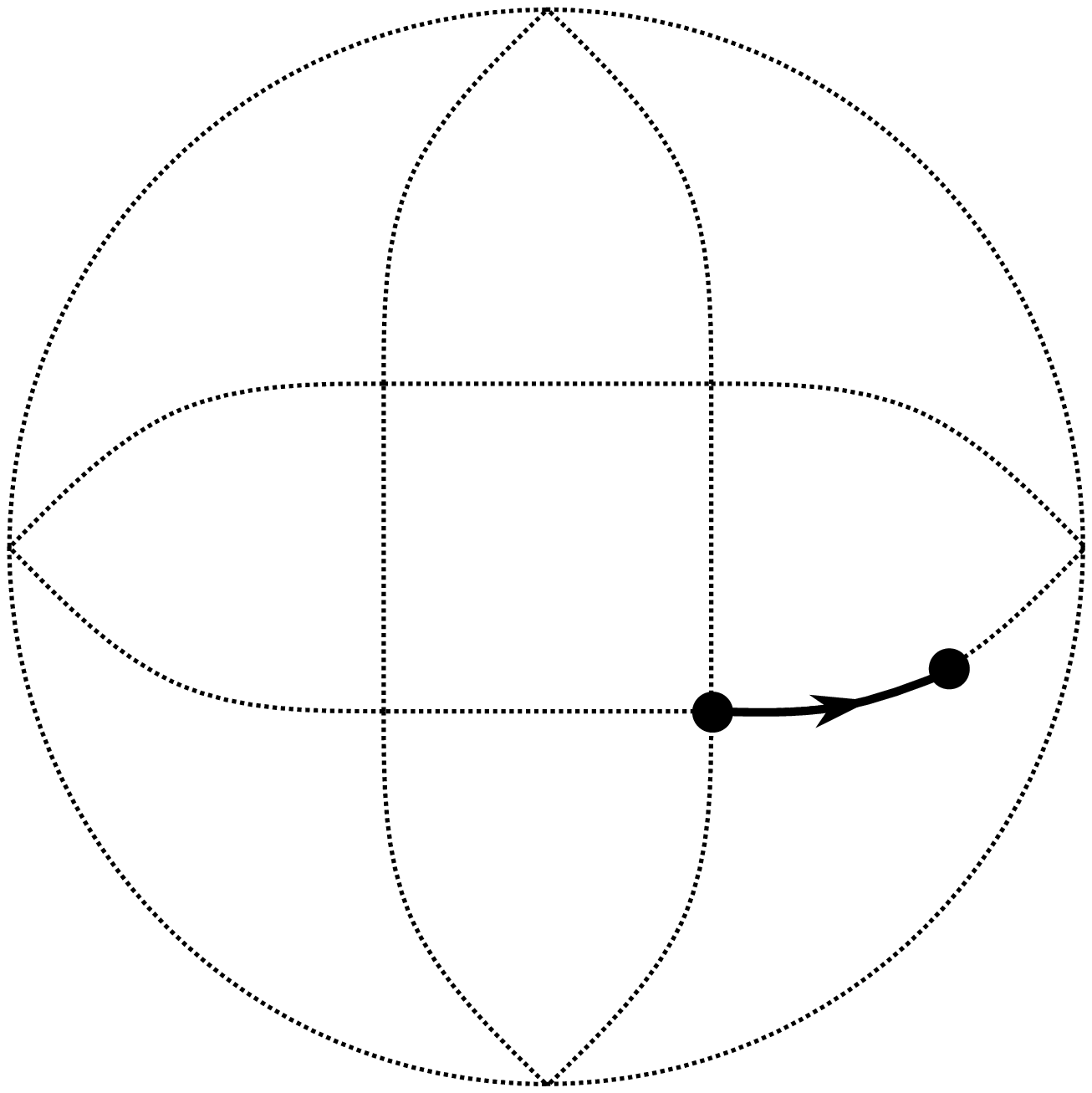} 
					\put(57,40){$p$}
					\put(85,30){$\rho$}
				\end{overpic}
			\end{center}
		\end{minipage}
		\begin{minipage}{4cm}
			\begin{center}
				\begin{overpic}[height=3cm]{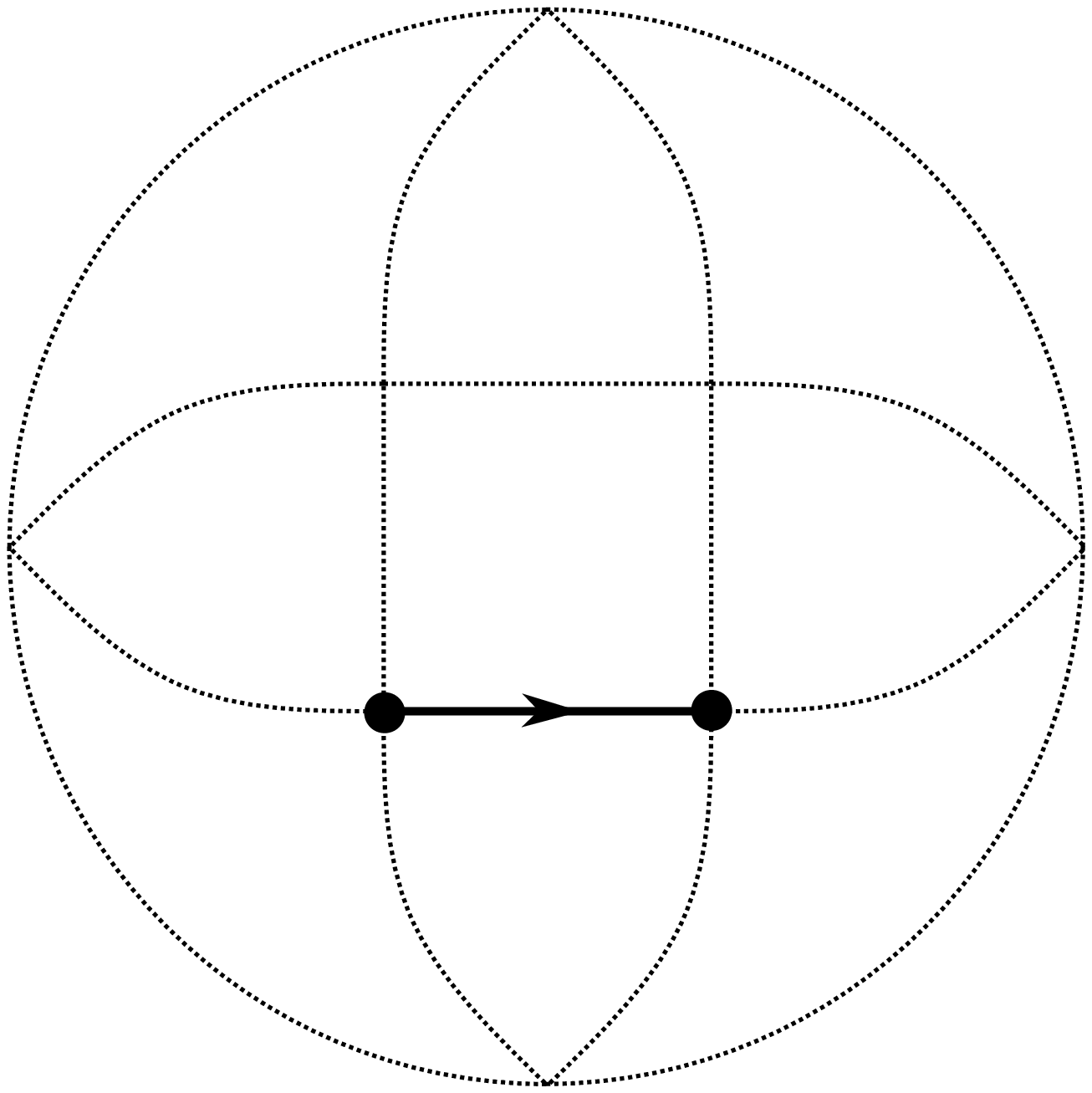} 
					\put(25,27){$p$}
					\put(57,40){$\rho$}
				\end{overpic}
			\end{center}
		\end{minipage}
	\end{center}
\caption{Some examples of the positions of $p$ and $\rho$. Remember that the north, south, east and west poles (i.e. the origins of the charts of the Poincar\'e compactification) are not hyperbolic saddles.}\label{Fig1}
\end{figure}
We claim that $p_Y$, $\rho_Y\in\Lambda$, for every $Y\in N$ (restricting $N$ if necessary). Suppose first that $p$ is a finite singularity. Hence, $p$ lies in one of the invariant straight lines $l_1,l_2,l_3,l_4$ given by $x=0$,  $x=1$, $y=0$, $y=0$. Let
	\[p_1=(0,0), \quad p_2=(1,0), \quad p_3=(1,1), \quad p_4=(0,1).\]
If $p=p_j$ for some $j\in\{1,2,3,4\}$, then it is clear $p_Y=p_j$ and thus $p_Y\in\Lambda$, $Y\in N$. Hence, suppose that $p\in l_i$, for some $i\in\{1,2,3,4\}$ and $p\neq p_j$, for every $j\in\{1,2,3,4\}$. Suppose for example that $p\in l_1$. That is, $p=(0,y_0)$, with $y_0\not\in\{0,1\}$. If $X=(P,Q)$ is given by \eqref{0}, then $g(p)=0$ and thus,
\begin{equation}\label{28}
	DX(p)=\left(\begin{array}{cc} -f(p) & 0 \vspace{0.2cm} \\ \star & \displaystyle y_0(y_0-1)\frac{\partial g}{\partial y}(p) \end{array}\right).
\end{equation}
Since $p$ is hyperbolic, it follows from \eqref{28} that $\frac{\partial g}{\partial y}(p)\neq0$. Hence, the curve $g(x,y)=0$ is transversal to $l_1$ at $p$ and thus there is a neighborhood $N$ of $X$ such that $p_Y\in l_1\subset\Lambda$, for every $Y\in N$. The case in which $p$ is a singularity at infinity follows similarly and thus the claim is proved. Therefore, $\gamma_Y\subset\Lambda$ is the segment between $p_Y$ and $\rho_Y$. In special, $\gamma_Y\to\rho_Y$ when $t\to+\infty$. The proof for the others separatrices of $p$ follows similarly. \end{proof}

\section{Proofs of the main theorems}\label{Sec6}

\noindent\textit{Proof of Theorem~\ref{Main3}:} First, observe that statement $(a)$ follows directly from Remark~\ref{Remark5}. Let us look at statement $(b)$. Given $X\in\mathcal{P}_d$, let $N\subset\mathfrak{X}_d$ be a small neighborhood of $X$ such that Lemmas~\ref{Lemma9}-\ref{Lemma13} hold. In special, observe that $N\subset\mathcal{P}_d$. Given $Y\in N$ and following the classical techniques of M. C. Peixoto and M. M. Peixoto \cite{PeiPei1959}, we can construct an homeomorphism sending each critical and canonical region of $p(X)$ to its respective region of $p(Y)$. Moreover, these local homeomorphims glue up in a global homeomorphism, creating a global homeomorphism between the orbits of $p(X)$ and $p(Y)$, preserving the equator $\mathbb{S}^1$ of the Poincar\'e sphere. In our context new types of canonical regions arises, but using also the techniques of Sotomayor \cite{Soto1974}, it follows that such homeomorphisms can be adapted to these new canonical regions. Moreover, restricting $N$ if necessary, it also follows that these family of homeomorphism is $\varepsilon$-close (in the $C^0$-way) of the identity $\text{Id}_{\mathbb{S}^2}$. For the construction of the homeomorphisms, see Appendix~\ref{Sec7}. {\hfill$\square$}

\noindent\textit{Proof of Theorem~\ref{Main2}:} Follows directly from Lemmas~\ref{Lemma1}-\ref{Lemma0}. {\hfill$\square$}

\noindent\textit{Proof of Theorem~\ref{Main1}:} Let $X\in\Sigma_d$. It follows from Theorem~\ref{Main3} that $\mathcal{P}_d$ is dense at $\mathfrak{X}_d$. Therefore, there is $Y\in\mathcal{P}_d$ such that $X$ is topologically equivalent to $Y$. Hence, $X$ satisfies all the statements of the Theorem~\ref{Main1}. {\hfill$\square$}

\noindent\textit{Proof of Theorem~\ref{Main4}:} Let $X\in\mathcal{P}_d$ and suppose that $X$ has a global first integral on $\mathbb{R}^2$. Thus, $X$ cannot have any limit cycle, source or sink. In special, it follows from Theorem~\ref{Main1} that every singularity of $X$ is topologically equivalent to a saddle. Consider the square
	\[\Delta=\{(x,y)\in\mathbb{R}^2\colon 0\leqslant x\leqslant 1, \; 0\leqslant y\leqslant 1\},\]
and let $p\in\text{Int}(\Delta)$. It follows from Poincar\'e-Bendixson Theorem that the orbit $\gamma$ through $p$ has an $\omega$ and $\alpha$-limit sets. Moreover, it follows from the first integral that such limit sets are necessarily saddle points. Thus, $\gamma\not\subset\Lambda$ is a connection between saddles, contradicting Theorem~\ref{Main1}. {\hfill$\square$}

\noindent\textit{Proof of Theorem~\ref{Main5}:} Let $X=(P,Q)$ be given by \eqref{0}. Suppose by contradiction that $\gamma$ is a non-hyperbolic algebraic limit cycle for $X$ (see Section~\ref{sub6}), with period $T>0$ and parametrization $\gamma(t)$. Let $F\colon\mathbb{R}^2\to\mathbb{R}$ be the algebraic invariant curve associated to $\gamma$. That is, $F$ is polynomial, it satisfies \eqref{24} and $\gamma\subset F^{-1}(0)$. Given $\varepsilon>0$ and $\delta_1$, $\delta_2\in\{-1,1\}$, let $Y=(R,S)$ be given by,
	\[\begin{array}{l}
		\displaystyle R(x,y)=x(x-1)\left(f(x,y)+\varepsilon\delta_1F(x,y)\frac{\partial F}{\partial x}(x,y)\right), \vspace{0.2cm} \\
		\displaystyle S(x,y)=y(y-1)\left(g(x,y)+\varepsilon\delta_2F(x,y)\frac{\partial F}{\partial y}(x,y)\right).
	\end{array}\]
Since $\gamma\subset F^{-1}(0)$, it follows that $\gamma$ is also a limit cycle for $Y$, with same period and same parametrization. Let,
	\[H(x,y)=x(x-1)\delta_1\left(\frac{\partial F}{\partial x}(x,y)\right)^2+y(y-1)\delta_2\left(\frac{\partial F}{\partial y}(x,y)\right)^2.\]
Observe that,
\begin{equation}\label{25}
	r(\gamma_Y)=\int_{0}^{T}\frac{\partial R}{\partial x}(\gamma(t))+\frac{\partial S}{\partial y}(\gamma(t))\;dt=\int_{0}^{T}\frac{\partial P}{\partial x}(\gamma(t))+\frac{\partial Q}{\partial y}(\gamma(t))\;dt+\varepsilon\int_{0}^{T} H(\gamma(t))\;dt.
\end{equation}
Since $\gamma$ is a non-hyperbolic limit cycle of $X$, it follows that $r(\gamma_X)=0$ and thus,
	\[\int_{0}^{T}\frac{\partial P}{\partial x}(\gamma(t))+\frac{\partial Q}{\partial y}(\gamma(t))\;dt=0.\]
Therefore, it follows from \eqref{25} that,
\begin{equation}\label{26}
	r(\gamma_Y)=\varepsilon\int_{0}^{T} H(\gamma(t))\;dt.
\end{equation}
Similarly to the proofs of Lemmas~\ref{Lemma4} and \ref{Lemma6}, observe that $\gamma\cap\Lambda=\emptyset$ and thus $H(\gamma(t))$ has constant sign. Moreover, observe that we can choose $\delta_1$, $\delta_2\in\{-1,1\}$ in such way that $H(\gamma(t))$ is positive or negative. Therefore, it follows similarly to the proofs of Lemmas~\ref{Lemma7} and \ref{Lemma8} that this leads to the bifurcation of a new limit cycle. Hence, $\gamma$ is not structurally stable. {\hfill$\square$}

\appendix

\section{The construction of the homeomorphisms}\label{Sec7}

In this section, we first present the techniques of M. C. Peixoto and M. M. Peixoto \cite{PeiPei1959} and Sotomayor \cite{Soto1974}, and then we present our adaptions of their techniques. Given $X\in\mathcal{P}_d$, let $N\subset\mathfrak{X}_d$ be a small neighborhood of $X$ such that Lemmas~\ref{Lemma9}-\ref{Lemma13} hold. Restricting $N$ if necessary, observe that we can assume $N\subset\mathcal{P}_d$. Given $Y\in N$, let $S$ and $S_Y$ be the set of separatrices (see Section~\ref{sub3}) of $X$ and $Y$. It follows from Lemmas~\ref{Lemma9}-\ref{Lemma13} that there is a bijection $\psi\colon S\to S_Y$ such that the following statements hold.
\begin{enumerate}[label=(\roman*)]
	\item Corresponding separatrices are of the same type;
	\item A subset of $S$ bounds a canonical region of $p(X)$ if, and only if, the corresponding subset of $S_Y$ bounds a canonical region of $Y$.
\end{enumerate}
Let $A_1,\dots A_k$ and $B_1,\dots, B_l$ be the canonical and critical regions of $p(X)$ (recall Definition~\ref{Def1}, Remark~\ref{Remark9} and Section~\ref{sub3}) and $C_1,\dots,C_k$ and $D_1,\dots,D_l$ be the respective canonical and critical regions of $p(Y)$ (i.e. $C_i=\psi(A_i)$ and $D_i=\psi(B_i)$). The technique of M. C. Peixoto and M. M. Peixoto \cite{PeiPei1959} work as follows. Given a canonical region $A_i$ of $X$, let $B_\alpha$ and $B_\omega$ be the critical regions given by the source and the sink of $A_i$ (see Section~\ref{sub3}) and let $E_i=A_i\backslash\{B_\alpha\cup B_\omega\}$. We recall that each canonical region has exactly one source and one sink (recall Proposition~\ref{Prop0}), i.e. one \emph{object} as source and one object as sink, where by object we mean a singularity, a limit cycle or, as allowed in our context, a hyperbolic saddle or a polycycle. Let also $D_\alpha$ and $D_\beta$ be the critical regions given by the source and the sink of the canonical region $C_i=\psi(A_i)$ and consider $F_i=C_i\backslash\{D_\alpha\cup D_\omega\}$. The regions $E_i$ and $F_i$ are called \emph{generic regions}. See Figures~\ref{Fig6} and \ref{Fig7}.
\begin{figure}[h]
	\begin{center}
		\begin{minipage}{3cm}
			\begin{center}
				\begin{overpic}[height=2cm]{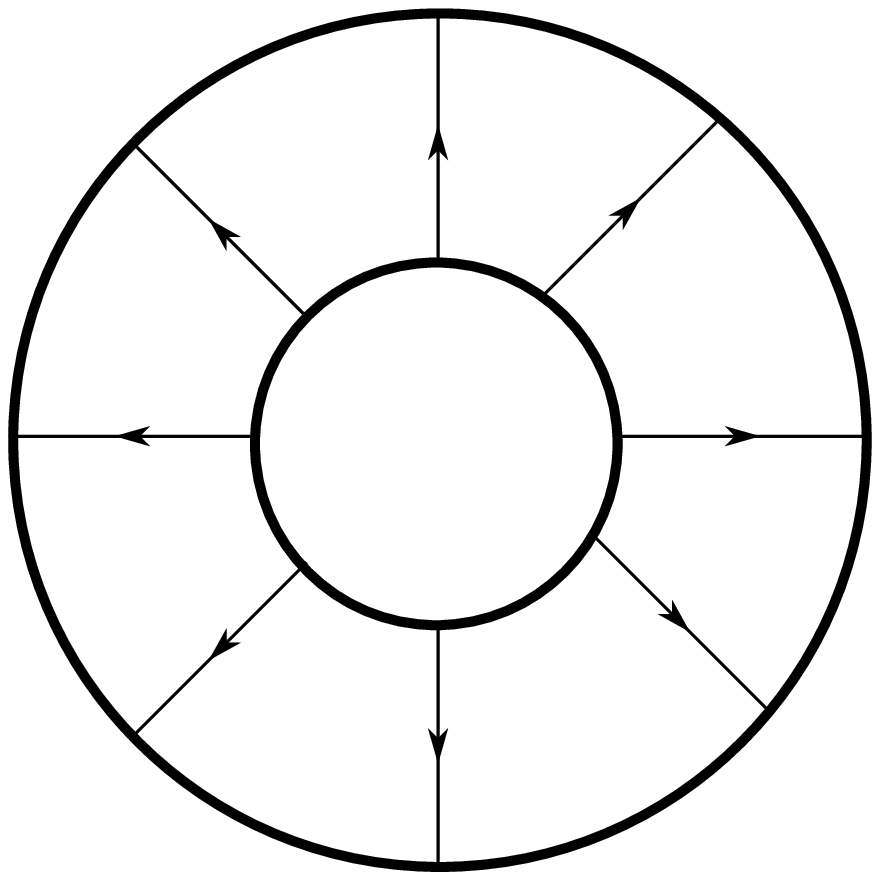} 
					\put(45,45){$\alpha$}
					\put(88,88){$\omega$}
				\end{overpic}
				
				Type $1$.
			\end{center}
		\end{minipage}
		\begin{minipage}{3cm}
			\begin{center}
				\begin{overpic}[height=2cm]{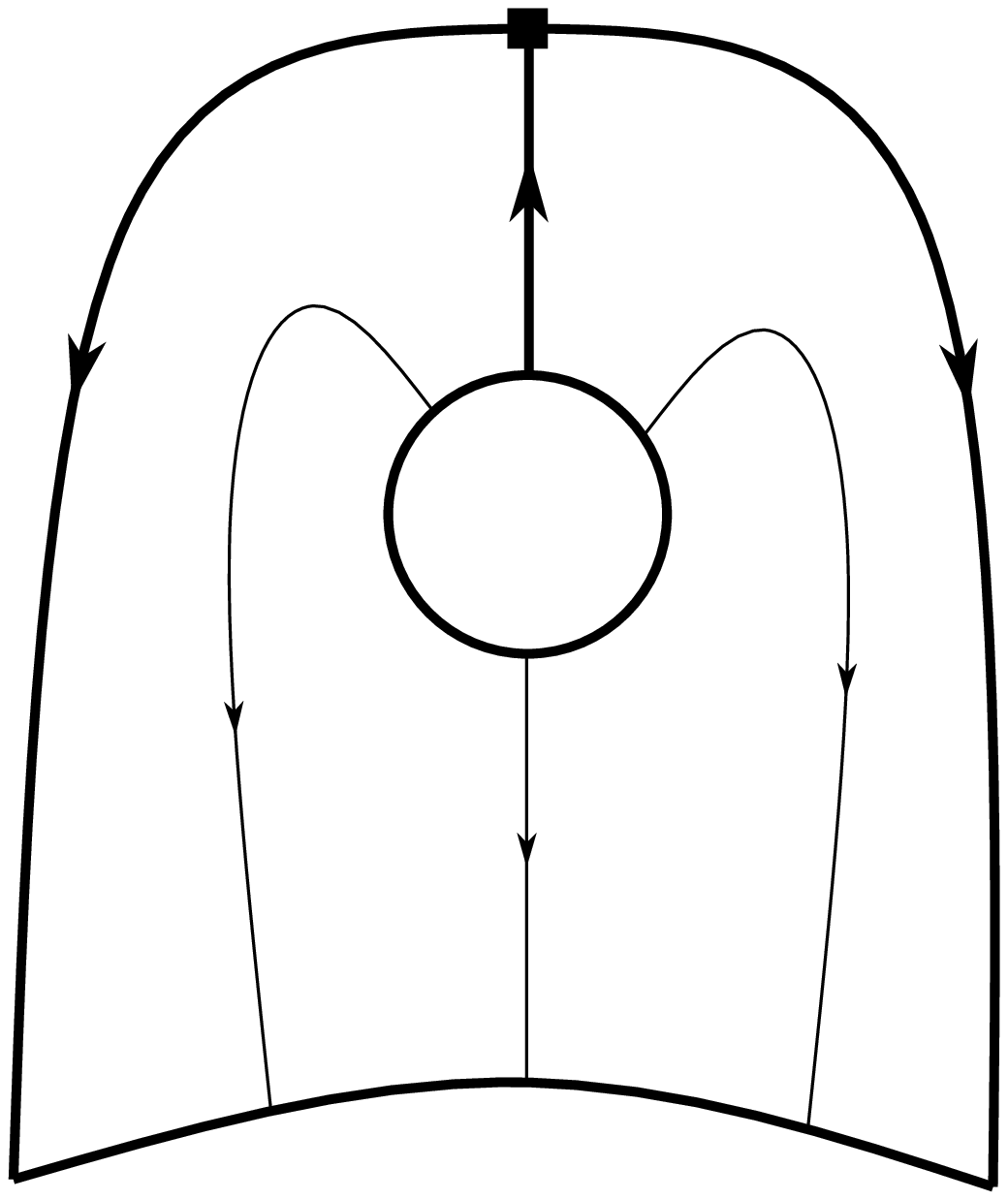} 
					\put(38,54){$\alpha$}
					\put(38,-3){$\omega$}
				\end{overpic}
				
				Type $2$.
			\end{center}
		\end{minipage}
		\begin{minipage}{3cm}
			\begin{center}
				\begin{overpic}[height=2cm]{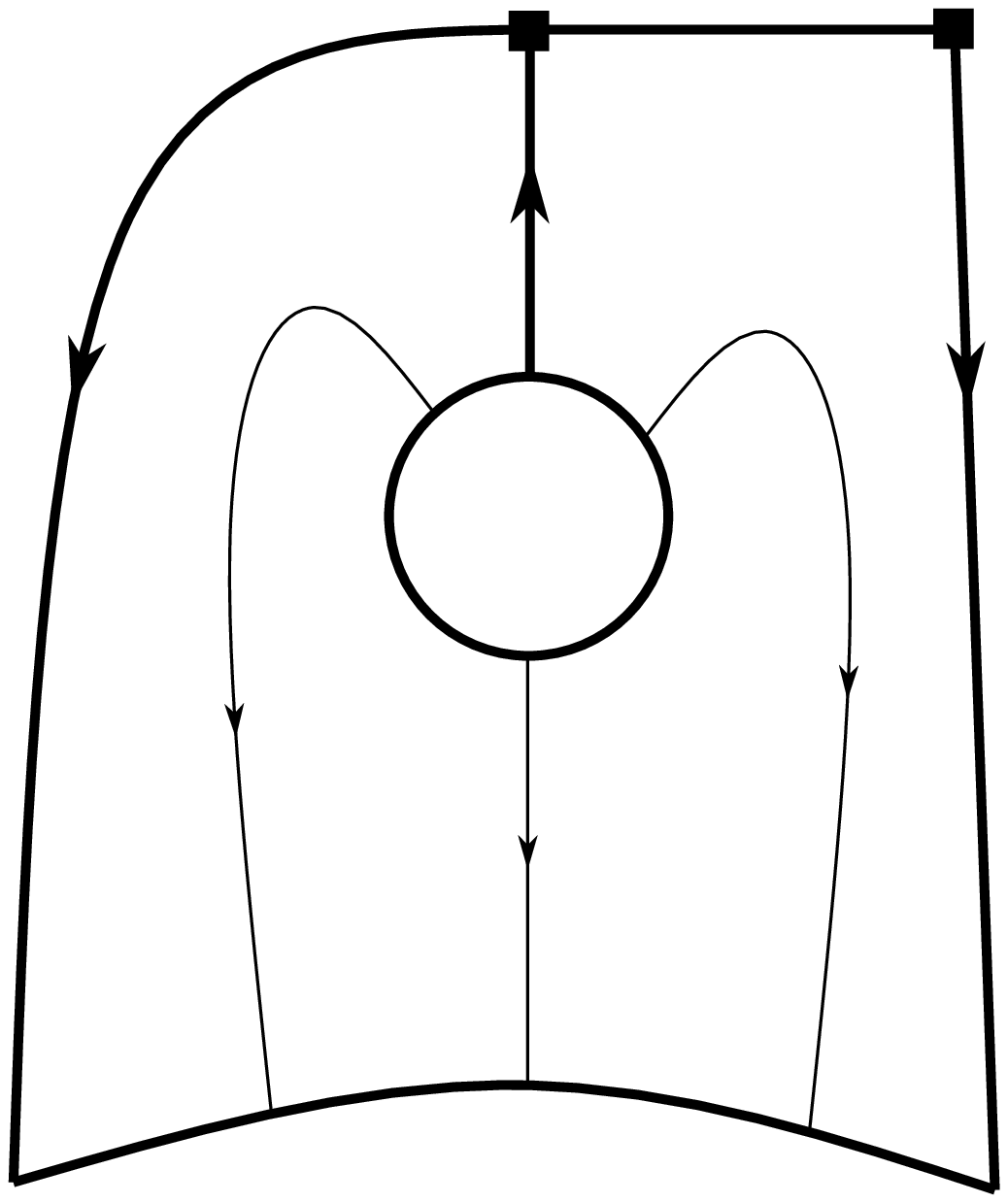} 
						\put(38,54){$\alpha$}
						\put(38,-3){$\omega$}
				\end{overpic}
				
				Type $3$.
			\end{center}
		\end{minipage}
		\begin{minipage}{3cm}
			\begin{center}
				\begin{overpic}[height=2cm]{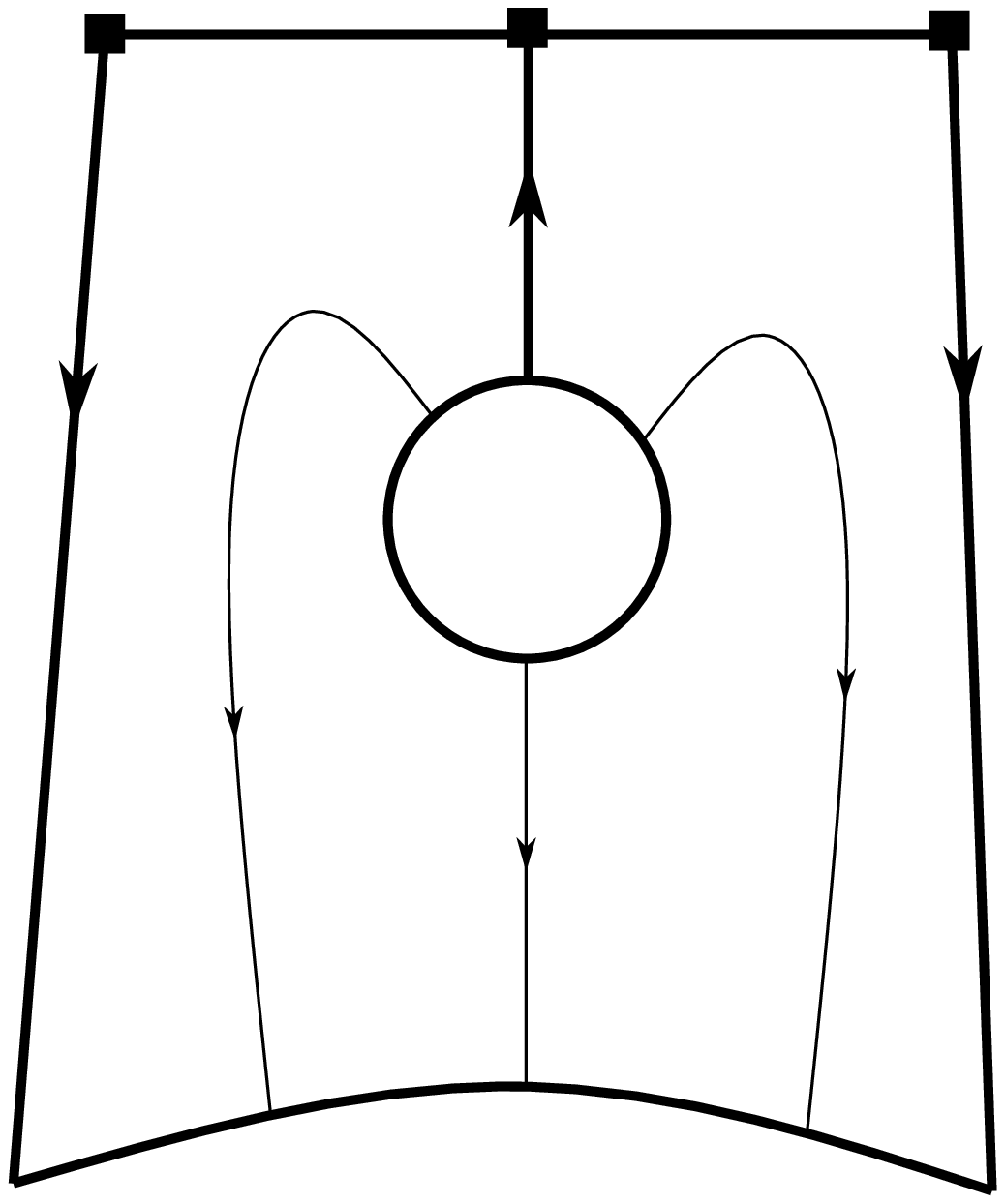} 
						\put(38,54){$\alpha$}
						\put(38,-3){$\omega$}
				\end{overpic}
				
				Type $4$.
			\end{center}
		\end{minipage}
		\begin{minipage}{3cm}
			\begin{center}
				\begin{overpic}[height=2cm]{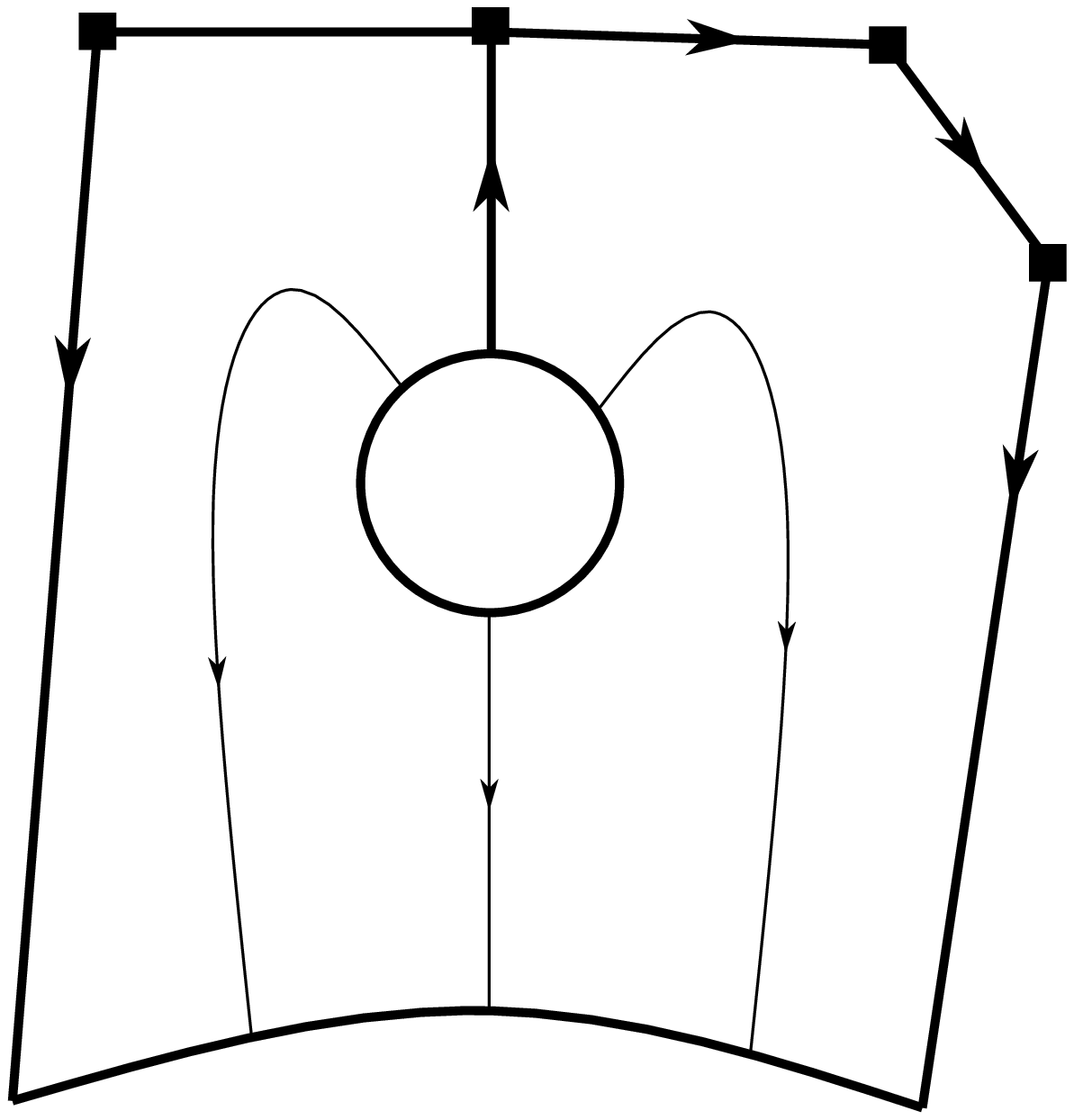} 
					\put(38,54){$\alpha$}
					\put(38,-3){$\omega$}
				\end{overpic}
				
				Type $5$.
			\end{center}
		\end{minipage}
	\end{center}
$\;$
	\begin{center}
		\begin{minipage}{4cm}
			\begin{center}
				\begin{overpic}[height=2cm]{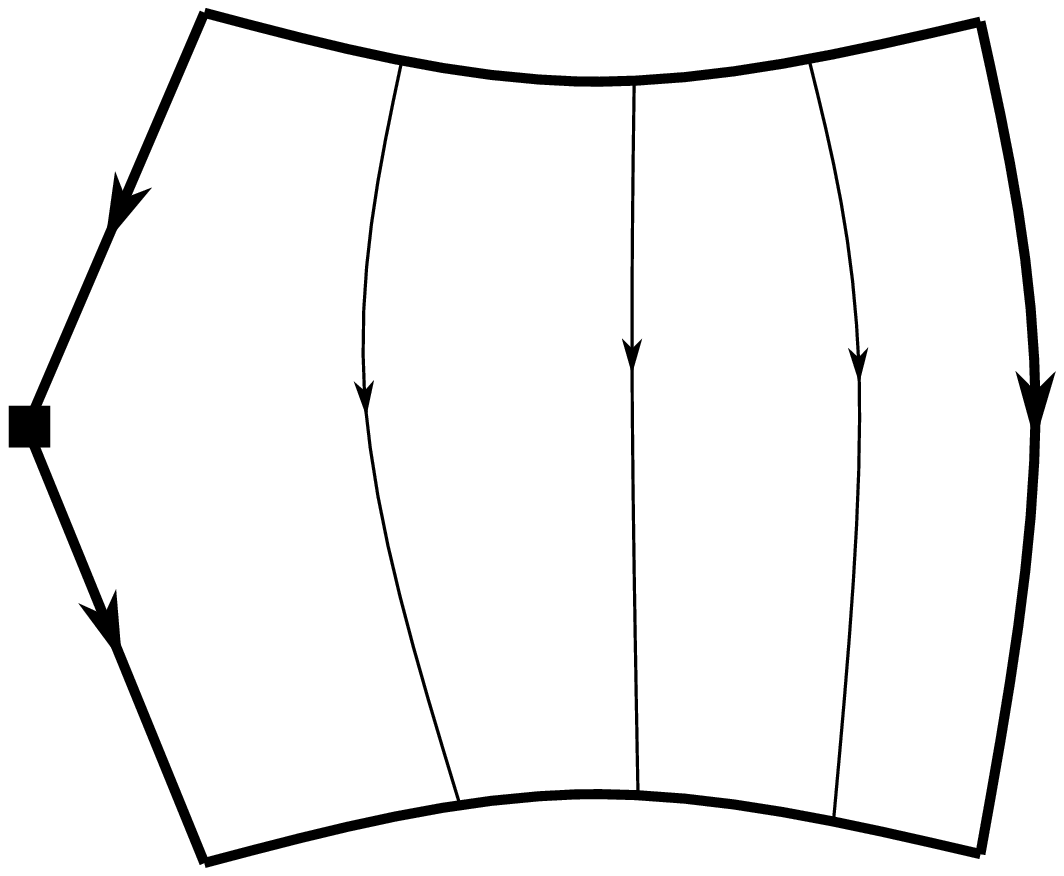} 
						\put(53,78){$\alpha$}
						\put(53,-3){$\omega$}
				\end{overpic}
				
				Type $6$.
			\end{center}
		\end{minipage}
		\begin{minipage}{4cm}
			\begin{center}
				\begin{overpic}[height=2cm]{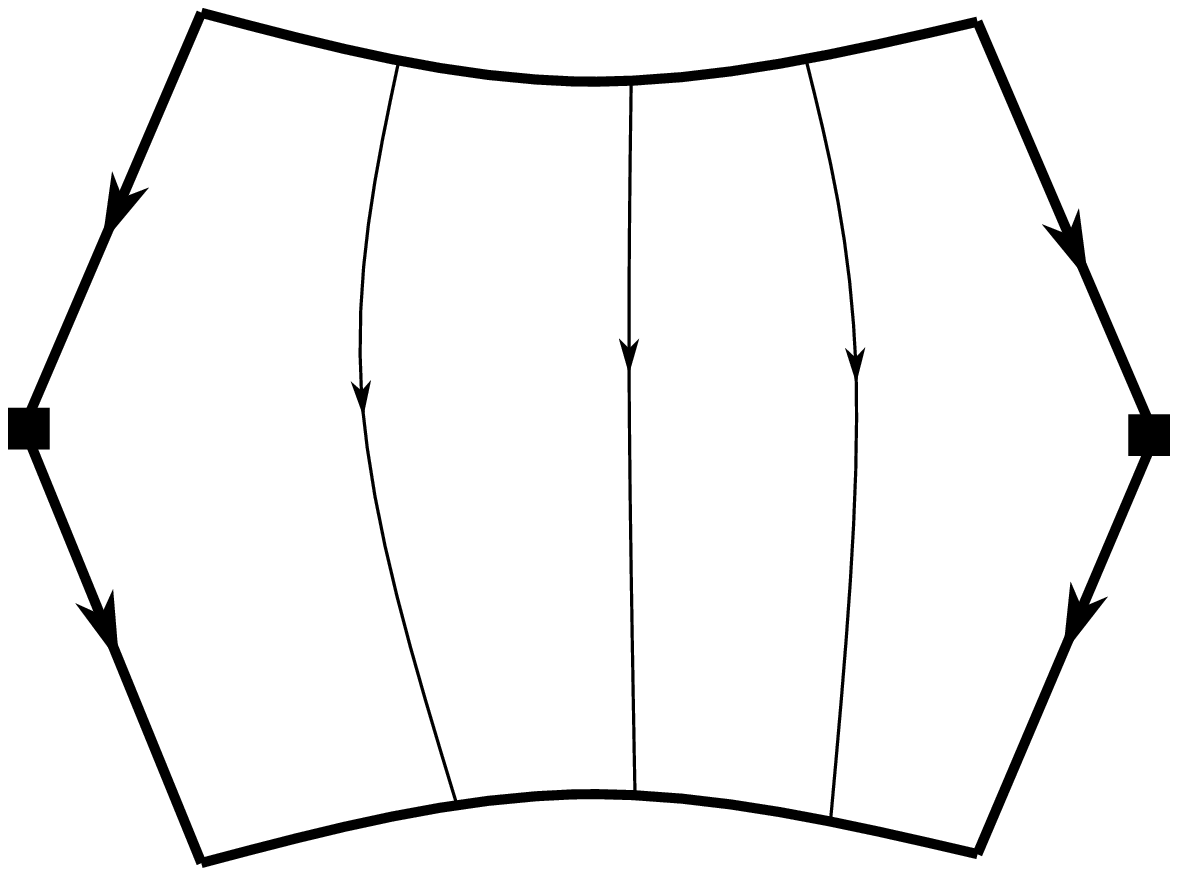} 
						\put(50,70){$\alpha$}
						\put(50,-3){$\omega$}
				\end{overpic}
				
				Type $7$.
			\end{center}
		\end{minipage}
		\begin{minipage}{4cm}
			\begin{center}
				\begin{overpic}[height=2cm]{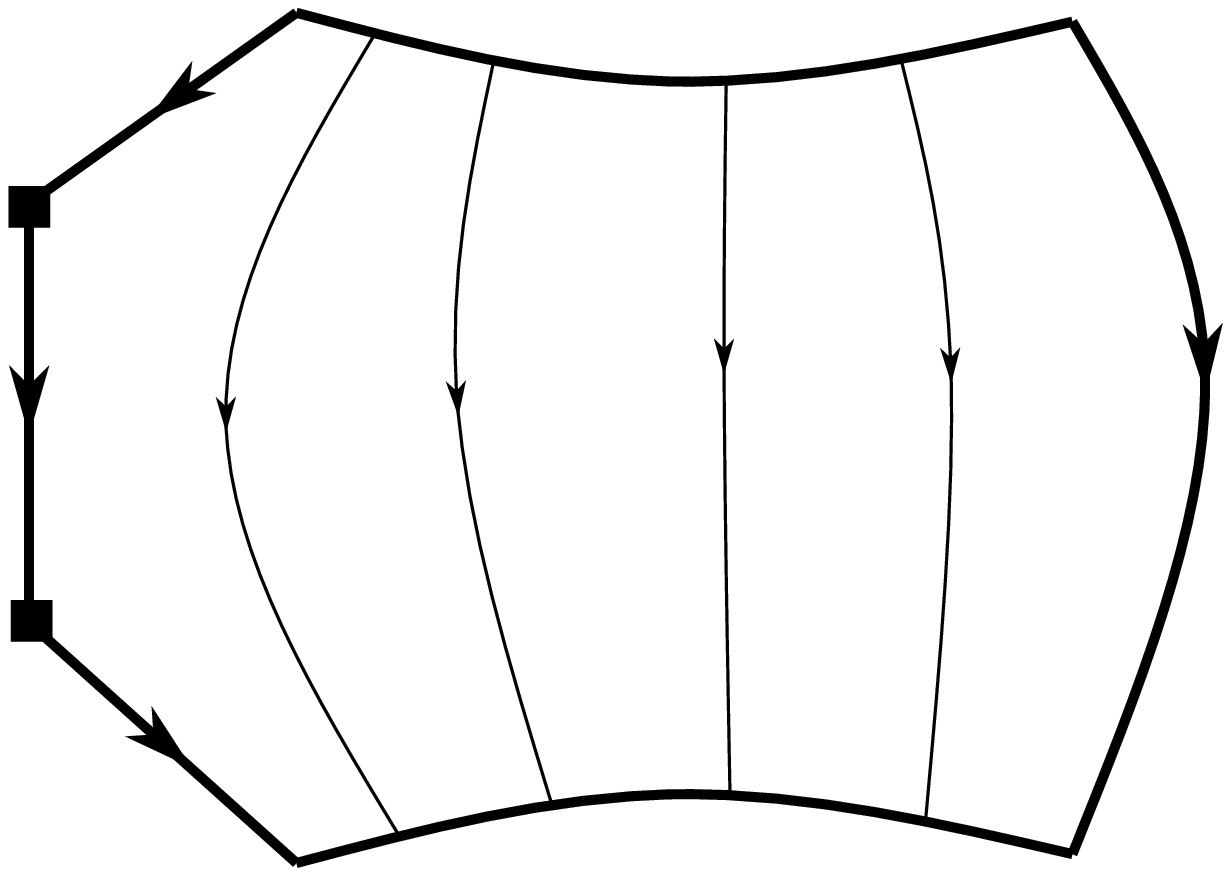} 
						\put(50,68){$\alpha$}
						\put(50,-3){$\omega$}
				\end{overpic}
				
				Type $8$.
			\end{center}
		\end{minipage}
		\begin{minipage}{4cm}
			\begin{center}
				\begin{overpic}[height=2cm]{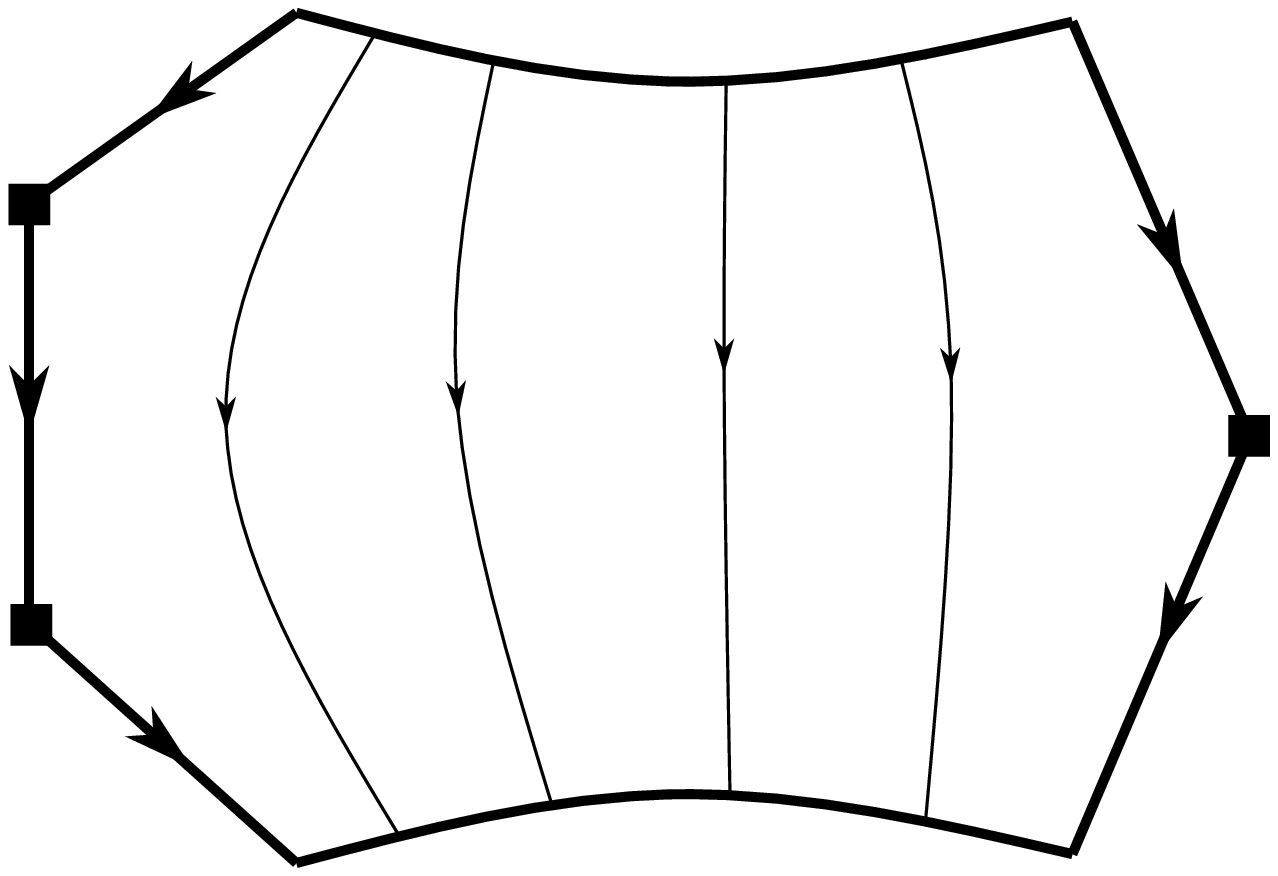} 
					\put(50,65){$\alpha$}
					\put(50,-3){$\omega$}
				\end{overpic}
				
				Type $9$.
			\end{center}
		\end{minipage}
	\end{center}
$\;$
	\begin{center}
		\begin{minipage}{4cm}
			\begin{center}
				\begin{overpic}[height=2cm]{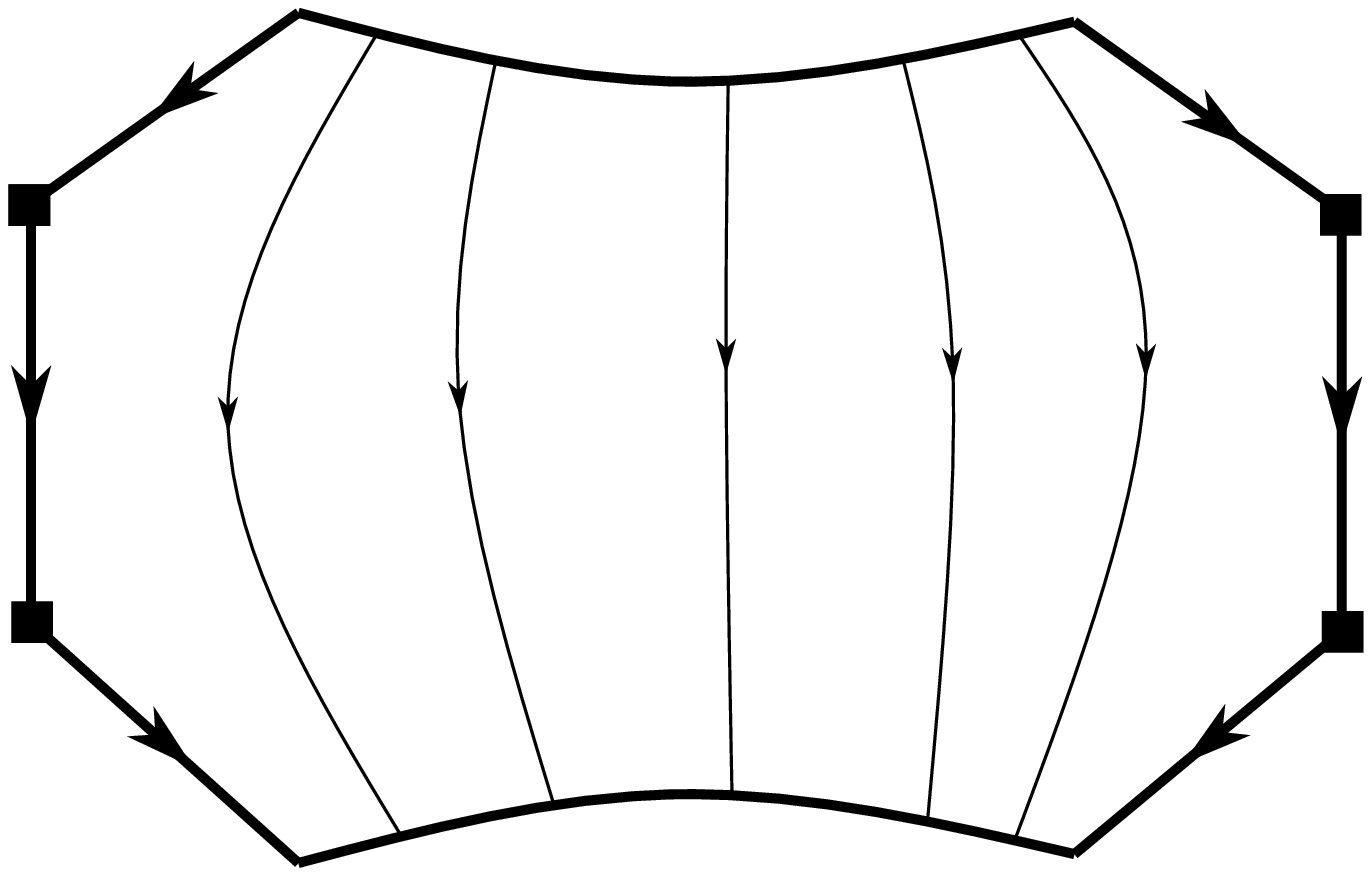} 
					\put(50,60){$\alpha$}
					\put(50,-3){$\omega$}
				\end{overpic}
				
				Type $10$.
			\end{center}
		\end{minipage}
		\begin{minipage}{4cm}
			\begin{center}
				\begin{overpic}[height=2cm]{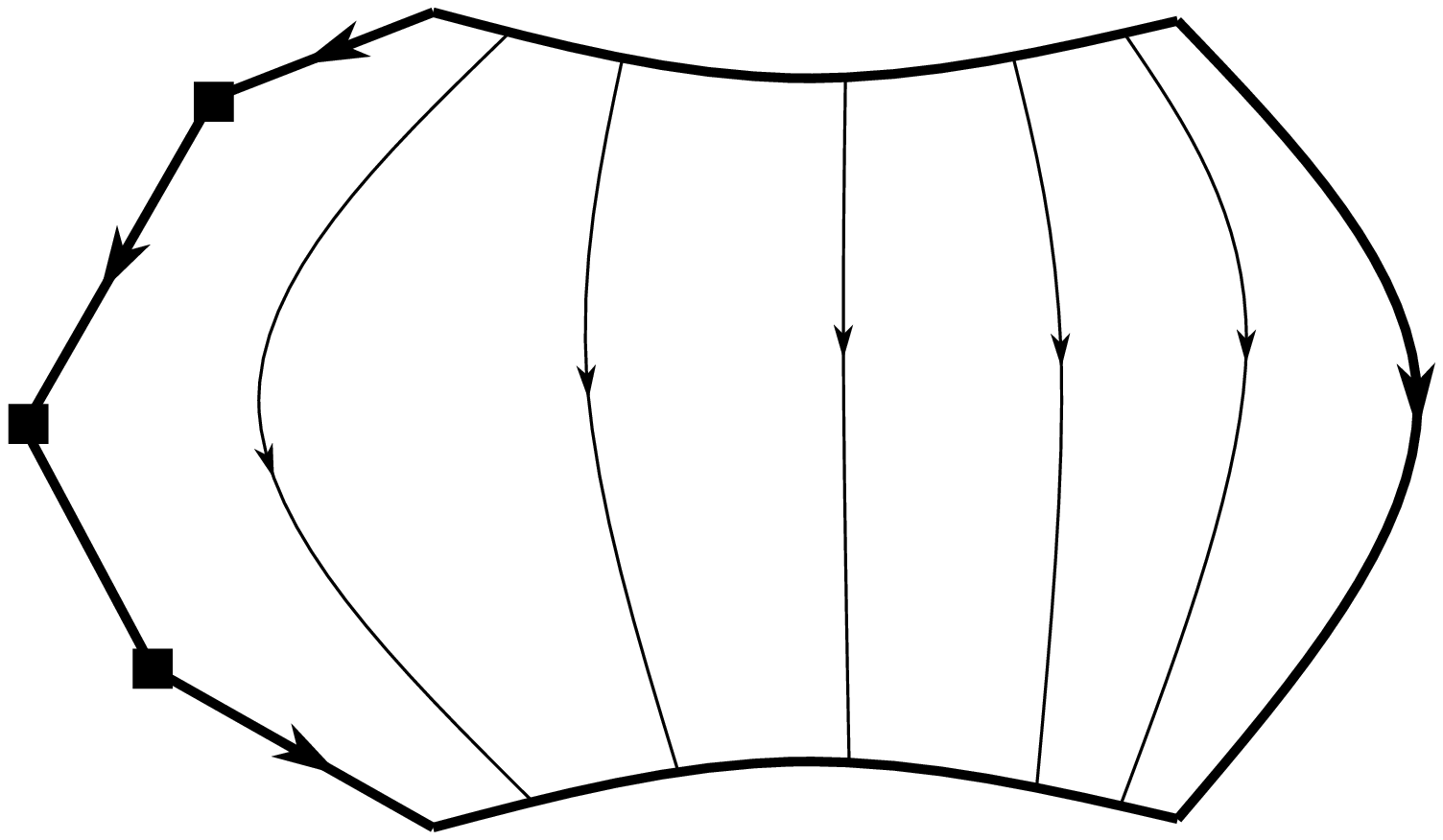} 
						\put(50,55){$\alpha$}
						\put(50,-3){$\omega$}
				\end{overpic}
			
				Type $11$.
			\end{center}
		\end{minipage}
		\begin{minipage}{4cm}
			\begin{center}
				\begin{overpic}[height=2cm]{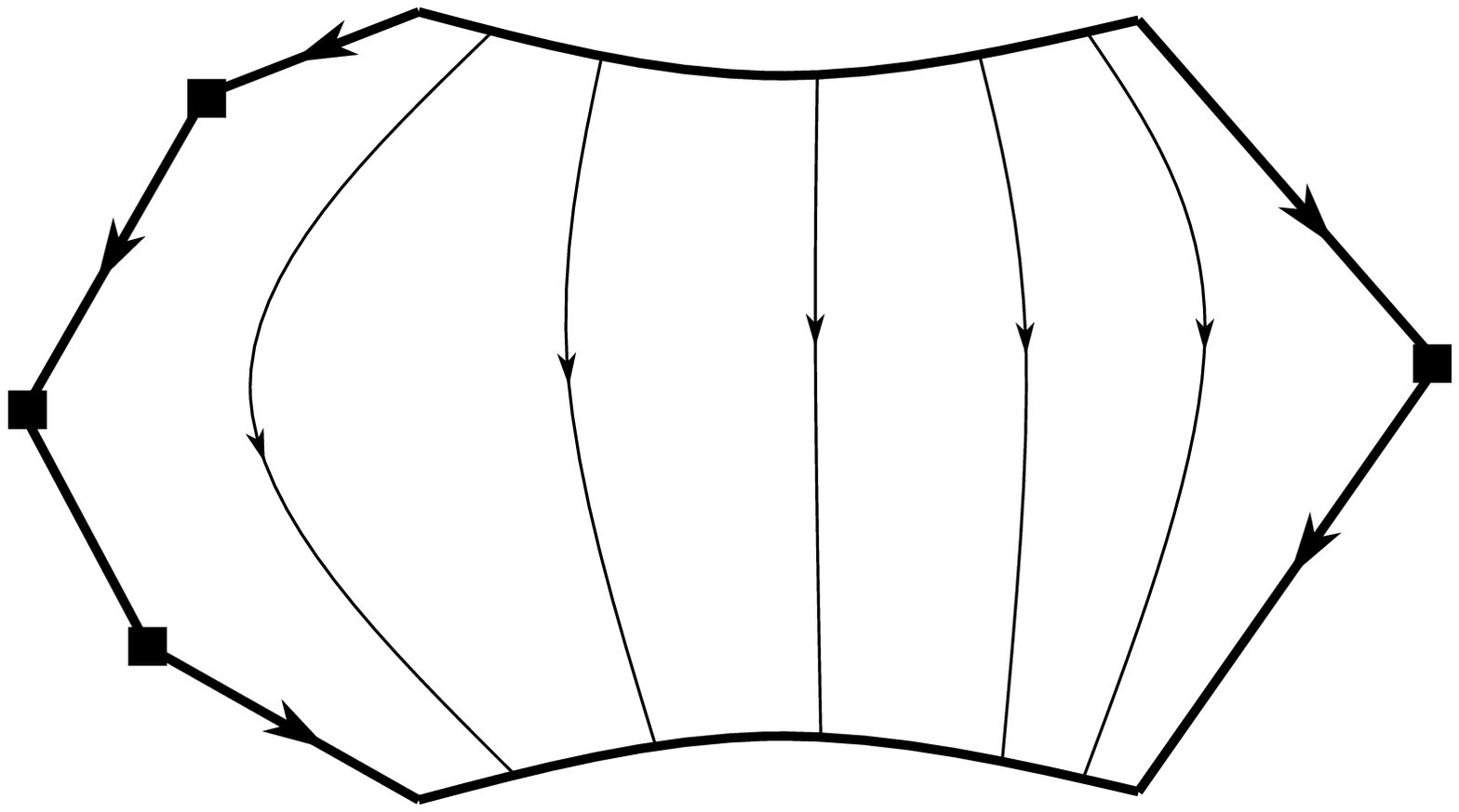} 
						\put(50,53){$\alpha$}
						\put(50,-3){$\omega$}
				\end{overpic}
			
				Type $12$.
			\end{center}
		\end{minipage}
		\begin{minipage}{4cm}
			\begin{center}
				\begin{overpic}[height=2cm]{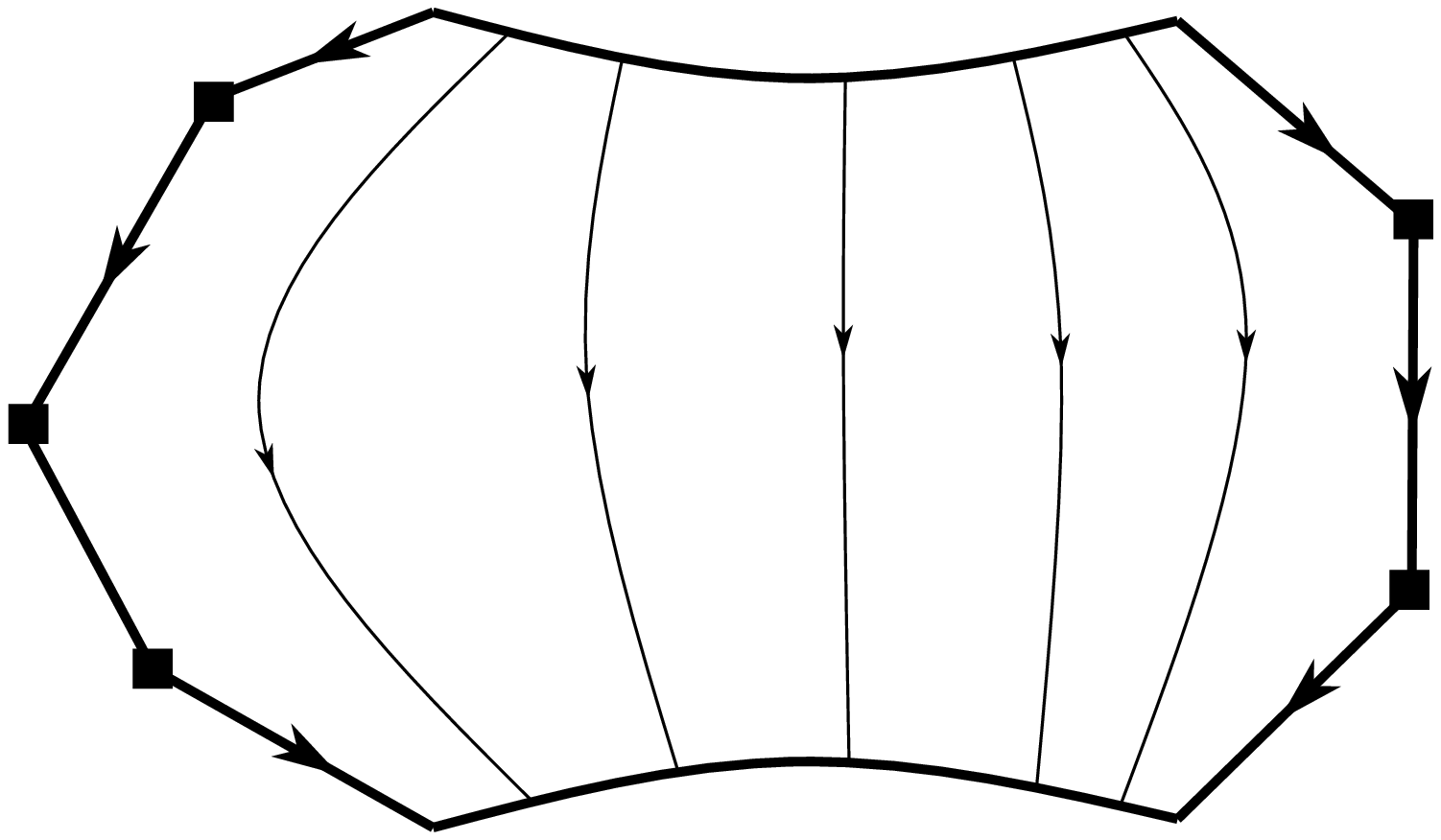} 
						\put(50,55){$\alpha$}
						\put(50,-3){$\omega$}
				\end{overpic}
			
				Type $13$.
			\end{center}
		\end{minipage}
	\end{center}
\caption{Some examples of generic regions. The square dots denotes hyperbolic saddles. Observe that the $\alpha$ and $\omega$-sets could be interchanged.}\label{Fig6}
	\begin{center}
		\begin{minipage}{5cm}
			\begin{center}
				\begin{overpic}[height=3cm]{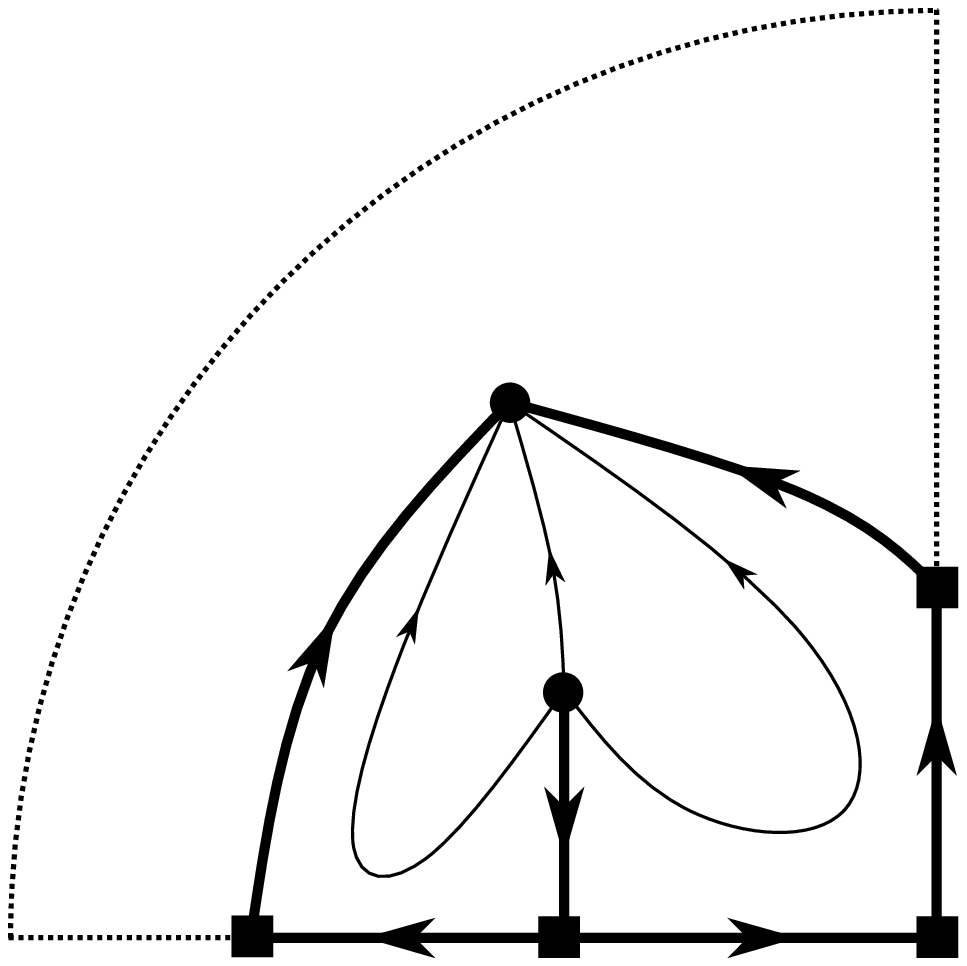} 
				\end{overpic}
				
				Type $5$.
			\end{center}
		\end{minipage}
		\begin{minipage}{5cm}
			\begin{center}
				\begin{overpic}[height=3cm]{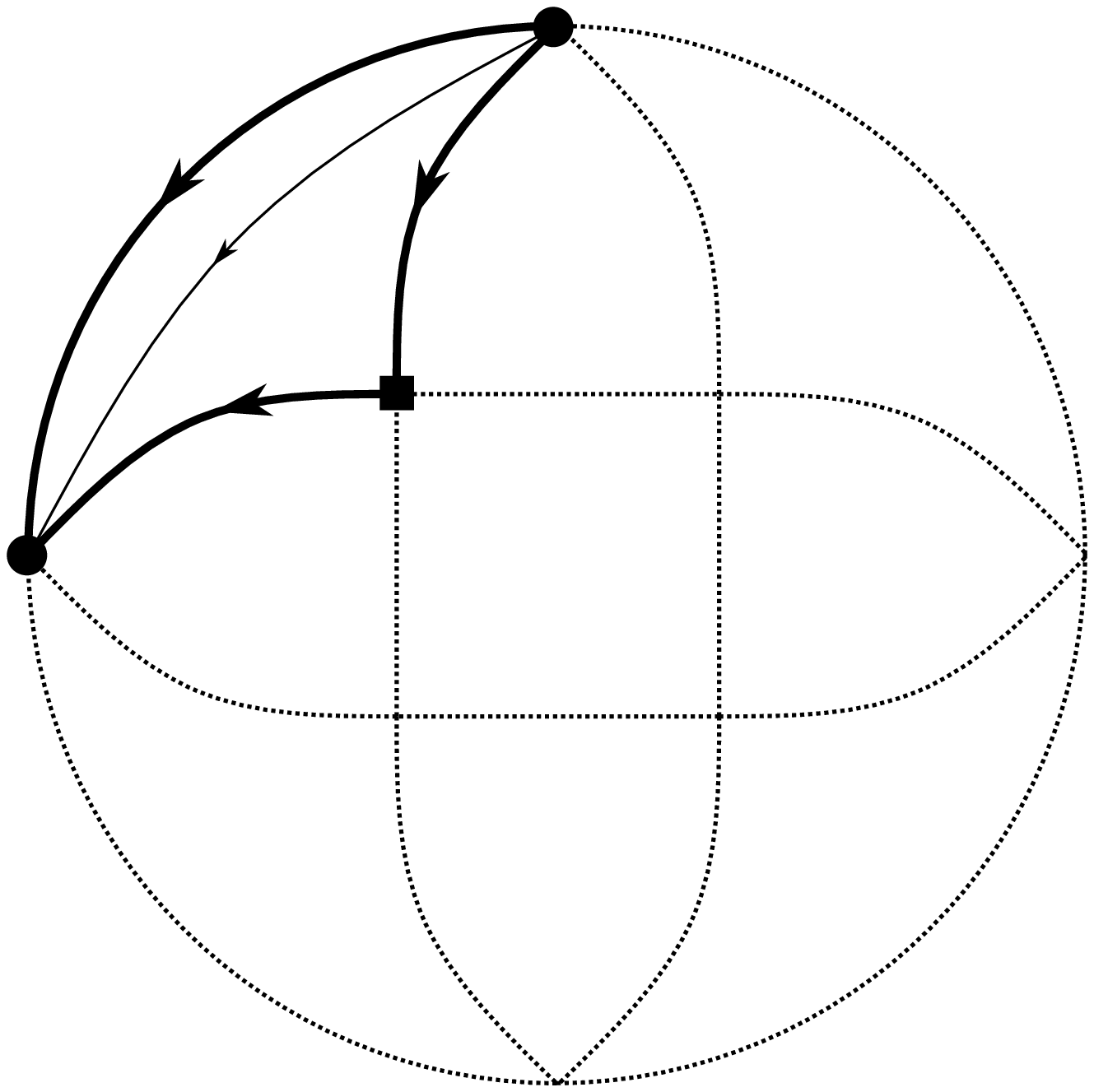} 
				\end{overpic}
				
				Type $6$.
			\end{center}
		\end{minipage}
		\begin{minipage}{5cm}
			\begin{center}
				\begin{overpic}[height=3cm]{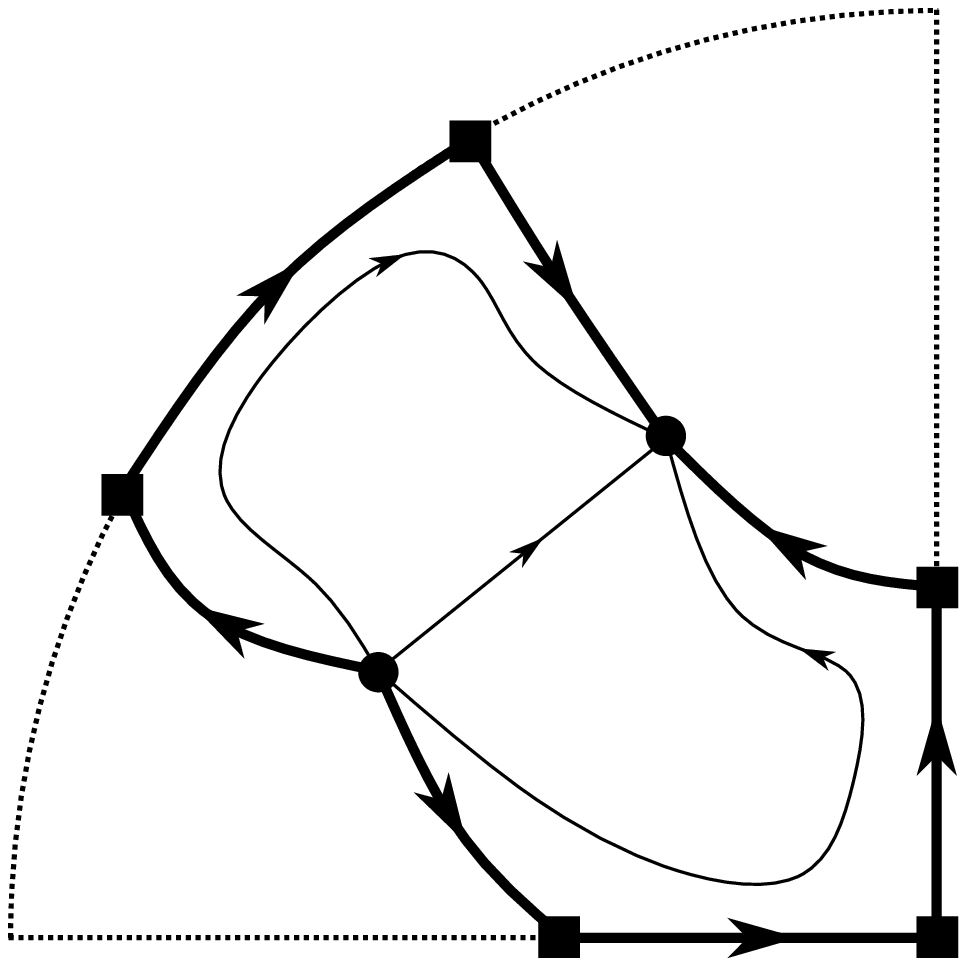} 
				\end{overpic}
				
				Type $13$.
			\end{center}
		\end{minipage}
	\end{center}
\caption{Some examples on the realization of generic regions.}\label{Fig7}
\end{figure}
For each $i\in\{1,\dots,k\}$ and $j\in\{1,\dots,l\}$, we construct homeomorphisms $f_i\colon\overline{E_i}\to\overline{F_i}$ (where $\overline{M}$ denotes the topological closure of $M$) and $g_j\colon\overline{B_j}\to\overline{D_j}$, sending orbits of $p(X)$ to orbits of $p(Y)$, such that they agree whenever they overlaps. Hence, these homeomorphisms give rises to a global homeomorphism $H\colon\mathbb{S}^2\to\mathbb{S}^2$, sending orbits of $p(X)$ to orbits of $p(Y)$. There are many possibilities for the generic regions, depending for example on how many hyperbolic saddles it has on its boundary. See Figures~\ref{Fig6} and \ref{Fig7}. Before we construct the homeomorphism, we state some results. We recall that $\mathbb{S}^2$ can be endowed with some \emph{Riemannian metric} (see \cite[Chapter $13$] {LeeManifolds}) and thus given two points $p$, $q\in\mathbb{S}^2$ and a curve $\gamma\subset\mathbb{S}^2$ from $p$ to $q$, not necessarily given by the flow of $p(X)$, the notion of \emph{length} of $\gamma$ (in relation to the Riemannian metric) is well defined. Such length will be denoted by $\ell(pq)$. Moreover, the curve $\gamma$ is also known as an \emph{arc} from $p$ to $q$.

\begin{proposition}[Lemma~$8$, \cite{PeiPei1959}]\label{Prop6}
	Let $p$ be a hyperbolic saddle to which the separatrices $\gamma_1$ and $\gamma_2$ tend when $t\to\infty$ and $t\to-\infty$. Through $p_1\in\gamma_1$ and $p_2\in\gamma_2$, consider small transversal sections $\sigma_1$ and $\sigma_2$. Given $q_1\in\sigma_1$, let $q_2\in\sigma_2$ be the point which the orbit through $q_1$ intersects $\sigma_2$. If $q_1\to p_1$, then $\ell(q_1q_2)\to\ell(p_1p)+\ell(p_2p)$.	
\end{proposition}

\begin{remark}
	On the context of Proposition~\ref{Prop6} we stress that by $q_1\to p_1$ we mean that $q_1$ is getting closer to $p_1$ in relation to the transversal section $\gamma_1$. More precisely, if we identify $\sigma_1=\{s\in\mathbb{R}\colon 0\leqslant s<\delta\}$ with $p_1$ given by $s=0$ and $q_1$ given by $s=s_1$ for some $s_1>0$, then by $q_1\to p_1$ we mean $s_1\to 0$.
\end{remark}

\begin{proposition}[Lemma~$10$, \cite{PeiPei1959}]\label{Prop7}
	Let $a_0b_0$ be an arc and $a_ib_i$, $i\geqslant1$, be a sequence of arcs converging uniformly to $a_0b_0$ in such a way that $\ell(a_ib_i)\to\ell(a_0b_0)$. Then the following statements holds.
	\begin{enumerate}[label=(\roman*)]
		\item Given points $p_i$ of $a_ib_i$ and $p_0$ of $a_0b_0$, let
			\[z_i=\frac{\ell(a_ip_i)}{\ell(a_ib_i)}, \quad z_0=\frac{\ell(a_0p_0)}{\ell(a_0b_0)}.\]
		Then $p_i\to p_0$ if, and only if, $z_i\to z_0$;
		\item The points on $a_ib_i$ converge uniformly to the points on $a_0b_0$, with equal $z$-coordinate.
	\end{enumerate}
\end{proposition}

As an example of the techniques of M. C. Peixoto and M. M. Peixoto \cite{PeiPei1959} we now present the construction the homeomorphism for a generic region $R$ of $X$, of type $7$ (see Figure~\ref{Fig6}), known as type $V$ in \cite{PeiPei1959}. Let $a_0a_1$ and $b_0b_1$ be the source $\alpha$ and the sink $\omega$ of $R$. Let also $p_0$ and $p_1$ be the hyperbolic saddles associated to $R$. Let $W\subset\mathbb{R}^2$ be the rectangle given by,
	\[W=\{(u,v)\in\mathbb{R}^2\colon 0\leqslant u\leqslant 1, \; 0\leqslant v \leqslant 2\}.\]
We now construct an homeomorphism $\varphi\colon\overline{R}\to W$ as follows (see Figure~\ref{Fig8}).
\begin{figure}[h]
	\begin{center}
		\begin{overpic}[width=10cm]{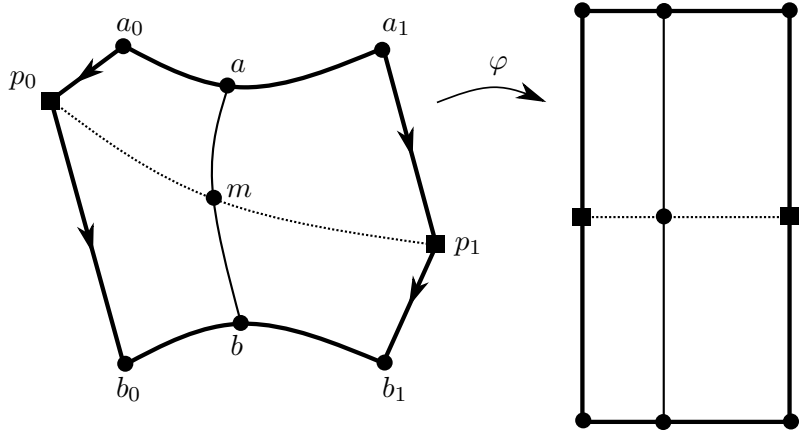} 
			\put(10,53){$a_0$}
			\put(45,53){$a_1$}
			\put(-4,46){$p_0$}
			\put(54.5,24){$p_1$}
			\put(10,4.5){$b_0$}
			\put(45,4.5){$b_1$}
			\put(24.5,31){$m$}
			\put(59,48){$\varphi$}
			\put(25,47.5){$a$}
			\put(25,10){$b$}
		\end{overpic}
	\end{center}
	\caption{An illustration of the homeomorphism of a generic region of type $7$.}\label{Fig8}
\end{figure}
The arc $a_0a_1$ is mapped onto the segment $0\leqslant u\leqslant 1$, $v=0$ by arc length. That is, given $a\in a_0a_1$, let $\varphi(a)=(u,0)$, where
\begin{equation}\label{35}
	u=\frac{\ell(a_0a)}{\ell(a_0a_1)}.
\end{equation}
Similarly, the arc $a_ip_i$ (resp. $p_ib_i)$ is mapped by arc length onto the segment $u=i$, $0\leqslant v\leqslant 1$ (resp. $u=i$, $1\leqslant v\leqslant 2$), $i\in\{0,1\}$. We now extend $\varphi$ to the interior of $\overline{R}$. Let
	\[\mu_0=\frac{\ell(a_0p_0)}{\ell(a_0b_0)}, \quad \mu_1=\frac{\ell(a_1p_1)}{\ell(a_1b_1)}.\]
Given $a\in a_0a_1$, with $\varphi(a)=(u,0)$, let $b\in b_0b_i$ be the point which the orbit through $a$ intersects $b_0b_1$. Let also $m\in ab$ be the unique point such that
\begin{equation}\label{27}
	\frac{\ell(am)}{\ell(ab)}=(1-u)\mu_0+u\mu_1,
\end{equation}
and define $\varphi(m)=(u,1)$. Given $r\in ab$ we now define $\varphi(r)=(u,v)$, where
	\[v=\left\{\begin{array}{ll}
			\displaystyle \frac{\ell(ar)}{\ell(am)}, & \text{ if } v\in am, \vspace{0.2cm} \\
			\displaystyle 1+\frac{\ell(ar)}{\ell(mb)}, & \text{ if } v\in mb.
		\end{array}\right.\]
Observe that $\varphi$ is a bijection. The continuity of $\varphi$ in all points except those in the curves $a_ip_ib_i$, $i\in\{0,1\}$ follows from the continuity of solutions with respect to the initial conditions (see Theorem~$4$, p. 92 of \cite{Perko2001}), together with \eqref{27} and Proposition~\ref{Prop7}. To get continuity in the points of $a_ip_ib_i$, $i\in\{0,1\}$, we have to use also Proposition~\ref{Prop6}. Similarly, $\varphi^{-1}$ is also continuous and thus $\varphi$ is indeed an homeomorphism. Now, in the exactly same way, we construct an homeomorphism $\varphi_Y\colon\overline{R_Y}\to W$ and thus $h=\varphi_Y^{-1}\circ\varphi\colon\overline{R}\to\overline{R_Y}$ is the desired homeomorphism. If $R$ is a region of type $2$ (known as type $IV$ in \cite{PeiPei1959}), then we observe that $R$ can be thought as a limit case of a region of type $7$, when the arcs $a_0p_0$ and $a_1p_1$ coincide. See Figure~\ref{Fig9}.
\begin{figure}[h]
	\begin{center}
		\begin{overpic}[width=10cm]{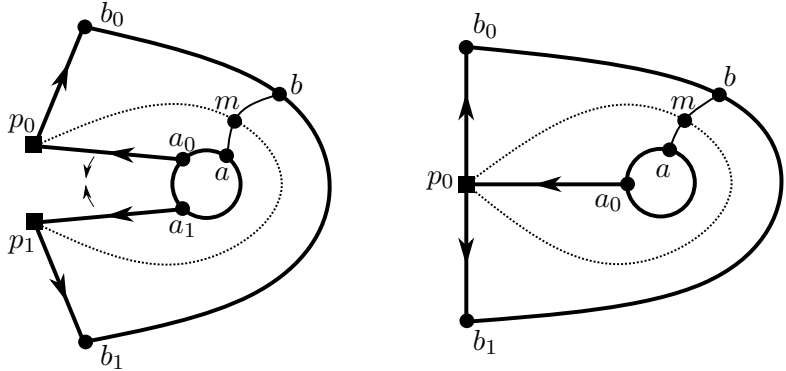} 
			\put(19,27){$a_0$}
			\put(19,15.5){$a_1$}
			\put(25,22.5){$a$}
			\put(-2,29.5){$p_0$}
			\put(-2,13.5){$p_1$}
			\put(10,43.5){$b_0$}
			\put(10,-2){$b_1$}
			\put(35,34){$b$}
			\put(25,32){$m$}
			
			\put(53,22){$p_0$}
			\put(75,19){$a_0$}
			\put(59,42){$b_0$}
			\put(59,0){$b_1$}
			\put(83,23){$a$}
			\put(92,35){$b$}
			\put(85,32){$m$}
		\end{overpic}
	\end{center}
	\caption{An illustration of the homeomorphism of a generic region of type $2$.}\label{Fig9}
\end{figure}
If $R$ is of type $6$ (known as type $III$ in \cite{PeiPei1959}), we need only to collapse the segment $a_1p_1$ at the point $a_1$. See Figure~\ref{Fig12}.
\begin{figure}[h]
	\begin{center}
		\begin{overpic}[width=5cm]{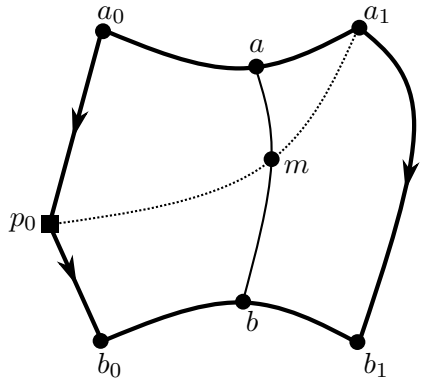} 
			\put(15,88){$a_0$}
			\put(85,88){$a_1$}
			\put(15,-6){$b_0$}
			\put(85,-6){$b_1$}
			\put(55,79){$a$}
			\put(64,47){$m$}
			\put(54,5){$b$}
			\put(-8,33){$p_0$}
		\end{overpic}
	\end{center}
	\caption{An illustration of the homeomorphism of a generic region of type $6$.}\label{Fig12}
\end{figure}
Let $R$ be a region of type $1$ (known as type $I$ in \cite{PeiPei1959}). In this case, $R$ is bounded by two closed curves $\alpha$ and $\omega$. We take an homeomorphism of $\alpha$ onto the circle, i.e. the interval $0\leqslant u\leqslant 1$, with identified end-points. Then, to each $m\in\overline{R}$, with $ab$ being the orbit through $m$, we associate the coordinates $(u,v)$, where $u$ is the abscissa of $a$ and
	\[v=\frac{\ell(am)}{\ell(ab)}.\]
This defines an homeomorphism from $\overline{R}$ to the cylinder $\mathbb{S}^1\times[0,1]$. To approach a region $R$ of type $9$ (known as type $VII$ in \cite{Soto1974}, p. $30$), we now present the technique developed by Sotomayor \cite{Soto1974}. Let $a_0a_1$ and $b_0b_1$ be the source $\alpha$ and the sink $\omega$ of $R$. Let also $p_0$, $p_1$ and $q$ be the hyperbolic saddles associated to $R$. See Figure~\ref{Fig10}.
\begin{figure}[h]
	\begin{center}
		\begin{overpic}[width=5cm]{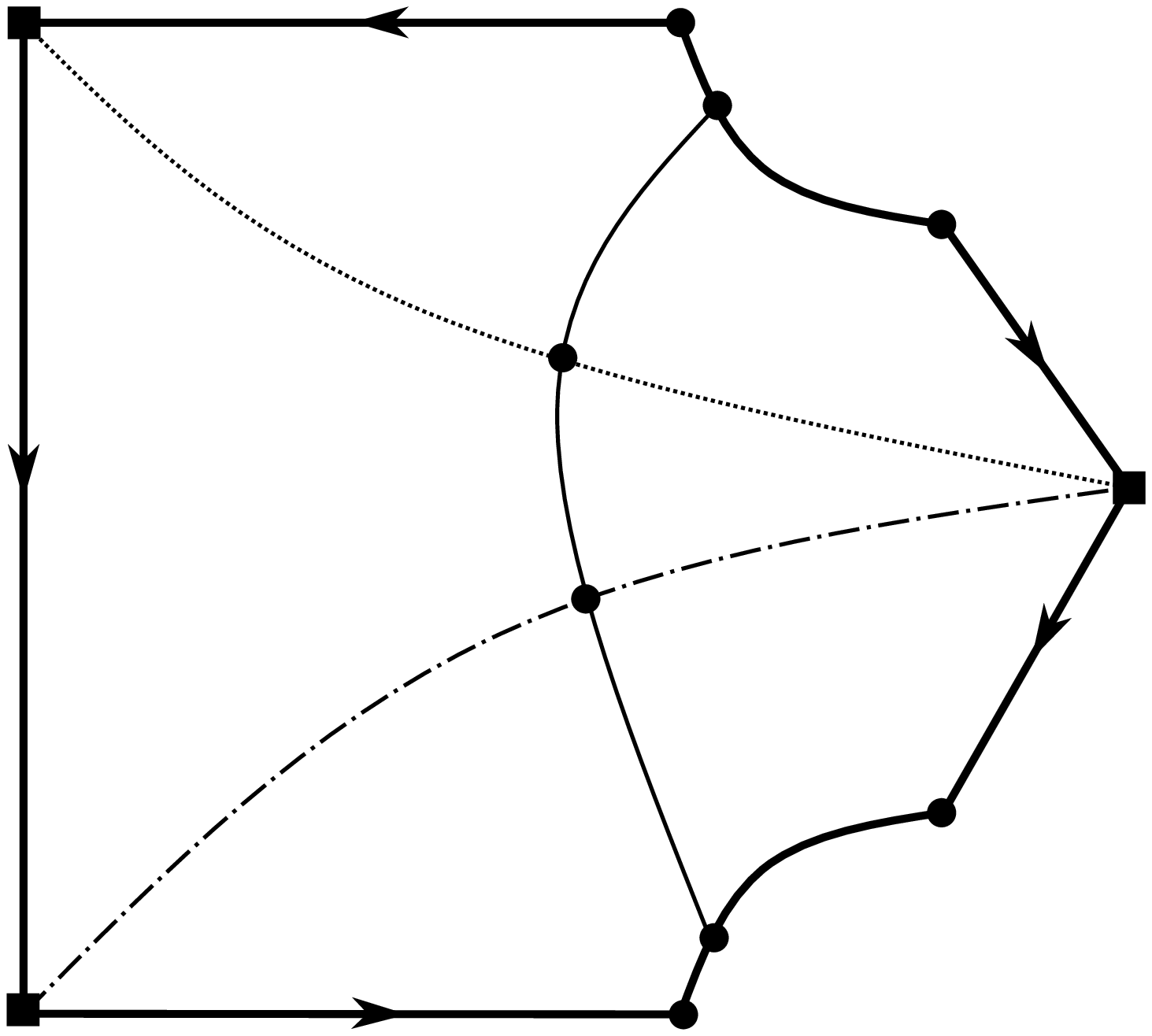} 
			\put(60,90){$a_0$}
			\put(65,80){$a$}
			\put(83,72){$a_1$}
			\put(60,-5){$b_0$}
			\put(65,5){$b$}
			\put(85,15){$b_1$}
			\put(102,46){$q$}
			\put(-5,92){$p_0$}
			\put(-5,-5){$p_1$}
			\put(52,60){$m_1$}
			\put(55,33){$m_2$}
		\end{overpic}
	\end{center}
	\caption{An illustration of the homeomorphism of a generic region of type $9$.}\label{Fig10}
\end{figure}
First let $a_0a_1$ be sent by arc length to the segment $0\leqslant u \leqslant 1$. Let also
	\[\alpha_1=\frac{\ell(a_0p_0)}{\ell(a_0p_0)+\ell(p_0p_1)+\ell(p_1b_0)}, \quad \beta_1=\frac{\ell(a_1q)}{\ell(a_1q)+\ell(qb_1)}, \quad \alpha_2=\frac{\ell(p_0p_1)}{\ell(p_0p_1)+\ell(p_1b_0)}, \quad \beta_2=1.\]
Given an orbit $ab$, with $a\mapsto u$, we divide it into three pieces $am_1$, $m_1m_2$, $m_2b$, such that
	\[\frac{\ell(am_1)}{\ell(ab)}=(1-u)\alpha_1+u\beta_1, \quad \frac{\ell(m_1m_2)}{\ell(m_1b)}=(1-u)\alpha_2+u\beta_2.\]
Similarly to the previous cases, we can now send the pieces $am_1$, $m_1m_2$, $m_2b$ by arc length to its respective points in $R_Y$ and thus create an homeomorphism $\varphi\colon\overline{R}\to\overline{R_Y}$. We now present how to create an homeomorphism on the other generic regions. Similarly to a region of type $7$ (see Figure~\ref{Fig8}), we can create an homeomorphism on the generic region of type $10$. See Figure~\ref{Fig11}.
\begin{figure}[h]
	\begin{center}
		\begin{overpic}[width=10cm]{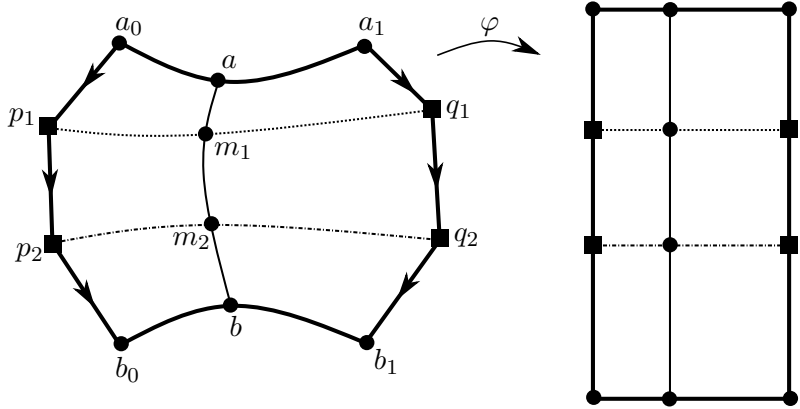} 
			\put(10,50){$a_0$}
			\put(24,45){$a$}
			\put(42,50){$a_1$}
			\put(-4,38){$p_1$}
			\put(-3,20){$p_2$}
			\put(53.5,39){$q_1$}
			\put(54.5,22){$q_2$}
			\put(10,4){$b_0$}
			\put(25,9.5){$b$}
			\put(44,5){$b_1$}
			\put(23,33.5){$m_1$}
			\put(17.5,21.5){$m_2$}
			\put(58,50){$\varphi$}
		\end{overpic}
	\end{center}
	\caption{An illustration of the homeomorphism of a generic region of type $10$.}\label{Fig11}
\end{figure}
Similarly to a region of type $9$ (see Figure~\ref{Fig10}), we can create an homeomorphism on the generic regions of type $8$ and $12$. See Figure~\ref{Fig13}.
\begin{figure}[h]
	\begin{center}
		\begin{minipage}{7cm}
			\begin{center}
				\begin{overpic}[height=4cm]{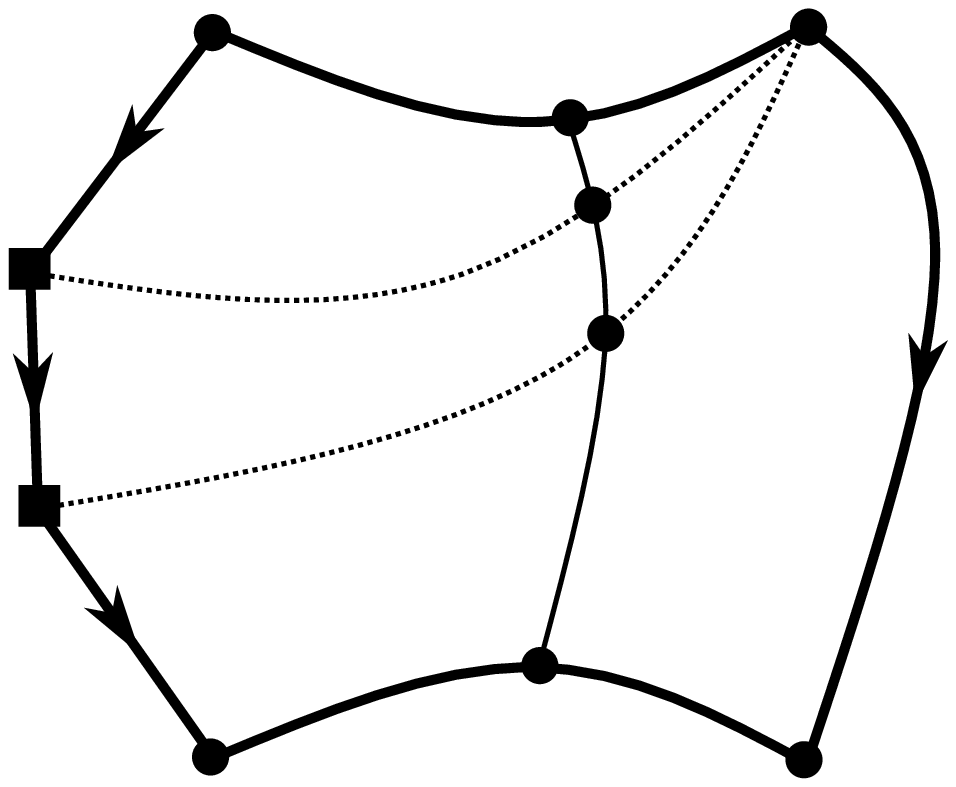} 
					\put(25,80){$a_0$}
					\put(58,73){$a$}
					\put(85,83){$a_1$}
					\put(-7,55){$p_1$}
					\put(-7,28){$p_2$}
					\put(50,62){$m_1$}
					\put(66,44){$m_2$}
					\put(20,-6){$b_0$}
					\put(55,4){$b$}
					\put(85,-6){$b_1$}
				\end{overpic}
				
				Type $8$
			\end{center}
		\end{minipage}
		\begin{minipage}{7cm}
			\begin{center}
				\begin{overpic}[height=4cm]{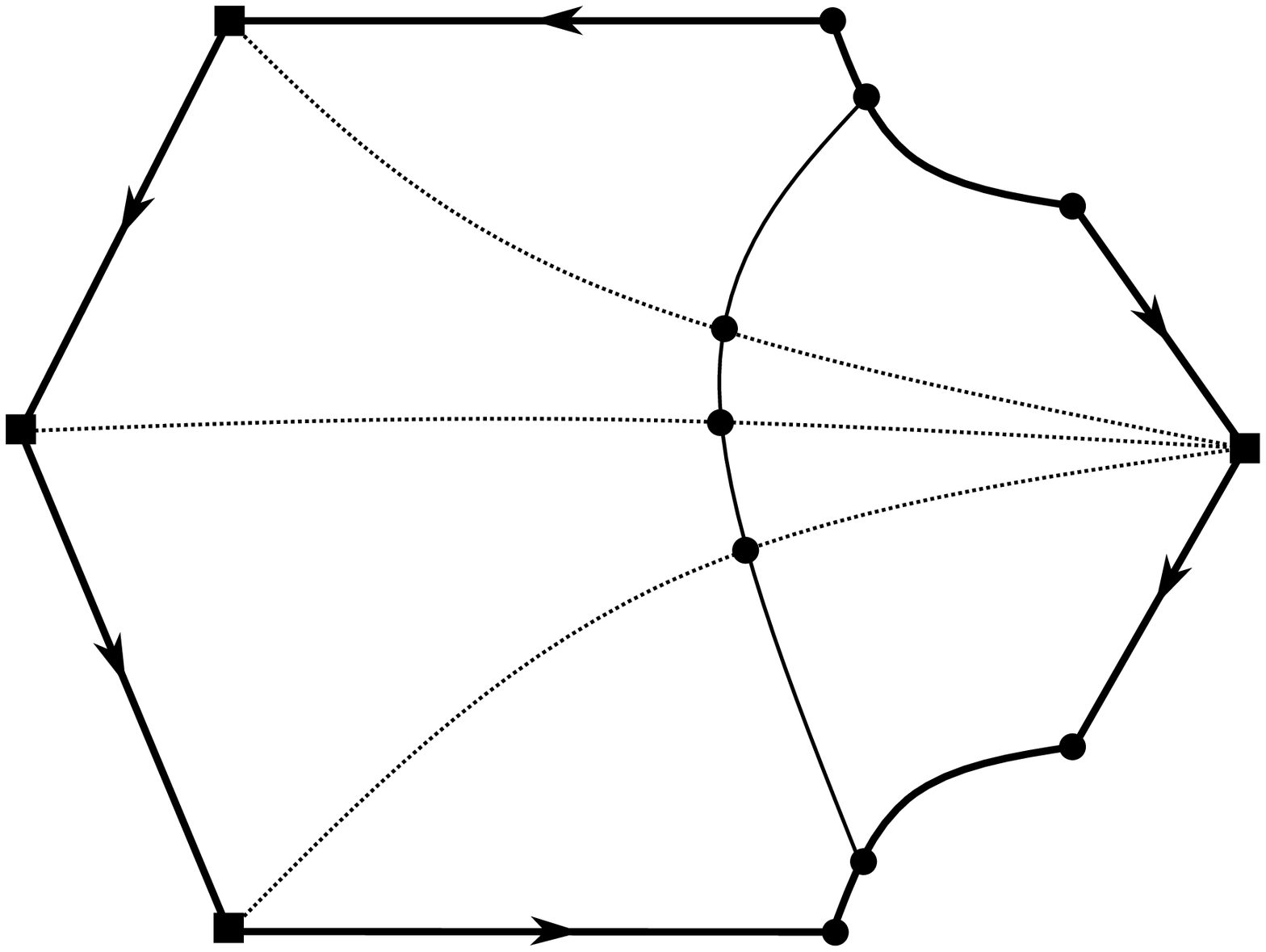} 
					\put(67,75){$a_0$}
					\put(71,67){$a$}
					\put(86,60){$a_1$}
					\put(10,75){$p_1$}
					\put(-8,40){$p_2$}
					\put(9,0){$p_3$}
					\put(101,38){$q_1$}
					\put(84,9){$b_0$}
					\put(70,4){$b$}
					\put(67,-5){$b_1$}
					\put(60,50){$m_1$}
					\put(47,37){$m_2$}
					\put(62,28){$m_3$}
				\end{overpic}
	
				Type $12$
			\end{center}
		\end{minipage}
	\end{center}
\caption{Illustrations of the homeomorphisms on generic regions of types $8$ and $12$.}\label{Fig13}
\end{figure}
To construct an homeomorphism on the other generic regions $R$ (including the ones which may not be in Figure~\ref{Fig6}), we work as follows. Let $a_0a_1$ and $b_0b_1$ be the source and sink of $R$. Let also $p_1,p_2,\dots,p_k$ and $q_1,q_2,\dots,q_l$ be the hyperbolic saddle of $R$. See Figure~\ref{Fig14}.
\begin{figure}[h]
	\begin{center}
		\begin{overpic}[width=10cm]{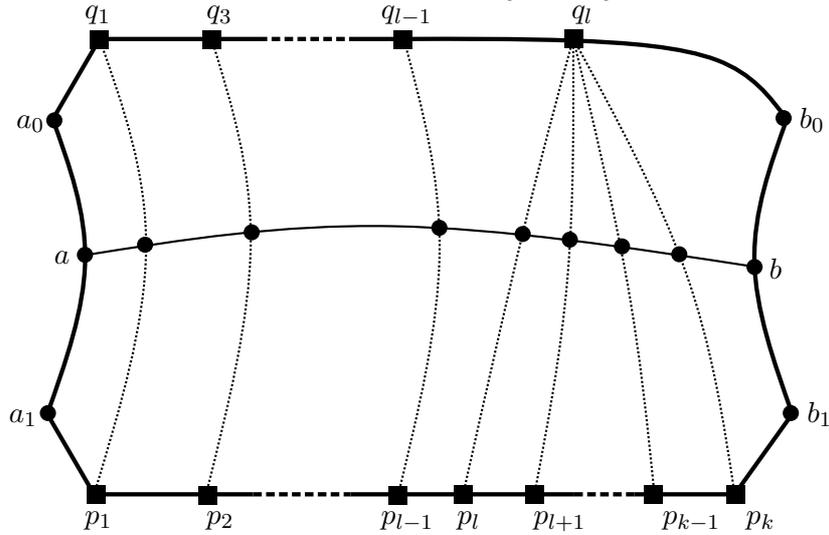} 
			\put(-3,50){$a_0$}
			\put(2,32){$a$}
			\put(-4,11){$a_1$}
			\put(100,50){$b_0$}
			\put(96,30){$b$}
			\put(101,11){$b_1$}
			
			\put(6,-2.5){$p_1$}
			\put(22,-2.5){$p_2$}
			\put(45,-2.5){$p_{l-1}$}
			\put(55,-2.5){$p_l$}
			\put(65,-2.5){$p_{l+1}$}
			\put(82,-2.5){$p_{k-1}$}
			\put(93,-2.5){$p_k$}
			
			\put(6,64.5){$q_1$}
			\put(22,64.5){$q_3$}
			\put(45,64.5){$q_{l-1}$}
			\put(70,64.5){$q_l$}
		\end{overpic}
	\end{center}
	\caption{An illustration of the homeomorphism of a generic generic region $R$.}\label{Fig14}
\end{figure}
Without loss of generality, suppose $k\geqslant l$. For each $i\in\{1,\dots,l\}$, we connect the saddles $p_i$ and $q_i$. Then, for each $j\in\{l,l+1,\dots,k\}$ we connect $p_j$ with $q_l$. If $l=0$, then we connect each $p_j$ with $a_0$, $j\in\{1,\dots,k\}$. Moreover, if necessary we can identify the arcs $a_0p_1$ and $a_0q_1$. Therefore, at this stage we have obtained an homeomorphism for each generic region, including those which may not be in Figure~\ref{Fig6}. Furthermore, since on the boundary of each generic region the homeomorphisms are constructed by \emph{arc length} (recall \eqref{35}, for example) it follows that they agree on the boundary of such regions. The homeomorphisms of the critical regions follows exactly as in \cite{PeiPei1959}. For the sake of self-containedness, we now present a briefly explanation of the construction of the homeomorphism in the case of a critical region of a hyperbolic limit cycle. Let $\gamma$ be a limit cycle of $X$ and $C$, $C'$ be the curves that bounds the critical region $S$ of $\gamma$, given by Lemma~\ref{Lemma10}. See Figure~\ref{Fig16}.
\begin{figure}[h]
	\begin{center}
		\begin{overpic}[width=10cm]{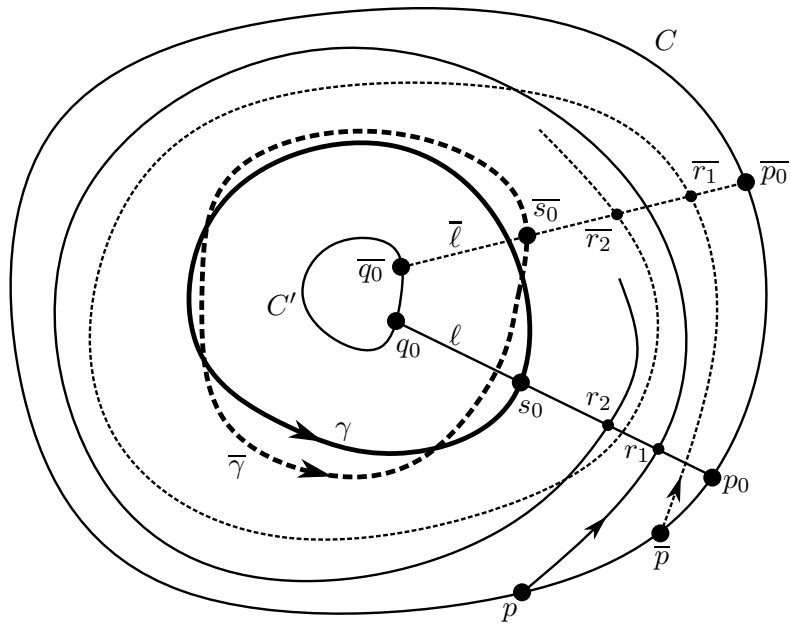} 
			\put(85,75){$C$}
			\put(34,40){$C'$}
			
			\put(94,17){$p_0$}
			\put(51,35){$q_0$}
			\put(81,21){$r_1$}
			\put(76,28){$r_2$}
			\put(65,0){$p$}
			\put(58,36){$\ell$}
			\put(67,27){$s_0$}
			\put(43,24){$\gamma$}
			
			\put(99,58){$\overline{p_0}$}
			\put(46,45){$\overline{q_0}$}
			\put(90,58){$\overline{r_1}$}
			\put(76,49){$\overline{r_2}$}
			\put(85,7){$\overline{p}$}
			\put(58,49){$\overline{\ell}$}
			\put(69,53){$\overline{s_0}$}
			\put(29,19){$\overline{\gamma}$}
		\end{overpic}
	\end{center}
	\caption{An illustration of the homeomorphism of the critical region of a limit cycle.}\label{Fig16}
\end{figure}
Without loss of generality, suppose that $\gamma$ is stable and has the counterclockwise orientation. Let $Y\in N\subset\mathcal{P}_d$ and denote by $\overline{\gamma}$ the associated limit cycle of $Y$. Let $\ell$ be a normal section of $\gamma$ and let $s_0$, $p_0$ and $q_0$ be given by the intersection of $\ell$ with $\gamma$, $C$ and $C'$, respectively. We recall that at this stage we already have an homeomorphism $\varphi$ defined on $C\cup C'$, given by the homeomorphisms on the generic regions, and thus we need to extend $\varphi$ to $S$. Hence, we can consider the points $\overline{p_0}=\varphi(p_0)$, $\overline{q_0}=\varphi(q_0)$ and thus the transversal section $\overline{\ell}$ passing through them. It follows from Lemma~\ref{Lemma10} that restricting $S$ if necessary, we can assume that the flow of $p(Y)$ is transversal to $\overline{\ell}$.	Let $S'$ (resp. $S''$) be the region bounded by $\gamma$ and $C$ (resp $C'$). We will focus on the extension of $\varphi$ to $S'$. The extension to $S''$ is similar. Given $p\in C$, let $\xi$ be the orbit of $X$ that enters $S'$ through $p$. Observe that $\xi$ intersects $\ell$ infinitely many times, defining a monotone sequence of points $(r_i)\subset\ell$ such that $r_i\to s_0$. Let $\overline{p}=\varphi(p)$ and let $\overline{\xi}$ be the orbit of $Y$ through $\overline{p}$. Observe that $\overline{\xi}$ intersects $\overline{\ell}$ infinitely many times, defining a monotone sequence $(\overline{r_i})\subset\overline{\ell}$ such that $\overline{r_i}\to\overline{s_0}$, for some $\overline{s_0}\in\overline{\gamma}$. Observe that the sequences $(r_i)$ and $(\overline{r_i})$, together with $p$ and $\overline{p}$, divides the orbits $\xi$ and $\overline{\xi}$ in sequences of arcs $\{\xi_i\}$ and $\{\overline{\xi_i}\}$. For each $i\in\mathbb{N}$ we send $\xi_i$ to $\overline{\xi_i}$ by arc length. This defines an extension of $\varphi$ to the open ring $S'$. Similarly, we can extend $\varphi$ to the open ring $S''$. Finally, we extend $\varphi$ to $\gamma$ by defining $\varphi(s_0)=\overline{s_0}$ and then by arc length on $\gamma$.

At this stage we have constructed the homeomorphisms on each generic and critical regions. Moreover, it follows from the construction that they agree wherever they overlaps and thus we have a global homeomorphism $\varphi\colon\mathbb{S}^2\to\mathbb{S}^2$, sending orbits of $p(X)$ to the orbits of $p(Y)$. We now briefly explain why $\varphi$ can be taken on a $\varepsilon$-neighborhood of the identity map. First, if $D$ is a critical region associated to a singularity (see Lemma~\ref{Lemma9}), then we can restrict $D$ such that for every $x$, $y\in D$ we have $||x-y||<\frac{1}{2}\varepsilon$. Hence, $\varphi|_D\colon D\to D$ is naturally in the $\varepsilon$-neighborhood of the identity. The case of the critical regions $S$ of limit cycles or polycycles follows similarly, one need only to restrict $S$ such that the normal section $\ell$ (recall Figures~\ref{Fig4} and \ref{Fig16}) has a small enough length. We now consider the case of a generic region $R$. For simplicity, assume that $R$ is of type $7$ (see Figure~\ref{Fig8}). For each $a\in a_0a_1$, it follows from the continuous dependence of initial conditions (see Theorem~$8$, p. 25 of \cite{And1971}) that there is a $\delta$-neighborhood of $X$, $\delta=\delta(a)$, such that if $\rho(X,Y)<\delta$ (recall \eqref{36}), then the arc of orbit $\varphi(a)\varphi(b)$, given by the orbit of $Y$ through $\varphi(a)$, is in a $\varepsilon$-neighborhood of the arc $ab$. For each $a\in a_0a_1$, let $\overline{\delta}(a)>0$ be the least upper bound with this property (such upper bound exist because for $\delta>0$ big enough, $Y$ will leave $\mathcal{P}_d$). Let $\overline{\delta}=\inf\bigl\{\overline{\delta}(a)\colon a\in a_0a_1\bigr\}\geqslant 0$. We claim that $\overline{\delta}>0$. Indeed, if we suppose by contradiction that $\overline{\delta}=0$, then it follows that there is a sequence of points $(\alpha_k)\subset a_0a_1$ such that $\delta_k=\overline{\delta}(\alpha_k)$ satisfies $\delta_k\to 0$. Since the arc $a_0a_1$ is compact, it follows that (passing to a subsequence if necessary) we can assume $\alpha_k\to\alpha$, for some $\alpha\in a_0a_1$. It follows again from the continuous dependence of initial conditions that there is an open arc $I\subset a_0a_1$ (in the relative topology of $a_0a_1$), with $\alpha\in I$, and $\delta_0>0$ such that if $\rho(X,Y)<\delta_0$, then for each $a\in I$ the arc $\varphi(a)\varphi(b)$ is in a $\varepsilon$-neighborhood of the arc $ab$. But this contradicts the fact that $\alpha_k\to\alpha$ and $\delta_k\to0$, proving the claim. Therefore, it follows from the definition of $\overline{\delta}$ that if $\rho(X,Y)<\overline{\delta}$, then $\varphi|_R$ is $\varepsilon$-close to the identity. Since there are a finite number of generic and critical regions, it follows that there is $\delta>0$ small enough such that if $\rho(X,Y)<\delta$, then $\varphi$ is $\varepsilon$-close to $\text{Id}_{\mathbb{S}^2}$.

\section*{Acknowledgments}

We thank to the reviewers their comments and suggestions which help us to improve the presentation of this paper. The authors are partially supported by CNPq, grant 304798/2019-3 and by S\~ao Paulo Research Foundation (FAPESP), grants 2019/10269-3 and 2021/01799-9.


\begin{thebibliography}{99}
	
\bibitem{AccMarOvi}
{\sc E. Accinelli et al}, 
{\it Who controls the controller? A dynamical model of corruption},
The Journal of Mathematical Sociology, \textbf{41} (2017), p. 220–247.

\bibitem{AndPon1937}
{\sc A. Andronov and L. Pontrjagin}, 
{\it Systemes Grossiers},
Doklady Akademii Nauk SSSR. \textbf{14} (1937), p 247–250.
	
\bibitem{And1971}
{\sc A. Andronov et al}, 
{\it Theory of bifurcations of dynamic systems on a plane},
Israel Program for Scientific Translations;[available from the US Department of Commerce, National Technical Information Service, Springfield, Va.], 1971.

\bibitem{BasBuzSan2022}
{\sc J. Bastos, C. Buzzi and P. Santana}, 
{\it Evolutionary Stable Strategies and Cubic Vector Fields},
to appear in Nonlinear Differential Equations and Applications (2024).

\bibitem{Bomze}
{\sc I. Bomze}, 
{\it Lotka-Volterra equation and replicator dynamics: a two-dimensional classification},
Biological cybernetics, \textbf{48} (1983), p. 201-211.

\bibitem{BueLop2018}
{\sc J. Buend\'{\i}a and V. L\'{o}pez}, 
{\it On the Markus-Neumann theorem},
J. Differential Equations, \textbf{265} (2018), p. 6036-6047.

\bibitem{Cai1979}
{\sc S. Cai}, 
{\it A note on: ``Classification of generic quadratic vector fields with no limit cycles''},
JZhejiang Daxue Xuebao, \textbf{4} (1979), p. 105-113.

\bibitem{Cherkas}
{\sc L. Cherkas}, 
{\it The stability of singular cycles},
Differentsial'nye Uravneniya, \textbf{4} (1968), pp. 1012-1017.

\bibitem{Bag1952}
{\sc F. DeBaggis}, 
{\it Dynamical systems with stable structures},
Contributions to the theory of nonlinear oscillations (1952), p. 37-59.
	
\bibitem{DumLliArt2006}
{\sc F. Dumortier, J. Llibre and J. C. Art\'es}, 
{\it Qualitative theory of planar differential systems},
Universitext, Springer-Verlag, Berlim, 2006.

\bibitem{DumSha1990}
{\sc F. Dumortier and D. Shafer}, 
{\it Restrictions on the equivalence homeomorphism in stability of polynomial vector fields},
Journal of the London Mathematical Society, \textbf{41} (1990), p. 100-108.

\bibitem{Ful}
{\sc W. Fulton}, 
{\it Algebraic Curves},
Mathematics Lecture Note Series, third edition (2008).

\bibitem{GasGiaTor2007}
{\sc A. Gasull, H. Giacomini and J. Torregrosa}, 
{\it Explicit non-algebraic limit cycles for polynomial systems},
Journal of Computational and Applied Mathematics, \textbf{200} (2007).

\bibitem{GelKapZel1994}
{\sc I. Gelfand, M. Kapranov and A. Zelevinsky}, 
{\it Discriminants, Resultants, and Multidimensional Determinants},
Birkhäuser Boston, MA, 1994.

\bibitem{GolGui2012}
{\sc M. Golubitsky and V. Guillemin}, 
{\it Stable mappings and their singularities},
Springer Science and Business Media, 2012.

\bibitem{Hines}
{\sc W. Hines}, 
{\it Evolutionary stable strategies: a review of basic theory},
Theoretical Population Biology, \textbf{31.2} (1987), p. 195-272.

\bibitem{JarLliSha2005}
{\sc X. Jarque, J. Llibre and D. Shafer}, 
{\it Structural stability of planar polynomial foliations},
Journal of Dynamics and Differential Equations, \textbf{17} (2005), p. 573-587.

\bibitem{Kooij}
{\sc R. Kooij}, 
{\it Cubic systems with four real line invariants},
Mathematical Proceedings of the Cambridge Philosophical Society, \textbf{118} (1995), p. 7-19.

\bibitem{Kot1982}
{\sc J. Kotus, M. Krych and Z. Nitecki}, 
{\it Global structural stability of flows on open surfaces},
Memoirs of the American Mathematical Society, \textbf{37} (1982).

\bibitem{LeeManifolds}
{\sc J. Lee}, 
{\it Introduction to Smooth Manifolds},
Graduate texts in mathematics $218$, second edition (2012).

\bibitem{Mar1954}
{\sc L. Markus}, 
{\it Global structure of ordinary differential equations in the	plane},
Trans. Amer. Math. Soc, \textbf{76} (1954),  p. 127-148.

\bibitem{Mar1961}
{\sc L. Markus}, 
{\it Structurally stable differential systems},
Annals of Mathematics, \textbf{73} (1961),  p. 1-19.

\bibitem{Neu1975}
{\sc D. Neumann}, 
{\it Classification of continuous flows on {$2$}-manifolds},
Proc. Amer. Math. Soc, \textbf{48} (1975), pp 73-81.

\bibitem{NickDye}
{\sc R. Nickalls and R. Dye}, 
{\it The Geometry of the Discriminant of a Polynomial},
The Mathematical Gazette JSTOR, \textbf{80} (1996), pp 279-285.

\bibitem{CanBer}
{\sc E. Ozkan-Canbolat, A. Beraha}, 
{\it Evolutionary stable strategies for business innovation and	knowledge transfer},
International Journal of Innovation Studies, \textbf{3} (2019), p. 55-70.

\bibitem{PeiPei1959}
{\sc M. C. Peixoto and M. M. Peixoto}, 
{\it Structural stability in the plane with enlarged boundary conditions},
An. Acad. Brasil. Ci, \textbf{31} (1959), p. 135-160.

\bibitem{Pei1959}
{\sc M. M. Peixoto}, 
{\it On Structural Stability},
Annals of Mathematics, \textbf{69} (1959), p. 199-222.

\bibitem{Pei1962}
{\sc M. M. Peixoto}, 
{\it Structural stability on two-dimensional manifolds},
Topology, \textbf{1} (1962), p. 101-120.

\bibitem{Pei1973}
{\sc M. M. Peixoto}, 
{\it On the classification of flows on 2-manifolds},
Dynamical systems, 1973, pp. 389-419.

\bibitem{Perko1987}
{\sc L. Perko}, 
{\it On the Accumulation of Limit Cycles},
Proceedings of the American Mathematical Society, \textbf{99} (1987), p. 515-526.

\bibitem{Perko1992}
{\sc L. Perko}, 
{\it Bifurcation of Limit Cycles: Geometric Theorys},
Proceedings of the American Mathematical Society, \textbf{114} (1992), p. 225-236.

\bibitem{Perko1994}
{\sc L. Perko}, 
{\it Homoclinic loop and multiple limit cycle bifurcation surfaces},
Trans. Amer. Math. Soc, \textbf{344} (1994), p. 101-130.

\bibitem{Perko2001}
{\sc L. Perko}, 
{\it Differential equations and dynamical systems},
vol. 7 of Texts in Applied Mathematics, Springer-Verlag, New York, third ed, 2001.

\bibitem{SchSig}
{\sc P. Schuster and K. Sigmund}, 
{\it Coyness, philandering and stable strategies},
Animal Behaviour, \textbf{29} (1981), p. 186-192.

\bibitem{San1977}
{\sc T. Santos}, 
{\it Classification of generic quadratic vector fields with no limit cycles},
Geometry and topology (1977), p. 605-640.

\bibitem{Sha1987}
{\sc D. Shafer}, 
{\it Structural stability and generic properties of planar polynomial vector fields},
Rev. Mat. Iberoamericana, \textbf{3} (1987), p. 337-355.

\bibitem{Sha1990}
{\sc D. Shafer}, 
{\it Structure and stability of gradient polynomial vector fields},
Journal of the London Mathematical Society, \textbf{1} (1990), p. 109-121.

\bibitem{SmiPri1973}
{\sc J. Smith and G. Price}, 
{\it The logic of animal conflict},
Nature, \textbf{246} (1973), p. 15-18.

\bibitem{Soto1974}
{\sc J. Sotomayor}, 
{\it Generic one-parameter families of vector fields on two-dimensional manifolds},
Publications Mathématiques de l'Institut des Hautes Études Scientifiques, \textbf{43} (1974), p. 5-46.

\bibitem{Soto1985}
{\sc J. Sotomayor}, 
{\it Stable planar polynomial vector fields},
Revista Matemática Iberoamericana, \textbf{1:2} (1985), p. 15-23.

\bibitem{SotGarMel2020}
{\sc J. Sotomayor, R. Gargia and L. Mello}, 
{\it Mauricio Matos Peixoto},
Revista Matemática Universitária, \textbf{1} (2020), p. 1-22.

\bibitem{TayJon1978}
{\sc P. Taylor and L. Jonker}, 
{\it Evolutionary stable strategies and game dynamics},
Mathematical Biosciences, \textbf{40} (1978), p. 145-156.

\bibitem{Tei1977}
{\sc M. Teixeira}, 
{\it Generic bifurcations in manifolds with boundary},
Journal of Differential Equations, \textbf{25} (1977), p. 65-89.

\bibitem{Trot}
{\sc D. Trotman}, 
{\it Stratification theory},
Handbook of geometry and topology of singularities I. Springer, Cham, 2020. 243-273.

\bibitem{Vel}
{\sc E. Velasco}, 
{\it Generic Properties of Polynomial Vector Fields at Infinity},
Transactions of the American Mathematical Society, \textbf{143} (1969), p. 201-222.

\bibitem{XiaoYu}
{\sc T. Xiao and G. Yu}, 
{\it Supply chain disruption management and evolutionarily stable strategies of retailers in the quantity-setting duopoly situation with homogeneous goods},
European Journal of Operational Research, \textbf{173} (2006), p. 648-668.

\end{thebibliography}
\end{document}